%% file: main.tex
\documentclass[12pt]{article}

\usepackage{graphicx} 
\usepackage{amsmath}
\usepackage{hyperref}
\usepackage{cleveref}
\usepackage[toc,title,page]{appendix}
\usepackage[margin=3cm]{geometry}
\usepackage{amsthm}
\usepackage{amssymb}
\usepackage{enumitem}
\setlist[enumerate]{nosep} 

\usepackage[backend=biber,
            style=numeric-comp,
            sortcites=true,
            giveninits=true, 
            useprefix=true,
            sorting=nyt,
            maxbibnames=9]{biblatex}
\DeclareFieldFormat{titlecase}{#1}
\crefname{equation}{equation}{equations}
\crefname{figure}{Figure}{Figures}

\counterwithin{equation}{section}

\newcommand{\xx}{x}

\newcommand{\dd}{\mathrm{d}}

\newcommand{\ff}{f}

\input{macros}

\newtheorem{prop}{Proposition}
\newtheorem{thm}{Theorem}
\newtheorem{definition}{Definition}
\newtheorem{lem}{Lemma}
\newtheorem{cor}{Corollary}
\numberwithin{prop}{section}
\numberwithin{lem}{section}
\numberwithin{thm}{section}
\numberwithin{definition}{section}
\numberwithin{cor}{section}

\theoremstyle{remark}
\newtheorem*{remark}{Remark}

\title{Existence and dimensional lower bound for the global attractor of a PDE model for ant trail formation}
\author{Matthias Rakotomalala\thanks{CMAP, CNRS, École polytechnique, Institut Polytechnique de Paris, 91120 Palaiseau, France}, Oscar de Wit\thanks{Department of Applied Mathematics and Theoretical Physics, University of Cambridge, Cambridge CB3 0WA, UK}}

\begin{document}

\maketitle
\begin{abstract}
We study the asymptotic behavior of a nonlinear PDE model for ant trail formation, which was introduced in~\cite{bertucci2024curvature}. We establish the existence of a compact global attractor and prove the nonlinear instability of the homogeneous steady state under an inviscid instability condition. We also provide a dimensional lower bound on the attractor. Alternatively, we prove that if the interaction parameter is sufficiently small, the homogeneous steady state is globally asymptotically stable.
\end{abstract}

\tableofcontents

\section{Introduction}
\input{introduction}

\section{Main result}
\label{sec:mainresults}
\input{MainResults}

\section{Semigroup theory and regularity results}
\label{sec:SemigroupTheory}
\input{SemigroupTheory}

\section{Existence of the Global Attractor}
\label{sec:ExistenceGlobalAttractor}
\input{GlobalAttractor}

\section{Linear analysis around the homogeneous solution}
\label{sec:LinearAnalysisHomogeneous}
\input{LinearisedOperator}

\section{Stability and Instability}
\label{sec:StabilityInstability}
\input{NonLinearInstability}
\input{NonLinearStability}

\newpage

\input{Appendix}

\section*{References}
\addcontentsline{toc}{section}{References}
\printbibliography[heading=none]
\end{document}

%% file: macros.tex
\DeclareMathOperator*{\esssup}{ess\,sup}
\newcommand{\defeq}{\mathrel{\mathop:}=}
\newcommand{\eqdef}{\mathrel{\mathop=}:}

\newcommand{\f}{f}
\newcommand{\nabalaxtheta}{\nabla_\xi}
\newcommand{\LinOpEllip}{\mathcal{L}^{E}}
\newcommand{\LinOpPara}{\mathcal{L}^{P}}

\newcommand{\LinOpThet}{\mathsf{L}_\sigma^{E}}
\newcommand{\ABLinVec}{\mathbf{A}}

\newcommand{\VOpLinVec}{\mathsf{V}_{\llambda}}

\newcommand{\AmatLinVec}{\mathsf{B}_{\ttau}}
\newcommand{\Id}{\text{Id}}

\newcommand{\LinOpThetPara}{\mathsf{L}_{\ssigmax}^{\mathsf{P}}}
\newcommand{\ABLinVecPara}{\mathbf{A}}
\newcommand{\VOpLinVecPara}{\mathsf{V}^\mathsf{P}_{\llambda}}
\newcommand{\AmatLinVecPara}{\mathsf{B}_{\ttau, \cchi, \vvarphi}^\mathsf{P}}

\newcommand{\SmatLinVecPara}{\mathsf{S}^\mathsf{P}_{\ssigmax}}

\newcommand{\Real}{\mathrm{Re}}
\newcommand{\Imag}{\mathrm{Im}}

\newcommand{\llambda}{\Breve{\lambda}}
\newcommand{\cchi}{\Breve{\chi}}

\newcommand{\vvarphi}{\Breve{\nu}}
\newcommand{\ssigmax}{\Breve{\sigma}_x}
\newcommand{\ttau}{\Breve{\tau}}

\newcommand{\muinvi}{\mu_{0}}

\newcommand{\thth}{\theta\theta}

\newcommand{\dR}{\mathbb{R}}
\newcommand{\dC}{\mathbb{C}}

\newcommand{\Torus}{\mathbb{T}}
\newcommand{\xonepik}{z}

%% file: introduction.tex
Collectives of ants display highly complex forms of behavior and sustain observable macroscopic non-trivial patterns such as ant trails. See~\cite{holldobler1990ants} for a general overview of ant biology. 



In this paper, we study the qualitative properties of the following nonlinear PDE model for ant trail formation that was introduced in \cite{bertucci2024curvature} as an extension of the physics-based model from~\cite{pohl2014dynamic},
\begin{equation}\label{sys:F}
\tag{$\mathcal{F}$}
    \begin{cases}
        \partial_t f & =\nabla_\xx\cdot(\sigma_x\nabla_\xx f-\lambda v f)+\partial_\theta(\sigma_\theta\partial_\theta f-\chi B_\tau[c] f),\\
        \mathsf{t}\partial_t c & = \sigma_c\Delta_\xx c-\gamma c+\int f \dd\theta.
    \end{cases}
\end{equation}
Here, the quantity $f(t,\xx,\theta),(t,x,\theta)\in[0,+\infty)\times\mathbb{T}^2_1\times\mathbb{T}_{2\pi}$, describes the phase-space density of an ensemble of interacting particles, given a position $\xx$ on the 2-dimensional torus and orientation $\theta$ at time $t$. For the velocity $v$, we use the notation $v(\theta)=
(\cos\theta,\sin\theta)^{\mathsf{T}}$. The equation for $f$ is coupled to a parabolic or elliptic equation for the chemical field $c(t,x)$, via the interaction mechanism described by $B_\tau[c]$. The term $B_\tau$ that we study is a curvature look-ahead mechanism and writes as
\begin{equation}\label{eq:BtauIntro}
\tag{$B_\tau$}
B_\tau[c]=v(\theta)^\perp\cdot\nabla_\xx c+\tau v(\theta)^\perp\cdot \nabla^2_\xx cv(\theta),
\end{equation} where $v^\perp(\theta)=(-\sin\theta,\cos\theta)^{\mathsf{T}}$. Furthermore, $\sigma_\xx,\lambda,\sigma_\theta,\chi,\sigma_c$ and $\gamma$ are positive constants and $\mathsf{t}\in\{0,1\}$, corresponding to either the elliptic coupling or the parabolic coupling. 

The PDE model consists of the following components: translational diffusion modulated by the constant $\sigma_\xx$, self-propulsion modulated by the speed $\lambda$, rotational diffusion modulated by the constant $\sigma_\theta$, the interaction term $B_\tau$ modulated by the constant $\chi$ and the equation for $c$ with diffusion modulated by the constant $\sigma_c$, decay modulated by the constant $\gamma$ and a source term coming from the spatial density of $f$, $\rho(t,x)\defeq\int f(t,x,\theta)\dd\theta$.

The model can be derived as a formal mean-field limit for a stochastic interacting particle model as explained in Section 2.1 of \cite{bertucci2024curvature}. 
The term $B_\tau$ can also be derived as the first-order Taylor expansion of the look-ahead term 
\begin{equation*}
    v(\theta)^\perp\cdot\nabla_\xx c(x+\tau v(\theta))=B_\tau[c]+O(\tau^2),
\end{equation*}
that was studied in \cite{bruna2024lane}, in the elliptic case $\mathsf{t}=0$. We refer to \cite{bertucci2024curvature,bruna2024lane} for global-in-time well-posedness theory and numerical results that illustrate qualitative behaviors of the solutions of the model. A discussion on how the modeling relates to other chemotaxis models, such as the Keller--Segel model can also be found in these papers. We refer to \cite{pohl2014dynamic,liebchen2017phoretic} for a truncated linear stability analysis and more numerical simulations for closely related models.



In this paper, we show the existence of a global attractor for model~\eqref{sys:F} and we show the nonlinear instability of the steady state $(f_\ast,c_\ast)\equiv(\frac{1}{2\pi},\frac{1}{\gamma})$, under the inviscid linear instability condition,
\begin{equation}
\label{ineq:linInstabIntro}
\tag{$\mathcal{U}_k$}
    \chi(2\pi k \tau + 1) > \lambda (\gamma + 4\pi \sigma_c k^2),
\end{equation}
for some integer wave number $k\geq 1$, and for sufficiently small $\sigma_\xx,\sigma_\theta$. We also show that under this condition, the dimension of the attractor is bounded from below by $4k$. Alternatively, we show that there exists $\chi^*>0$, such that for $0 \leq \chi< \chi^*$, the steady state $(f_*,c_*)$ is globally asymptotically stable.

This instability condition \eqref{ineq:linInstabIntro}, for the case $\tau=0$, was already derived in a physics paper \cite{pohl2014dynamic}. The condition \eqref{ineq:linInstabIntro} indicates that the steady state $(f_*,c_*)$ of the model \eqref{sys:F} is unstable if the interaction strength is sufficiently large compared to the speed.





By proving regularity results for the model, we prove the existence of the global attractor, as was done for some PDE models within the framework of dynamical systems theory, such as, for example, the nonlinear heat equation, the Navier--Stokes equation \cite{robinson2001infinite}, and the Kuramoto model \cite{giacomin2012global}.

The lower bound on the dimension of the global attractor follows from a spectral analysis of the linearized operator around the steady state $(f_*,c_*)$. This analysis shows the existence of linearly unstable eigenfunctions for the linearized operator. By an adaptation of the nonlinear perturbation theorem of \cite{henry2006geometric}, the linear instability implies nonlinear instability and hence the lower bound on the dimension of the global attractor. This approach was also followed for the three-dimensional Navier--Stokes system in \cite{ghidaglia1991lower}, and the Ginzburg-Landau model in \cite{ghidaglia1987dimension}.

Our result can be seen as a mathematical foundation for the emergence of ant trails for a well-motivated model of collective ant movement. That is, our result shows when we are guaranteed to expect to have a non-trivial attractor for this model. Furthermore, given \eqref{ineq:linInstabIntro} holds for some integer $k\geq 1$, the non-trivial orbits we construct each correspond to $ 1 \leq j \leq k $ parallel and equidistant trails on the two-dimensional torus up to translation and rotation, see Section~\ref{sec:LinearAnalysisHomogeneous}. The bifurcation of inhomogeneous stationary solutions and their stability is left for future work, as this requires a finer analysis of the spectrum of the linearized equation around the homogeneous solution.


In Section~\ref{sec:mainresults}, we introduce the definition of the weak solutions of \eqref{sys:F} we consider, and state the main result Theorem~\ref{thm:MainTheorem}. In Section~\ref{sec:SemigroupTheory}, we prove the properties for the existence of the semigroup associated to \eqref{sys:F}. Then, in Section~\ref{sec:ExistenceGlobalAttractor}, we prove the existence of the global attractor for \eqref{sys:F}. In Section~\ref{sec:LinearAnalysisHomogeneous}, we show, using spectral analysis, the existence of unstable eigenfunctions of the linearized equations around the steady state $(f_*,c_*)$, provided the instability condition \eqref{ineq:linInstabIntro} holds and $\sigma_\xx$ and $\sigma_\theta$ are sufficiently small. Finally, in Section~\ref{sec:StabilityInstability} we prove a dimensional lower bound on the global attractors by constructing solutions to the nonlinear problem for each unstable eigenfunction as obtained in Section~\ref{sec:LinearAnalysisHomogeneous}, and we also prove the global asymptotic stability of the steady state $(f_*,c_*)$ under a smallness condition on $\chi$.

%% file: MainResults.tex
In this section, we introduce notation and recall definitions and results from dynamical systems theory, and then state the main result of this paper.

\subsection{Notations}
In the following, we study the function spaces defined on the domain $\mathbb{T}^2_1\times \mathbb{T}_{2\pi}$ and $\mathbb{T}^2_1$, where $\mathbb{T}_{L} = \dR/L\mathbb{Z}$ is the $L$-periodic torus. For notational conciseness,  we denote by $L^p_x(L^r_\theta)$  for $1\leq p,r\leq \infty$ the space of real-valued integrable functions $L^p(\mathbb{T}^2_1, L^r(\mathbb{T}_{2\pi}))$, equipped with the norm,
\begin{equation*}
    \|f\|_{L^p_x(L^r_\theta)} = \left(\int_{\Torus^2_1} \left(\int_{\mathbb{T}_{2\pi}} |f(x,\theta)|^p \dd\theta\right)^\frac{p}{r}\dd x\right)^\frac{1}{p},
\end{equation*}
changing for $\esssup$ if $r = \infty$,
\begin{equation*}
    \|f\|_{L^p_x(L^\infty_\theta)} = \left(\int_{\Torus^2_1} \left(\esssup_{\theta \in \mathbb{T}_{2\pi}} |f(x,\theta)|\right)^p\dd x\right)^\frac{1}{p}.
\end{equation*}

Similarly, we denote $L^q_{t,loc}(L^p_x(L^r_\theta))$ for the Bochner spaces $L^q_{loc}([0,+\infty),L^p_x(L^r_\theta))$ of locally-in-time integrable functions taking values in the Banach space $L^p_x(L^r_\theta)$. For integrals we often drop the symbol $dx$ and the domain, when it is obvious what the domain is. The space $W^{k,p}$ denotes the Sobolev space of functions with $k$-weak derivatives in $L^p$, and we use the subscripts $\cdot_x$, $\cdot_\theta$ and $\cdot_{x,\theta}$ to specify the domains, similarly as in the $L^p$ case, e.g. $W^{k,p}_x = W^{k,p}(\mathbb{T}^2_1)$. We denote by $H^k$ the Hilbert space $W^{k,2}$, $p=2$.
If $\mathcal{Y}$ is a Banach space, $C_t(\mathcal{Y})$ is the space of continuous functions, $C([0,+\infty),\mathcal{Y})$. Depending on the context, when working with a finite time horizon $T>0$, $C_t(\mathcal{Y})$ denotes the Banach space of continuous functions $C([0,T],\mathcal{Y})$ equipped with the $\sup$-norm.  If $\mathcal{Y}$ is an $L^p$ space, $(\mathcal{Y})_{+}$ denotes the cone of non-negative functions in $\mathcal{Y}$. We use the $*$-symbol to represent the periodic convolution operation.

Given a function $g\in W^{2,p}_x$, we denote by $B_\tau[g]$ the function in $L^p_x(L^\infty_\theta)$ defined as,
\begin{equation*}
    B_\tau[g](x,\theta) = v^\perp(\theta)\cdot\nabla_\xx g(\xx)+\tau v^\perp(\theta) \cdot\nabla_\xx^2 g(\xx) v(\theta).
\end{equation*}

Recalling the definition of $v$, one obtains that for any integer number $k \geq 0$, $\partial^k_\theta B_\tau[g] \in L^p_x(L^\infty_\theta)$.
We use the notation for the full gradient $\nabla_\xi = \begin{pmatrix} \partial_{x_1},\partial_{x_2},\partial_{\theta}\end{pmatrix}^\mathsf{T}$.

Finally, we define for a fixed $\frac{12}{5} \leq r < 6$, required for Rellich–Kondrachov's embedding theorem as we explain later, the following convex functional spaces,
\begin{align}
    \label{def:StateSpaceSemigroupParabolic}
    \tag{$Y^P$}
    Y^P &\defeq \left\{(f,c) \in (L^r_{x,\theta})_+ \times W^{2,6}_x \Big| \int f \dd \theta \in H^1_x , \int f \dd \theta\dd x  = 1\right\},\\
    \label{def:StateSpaceSemigroupElliptic}
    \tag{$Y^E$}
    Y^E &\defeq  \left\{f\in (L^r_{x,\theta})_+\cap L^6_x(L^1_\theta)_+ \Big| \int f \dd\theta \dd x  = 1\right\}.
\end{align}
These spaces will be used for defining the semigroup for \eqref{sys:F}, for the parabolic and elliptic case, respectively.

\subsection{Preliminaries}

We recall the following notions from dynamical systems theory. Proofs of the following propositions can be found in \cite[Chapter 10]{robinson2001infinite}.

\begin{definition}
    \label{def:Semigroup}
    Let $X$ be a Banach space, and $Y$ be a closed subset of $X$. We say that $(Y,\{S(t)\}_{t\geq 0})$ is a semi-dynamical system, if for all $0\leq t$, there is a map $S(t): Y \to Y$ and the following properties hold,
    \begin{enumerate}
        \item $S(0) = \mathrm{Id}_Y$,
        \item $S(t)\circ S(s) = S(s)\circ S(t) = S(t+s)$.
        \item For any $u \in Y$, $S(t)u$ is continuous in $u$ and $t$.
    \end{enumerate}
    If the above are satisfied, $S$ is said to be a $C^0$-semigroup on $Y$.
\end{definition}

\begin{definition}
    The global attractor $\mathcal{A}\subset \subset Y$ of a semi-dynamical system $(Y,\{S(t)\}_{t\geq0})$, if it exists, is the maximal compact invariant set,
    \begin{equation*}
        S(t)\mathcal{A} = \mathcal{A} \ \ \text{ for all } t \geq 0,
    \end{equation*}
    and the minimal set that attracts all bounded sets,
    \begin{equation*}
        \text{dist}(S(t)V,\mathcal{A})\xrightarrow[]{} 0 \text{ as } t\xrightarrow[]{} \infty,
    \end{equation*}
    for any bounded set $V\subset Y$.
\end{definition}

\begin{definition}\label{def:dissip}
    A semi-dynamical system is said to be dissipative if it possesses a compact absorbing set $K \subset \subset Y$. That is, for any bounded set $B$ there exists $t_0(B)\geq 0$ such that,
    \begin{equation*}
        S(t)B\subset K \text{ for all } t\geq t_0(B).
    \end{equation*}
\end{definition}

\begin{prop}
    \label{prop:ExistenceofAttractionforDissipation}
    If $(Y,\{S(t)\}_{t\geq0})$ is dissipative and $K\subset\subset Y$ is a compact absorbing set, then there exists a global attractor $\mathcal{A}$, and
    \begin{equation*}
        \mathcal{A} = \bigcap_{t\geq 0} S(t)K.
    \end{equation*}
    If $Y$ is connected then so is $\mathcal{A}$.
\end{prop}

\begin{definition}
    The unstable manifold $\mathcal{M}^u(z_*)$ at a point $z_*\in Y$ is defined as the set of all $u_0 \in Y$ such that there exists a global solution $u\in C((-\infty,0],Y)$ satisfying $u(0) = u_0$, and for any $t\geq s \geq 0$,
    $S(s)u(-t) = u(s-t)$, and $\lim_{t\to -\infty} u(t) = z_*$,
    \begin{equation*}
        \mathcal{M}^u(z_*) = \{u_0 \in Y| S(-t)u_0\to z_* \text{ as } t\to \infty\}.
    \end{equation*}
\end{definition}

\begin{prop}
    If $\mathcal{K}$ is a compact invariant set, then
    \begin{equation*}
        \mathcal{M}^u(\mathcal{K}) \subset \mathcal{A}.
    \end{equation*}
\end{prop}

We now recall the notion of solutions introduced in~\cite{bertucci2024curvature}, for the parabolic case. This allows to obtain existence and uniqueness results for the parabolic system, and we introduce the notion of the solution for the elliptic system. We later require the initial conditions to lie in the spaces~\eqref{def:StateSpaceSemigroupParabolic} and \eqref{def:StateSpaceSemigroupElliptic} to define the associated $C^0$-semigroups, and provide further regularity on such solutions.

\begin{definition}
    \label{def:FormicidaeParabolicSolution}
    Let $4<p< \infty$, and suppose that $\gamma \geq 0, \sigma_x > 0, \sigma_\theta > 0, \sigma_c > 0, \lambda,\tau,\chi \geq 0$.
    For $f_0 \in L^p_x(L^1_\theta)_+, c_0 \in W^{2,p}_x$, a couple $(f, c)\in C(\dR_+, L^p_x(L^1_\theta)_+)\times L^p_{t, loc}(\dR_+,W^{2,p}_x)$, is said to be a solution of the Cauchy problem~\eqref{sys:FormicidaeParabolic},
    \begin{equation}
        \label{sys:FormicidaeParabolic}
        \tag{$\mathcal{F}^P$}
        \begin{cases}
            \partial_t \ff=\nabla_\xx\cdot(\sigma_\xx\nabla_\xx\ff -\lambda v\ff)+\partial_\theta(\sigma_\theta\partial_\theta \ff -\chi B_\tau[c]\ff),\\
            \partial_t c = - \gamma c + \sigma_c \Delta_x c + \int f \dd\theta
        \end{cases}
    \end{equation}
    if $f$ stays positive, its mass is preserved for all times, the Fokker-Planck equation holds in the distributional sense,
    \begin{equation*}
        \int \varphi(t) f(t)  - \int \varphi(0) f_0 = \int_0^t \int (\partial_t \varphi + \sigma_\theta \partial_{\thth} \varphi + \sigma_x\Delta_{x}\varphi + \chi B_\tau[c]\partial_\theta \varphi + \lambda v \cdot \nabla_x \varphi) f \dd s,
    \end{equation*}
    $\forall \varphi \in C^{\infty}_b, \forall t \geq 0$, and the chemical field equation holds in $L^p_x$ pointwise in time.
\end{definition}

\begin{definition}
    \label{def:FormicidaeEllipticSolution}
    Let $4<p< \infty$, and suppose that $\gamma \geq 0, \sigma_x > 0, \sigma_\theta > 0, \sigma_c > 0, \lambda,\tau,\chi \geq 0$.
    For $f_0 \in L^p_x(L^1_\theta)_+$, a function $f \in C(\dR_+, L^p_x(L^1_\theta)_+)$, is said to be a solution of the Cauchy problem~\eqref{sys:FormicidaeElliptic},
    \begin{equation}
        \label{sys:FormicidaeElliptic}
        \tag{$\mathcal{F}^E$}
        \begin{cases}
            \partial_t \ff=\nabla_\xx\cdot(\sigma_\xx\nabla_\xx\ff -\lambda v\ff)+\partial_\theta(\sigma_\theta\partial_\theta \ff -\chi B_\tau[c]\ff),\\
            \gamma c(t) - \sigma_c \Delta_x c(t) = \int f(t) \dd\theta
        \end{cases}
    \end{equation}
    if its mass is preserved for all times, the Fokker-Planck equation holds in the distributional sense as above and the chemical field equation holds in $L^p_x$ pointwise in time.
\end{definition}

\subsection{Statement of the main result}

We now state the main result of this paper. This statement summarizes Theorem~\ref{thm:FormicidaeParabolicSemiGroup} which is about the well-posedness of the parabolic-parabolic semigroup, Theorem~\ref{thm:InstabilityParabolic} on the parabolic-parabolic nonlinear instability result and Theorem~\ref{thm:GlobalAsymptoticStability} on the parabolic-parabolic small-$\chi$ global stability. The same statement holds for the elliptic case~\eqref{sys:FormicidaeElliptic}, summarizing Theorem~\ref{thm:FormicidaeEllipticSemiGroup}, Theorem~\ref{thm:InstabilityElliptic} and Theorem~\ref{thm:GlobalAsymptoticStabilityElliptic} with the same order of contents for the parabolic-elliptic model. We only present the parabolic-parabolic case in full here now.

\begin{thm}
    \label{thm:MainTheorem}
    For any $\sigma_x, \sigma_\theta, \sigma_c, \gamma > 0$, the system~\eqref{sys:FormicidaeParabolic} defines a semidynamical system,
    \begin{equation*}
        (Y^P,\{S^P(t)\}_{t\geq0 }),
    \end{equation*}
    where $S^P$ is the solution operator of~\eqref{sys:FormicidaeParabolic} in~\eqref{def:StateSpaceSemigroupParabolic}. The $C^0$-semigroup $S^P$ is dissipative and possesses a compact global attractor,
    \begin{equation*}
        \mathcal{A}^P \subset \subset Y^P.
    \end{equation*}
    Furthermore, two distinct cases arise depending on the value of the parameters:
    \begin{enumerate}
        \item For any $\lambda,\tau \geq 0$, there exists $\chi^*>0$, such that for any $0\leq \chi < \chi^*$, the associated attractor is given by,
        \begin{equation*}
            \mathcal{A}^P = \{(f_*,c_*)\},
        \end{equation*}
        recalling that $(f_*,c_*)$ is the homogeneous steady state $(f_*, c_*)=(1/2\pi,1/\gamma)$.
    
        \item Alternatively, if the instability condition
        \begin{equation}
            \tag{$\mathcal{U}_k$}
            \chi(2\pi k \tau + 1) > \lambda (\gamma + 4\pi \sigma_c k^2),
        \end{equation}
        holds for some integer wavenumber $k\geq 1$, then there exists $\sigma_\theta^*>0$ such that for $\sigma_\theta \in (0,\sigma_\theta^*)$, there exists $\sigma_x^*(\sigma_\theta)>0$ such that for $\sigma_x \in (0,\sigma_x^*)$, we obtain
        \begin{equation*}
            4k \leq \dim \mathcal{A}^P.
        \end{equation*}                             
    \end{enumerate}
\end{thm}

%% file: SemigroupTheory.tex
In this section, we prove that the systems~\eqref{sys:FormicidaeParabolic} and \eqref{sys:FormicidaeElliptic} each generate a semigroup, by building on the existence and uniqueness theory established in \cite{bertucci2024curvature}, in the spaces~\eqref{def:StateSpaceSemigroupParabolic} and \eqref{def:StateSpaceSemigroupElliptic}. Furthermore, we prove that solutions in this functional framework enjoy additional regularity, specifically \( f \in L^2_{t,\text{loc}}(H^1_{x,\theta}) \). Finally, we prove the existence of the global attractors associated with these semigroups, given the results of Section~\ref{sec:ExistenceGlobalAttractor}.

We here recall the existence and uniqueness theory of the parabolic-parabolic system~\eqref{sys:FormicidaeParabolic}, obtained in \cite[Theorem 3.5]{bertucci2024curvature}.
\begin{prop}
    \label{prop:ExistenceUniquenessFormicidaeParabolicSolution}
    Let $4<p< \infty$, suppose that $\gamma \geq 0, \sigma > 0, \lambda > 0$ and  $\chi \geq 0$.
    
    Then, 
    for any initial condition $f_0 \in L^p_x(L^1_\theta)_+$, $c_0 \in W^{2,p}_x$,
    there exists a unique global solution $(f,c) \in C(\dR_+, L^p_x(L^1_\theta)_+)\times L^p_{t,loc}(\dR_+,W^{2,p}_x)$ of system~\eqref{sys:FormicidaeParabolic} in the sense of Definition~\ref{def:FormicidaeParabolicSolution}.
\end{prop}

The fixed point argument of \cite[Theorem 3.5]{bertucci2024curvature} can be easily adapted to obtain the existence and uniqueness of the elliptic system, by controlling the norm of $c$ in $C_t(W^{2,6}_x)$ with the norm of $f$ in $C_t(L^6_x(L^1_\theta)_+)$, using elliptic regularity theory \cite[Chapter 1]{krylov2008lectures}.
\begin{prop}
    \label{prop:ExistenceUniquenessFormicidaeEllipticSolution}
    Let $4<p< \infty$, suppose that $\gamma \geq 0, \sigma > 0, \lambda > 0$ and  $\chi \geq 0$.
    
    Then, 
    for any initial condition $f_0 \in L^p_x(L^1_\theta)_+$,
    there exists a unique global solution $f \in C(\dR_+, L^p_x(L^1_\theta)_+)$ of system~\eqref{sys:FormicidaeElliptic} in the sense of Definition~\ref{def:FormicidaeEllipticSolution}.
\end{prop}

We now establish the following additional regularity of distributional solutions of the Fokker-Planck equation~\eqref{eq:fBgeneral}, assuming further integrability on the initial phase space density $f_0$, and for a given $\theta$-drift function $B(t,\xx,\theta)$.

\begin{prop}
    \label{prop:ActiveMatterEquation}
    Let $4<p< \infty$, $1\leq s < +\infty$ and $q\geq \frac{p}{p-1}$. For any $B \in L^p_{t,loc}(L^p_x(L^\infty_\theta))$,  and $f_0 \in L^q_x(L^s_\theta)_+$, there exists a unique $f \in C_t(L^q_x(L^s_\theta)_+)$ distributional solution of,
    \begin{equation}\label{eq:fBgeneral} \tag{$\mathcal{F}^B$}
        \partial_t f=\nabla_\xx\cdot(\sigma_\xx\nabla_\xx f-\lambda v f)+\partial_\theta(\sigma_\theta \partial_\theta f-\chi Bf).
    \end{equation}
    Furthermore, if $q\geq \frac{2p}{p-1}$ and $s \geq 2$, and $\partial_\theta B \in L^p_{t,loc}(L^p_x(L^\infty_\theta))$ then $f$ is in the space,
    \begin{equation*}
        f \in L^{2}_{t,loc}(H^1_{x,\theta})\cap C_t(L^q_x(L^s_\theta)_+).
    \end{equation*}
\end{prop}
\begin{proof}
    The proof of \cite[Theorem 4.1]{bertucci2024curvature} can be extended from $s=1$ to $1\leq s<+\infty$. We only give it as an \textit{a priori} estimate, and the existence, uniqueness and stability results are obtained as in~\cite[Section 4]{bertucci2024curvature}.

    For this, let $g : (0,+\infty)\times \Torus_1^2 \times \Torus_{2\pi} \to \dR_+$, be the fundamental solution of the anisotropic heat equation,
    \begin{equation}
        \partial_t g = \sigma_x \Delta_x g + \sigma_\theta \partial_{\theta\theta} g,
    \end{equation}
    satisfying the integrability estimates,
    \begin{align}
        \label{est:HeatEqFurtherRegFokkerplanck1}
        \nabla_x g  &\in L^1_{t,loc}\left(L^1_{x,\theta}\right), \\
        \label{est:HeatEqFurtherRegFokkerplanck2}
        \partial_\theta g  &\in L^{\frac{p}{p-1}}_{t,loc}\left(L^{\frac{p}{p-1}}_x(L^1_\theta)\right),
    \end{align}
    given $p>4$, see e.g~\cite[Proposition 3.2]{bertucci2024curvature}.
    
    We also note by hypothesis that,
    \begin{equation}
        \label{est:furtherRegFokkerplanck}
        Bf\in L^p_{t,loc}\left(L^{\frac{pq}{p+q}}_x(L^s_\theta)\right).
    \end{equation}
    
    Let $f$ be an \textit{a priori} solution of the integral equation,
    \begin{equation}
        f_t = f_0 * g_t - \int _0^t (\partial_\theta g_{t-s} * (B_s f_s))\mathrm{d}s - \int^t_0 (\nabla_x g_{t-s} * (\lambda vf_s))\mathrm{d}s.
    \end{equation}
    Then applying Bochner integral inequality and Young convolution inequality,
    \begin{align*}
         \|f_t\|_{L^q_x(L^s_\theta)} & \leq  \|f_0 * g_t\|_{L^q_x(L^s_\theta)}  + \int _0^t \|\partial_\theta g_{t-s} * (B_s f_s)\|_{L^q_x(L^s_\theta)} \mathrm{d}s \\
         &\hspace{3em} + \int^t_0 \|\nabla_x g_{t-s} * (\lambda vf_s)\|_{L^q_x(L^s_\theta)}  \mathrm{d}s,\\
         & \leq  \|f_0\|_{L^q_x(L^s_\theta)}  \|g_t\|_{L^1_{x,\theta}} + \int _0^t \|\partial_\theta g_{t-s}\|_{L^{p/(p-1)}_x(L^1_\theta)} \|B_s f_s\|_{L^{pq/(p+q)}_x(L^s_\theta)} \mathrm{d}s \\
         &\hspace{3em} + \lambda \int^t_0 \|\nabla_x g_{t-s}\|_{L^1_{x,\theta}}\|f_s\|_{L^q_x(L^s_\theta)}\mathrm{d}s.
    \end{align*}
    From the integrability estimates \eqref{est:furtherRegFokkerplanck}, \eqref{est:HeatEqFurtherRegFokkerplanck1} and \eqref{est:HeatEqFurtherRegFokkerplanck2}, the right hand side is bounded and this implies the preservation of integrability \textit{a priori}. The same computations lead to existence and uniqueness by applying Banach-Picard fixed point theorem, the stability estimate and the growth estimate follow from the Gr\"onwall type inequality~\cite[Proposition 3.3]{bertucci2024curvature}. We refer to~\cite[Section 4.]{bertucci2024curvature} for more details. We conclude the existence of a unique positive global in time solution in the space,
    \begin{equation*}
        f \in C_t(L^q_x(L^s_\theta)_+).
    \end{equation*}
    
    For the existence of a weak derivative, we proceed as follows. We show that we have a sequence of mollified solutions, uniformly bounded in  $L^2([0,T],H^1_{x,\theta})$, that converges to the solution $f$, using the stability of the Fokker--Planck equation. The Banach-Alaoglu theorem then gives that $f$ is in $L^{2}_{t,loc}(H^1_{x,\theta})$. 
    
    The following computation is thus made rigorous by mollifying the initial data and the scalar-field $B$, and performing the estimate on smooth solutions.
    We use the following operator notation,
    \begin{equation*}
        \widetilde{\nabla}_{\xi} = \begin{pmatrix}
            \sqrt{\sigma_x}\partial_{x_1} \\
            \sqrt{\sigma_x}\partial_{x_2} \\
            \sqrt{\sigma_\theta}\partial_{\theta}
        \end{pmatrix}.
    \end{equation*}
    
    Multiplying the equation by $f$ and integrating in $x,\theta$ we obtain, 
    \begin{align*}
        \frac{d}{dt} \int \frac{f^2}{2} &= - \int |\widetilde{\nabla}_{\xi} f|^2 + \chi \int \partial_\theta f f B + \lambda \int \nabla_x f \cdot v f,\\
        &= - \int |\widetilde{\nabla}_{\xi}  f|^2 + \chi \int \partial_\theta \left(\frac{f^2}{2}\right) B + \lambda \int \nabla_x \cdot \left(v \frac{f^2}{2}\right),\\
        & = - \int |\widetilde{\nabla}_{\xi} f|^2 - \chi \int \frac{f^2}{2}\partial_\theta B.
    \end{align*}
    Integrating over time, and using the integrability of $f$ and $B$ in H\"older inequality, we obtain that,

    \begin{align*}
        \int^t_0 \int |\widetilde{\nabla}\f|^2 \mathrm{d}s &\leq \int \frac{\f_0^2}{2} + \frac{\chi}{2} \int^t_0 \left( \int f^2 |\partial_\theta B| d\theta d x\right) \mathrm{d}s, \\
        &\leq \int \frac{\f_0^2}{2} + \frac{\chi}{2}\int^t_0 \|f^2\|_{L^{\frac{p}{p-1}}_x(L^1_\theta)}\|\partial_\theta B\|_{L^p_x(L^\infty_\theta)}\mathrm{d}s,\\
        &\leq \int \frac{\f_0^2}{2} + \frac{\chi}{2}\int^t_0 \|f\|^2_{L^{\frac{2p}{p-1}}_x(L^2_\theta)}\|\partial_\theta B\|_{L^p_x(L^\infty_\theta)}\mathrm{d}s,\\
        &\leq \int \frac{\f_0^2}{2} + \frac{\chi C}{2}\int^t_0 \|f\|^2_{L^{q}_x(L^s_\theta)}\|\partial_\theta B\|_{L^p_x(L^\infty_\theta)}\mathrm{d}s,
    \end{align*}
    where $C>0$ is the constant from the embedding $L^{q}_x(L^s_\theta) \subset L^{\frac{2p}{p-1}}_x(L^2_\theta)$, by hypothesis on the exponents. The growth estimate on the integrability of $f$ concludes the proof.
\end{proof}

\subsection{Parabolic-parabolic case}

We now state the main theorems of this section. That is, we prove the existence of the semigroups, the additional regularity of the semigroup solutions, and the existence of the global attractors.

\begin{thm}
    \label{thm:FormicidaeParabolicSemiGroup}
    Suppose $\sigma_x>0,\sigma_\theta>0,\sigma_c>0$, and $6 > r\geq \frac{12}{5}$.
    The system~\eqref{sys:FormicidaeParabolic} defines a semidynamical system $(Y^P,\{S^P(t)\}_{t\geq0 })$, where $S^P$ is the solution operator of system~\eqref{sys:FormicidaeParabolic}.
    
    Furthermore the semigroup solutions satisfy the additional regularity,
    \begin{equation*}
        f\in L^{2}_{t,loc}(H^1_{x,\theta}).
    \end{equation*}
    
    Finally, for $\gamma >0$, there exists a global attractor $\mathcal{A}^P\subset\subset Y^P$ of the semigroup $S^P$.
\end{thm}

\begin{proof}
    We first note that for $(f_0, c_0)$ in $Y^P$, we have that $f_0\in L^6_x(L^1_\theta)_+$ since $\rho_0 = \int f_0 d\theta \in H^1_x$ by the Sobolev embedding. We can apply the existence and uniqueness result of Proposition~\ref{prop:ExistenceUniquenessFormicidaeParabolicSolution} providing a solution $(f,c)$ in the space,
    \begin{equation*}
        c\in L^6_{t,loc}(W^{2,6}_x), f\in C_t(L^6_x(L^1_\theta)_+).
    \end{equation*}
    Since the associated $B[c]$ satisfies the hypothesis of Proposition~\ref{prop:ActiveMatterEquation} with $p = 6$, $q = r \geq 12/5$, then
    \begin{equation*}
        f\in C_t((L^r_{x,\theta})_+) \cap L^{2}_{t,loc}(H^1_{x,\theta}).
    \end{equation*}
    
    We now check the continuity of $c$ and $\rho$ in $W^{2,6}_x$ and $H^1_x$, respectively. We proceed as follows. Using the additional regularity of $f$, we have that $\rho$ is a solution of the heat equation with a source term,
    \begin{equation*}
        \partial_t \rho = \sigma_x \Delta_x \rho - \lambda  \int v\cdot \nabla_x f \dd\theta,
    \end{equation*}
    with initial condition in $\rho_0\in H^1_x$, and the source term $\lambda  \int v\nabla_x f \dd\theta \in L^2_{t,loc}(L^2_x)$. By parabolic regularity theory \cite[Chapter 7, Theorem 5]{evans2010partial} we have,
    \begin{equation*}
        \rho \in L^\infty_{t,loc}(H^1_x) \cap L^2_{t,loc}(H^2_x) \cap H^1_{t,loc} (L^2_x).
    \end{equation*}
    In particular, we have $\nabla^2_x \rho \in L^2_{t,loc}(L^2_x)$. For the continuity of $c$ in $W^{2,6}_x$, we make use of the integrability properties of the fundamental solution of the heat equation on $\Torus_1^2$.

    For this, let $g : (0,+\infty)\times \Torus_1^2 \to \dR_+$, be the fundamental solution to the heat equation on the two dimensional torus,
    \begin{equation*}
        \partial_t g = \sigma_c \Delta_x g,
    \end{equation*}
    satisfying the integrability estimate,
    \begin{equation*}
        g \in L^2_{t,loc}(L^{3/2}_x(\Torus^2_1)),
    \end{equation*}
    see e.g~\cite[Appendix]{bertucci2024curvature}. Then, by the Duhamel formula,
     \begin{equation*}
         \nabla^2_x c(t) = e^{-\gamma t}g_t  * \nabla^2_x c_0 + \int_0^t e^{-\gamma (t-s)}g_{t-s} * \nabla^2_x \rho_s \mathrm{d}s,
    \end{equation*}
    and applying the Young convolution inequality, we obtain
    \begin{align*}
         \|\nabla^2_xc(t)\|_{L^6_x} &\leq e^{-\gamma t}\|g_t*\nabla^2_xc_0\|_{L^6_x} + \int_0^t e^{-\gamma (t-s)}\|g_{t-s}*\nabla^2_x\rho_s\|_{L^6_x} \mathrm{d}s,\\
         &\leq e^{-\gamma t}\|g_t\|_{L^1_x} \|\nabla^2_xc_0\|_{L^6_x} + e^{-\gamma (t-s)}\int_0^t \|g_{t-s}\|_{L^{3/2}_x}\|\nabla^2_x\rho_s\|_{L^2_x} \mathrm{d}s.
    \end{align*}
    The integrability of $g$ and $\nabla_x^2 \rho$ ensures that the right hand side is bounded. The same can be applied to $\nabla_x c(t)$ and $c(t)$, which together leads to 
    \begin{equation*}
        c \in C_t(W^{2,6}_x).
    \end{equation*}
    The continuity of $\rho$ in $H^1_x$, follows by applying Lemma~\ref{lem:weakL2convergenceLinfty} with $V = H^1_x$ and $H = L^2_x$. The continuity of the semi-group $(f_0,c_0,t) \mapsto (f(t),c(t))$ in $Y^P$, follows from the previous stability estimates, together with a joint estimate as in \cite[Section 5.1]{bertucci2024curvature}.

    The existence of the global attractor is an application of Proposition~\ref{prop:ExistenceofAttractionforDissipation}, using the estimates 
    of Theorem~\ref{thm:AbsorbH3H1Parabolic} and Corollary~\ref{cor:H4CParabolic}, together with the Rellich–Kondrachov compact embedding theorem in dimension three and two,
    \begin{equation*}
        H^1_{x,\theta} \subset\subset L^r_{x,\theta},\hspace{1em} H^2_x \subset\subset H^1_x,\hspace{1em} H^4_x \subset\subset W^{2,6}_x,
    \end{equation*}
    for any $\frac{12}{5}\leq r < 6$, as was fixed earlier.
\end{proof}

\subsection{Parabolic-elliptic case}
Similarly, we prove the equivalent statement of Theorem~\ref{thm:FormicidaeParabolicSemiGroup} for the parabolic-elliptic case.

\begin{thm}
    \label{thm:FormicidaeEllipticSemiGroup}
    Suppose $\sigma_x>0,\sigma_\theta>0,\sigma_c>0$, and $6 > r\geq \frac{12}{5}$.
    The system~\eqref{sys:FormicidaeElliptic} defines a semidynamical system $(Y^E,\{S^E(t)\}_{t\geq0 })$, where $S^E$ is the solution operator of system~\eqref{sys:FormicidaeElliptic}.
    
    Furthermore the semigroup solutions satisfy the additional regularity,
    \begin{equation*}
        f\in L^{2}_{t,loc}(H^1_{x,\theta}).
    \end{equation*}
    
    Finally, for $\gamma >0$, there exists a global attractor $\mathcal{A}^E\subset\subset Y^E$ of the semigroup $S^E$.
\end{thm}

\begin{proof}
    We first apply the existence and uniqueness result of Proposition~\ref{prop:ExistenceUniquenessFormicidaeEllipticSolution} providing a solution $f$ in the space,
    \begin{equation*}
        f\in C_t(L^6_x(L^1_\theta)_+)\cap C_t((L^r_{x,\theta})_+).
    \end{equation*}

    Since the associated $\partial_\theta B[c] \in L^\infty_t(L^6_x(L^\infty_\theta)_+)$ satisfies the hypothesis of Proposition~\ref{prop:ActiveMatterEquation} with $p = 6$, $q = r \geq 12/5$, then
    \begin{equation*}
        f\in L^{2}_{t,loc}(H^1_{x,\theta}).
    \end{equation*}

    This concludes the existence of the semigroup in $Y^E$ and the additional regularity.
    Finally, the existence of the global attractor is an application of Proposition~\ref{prop:ExistenceofAttractionforDissipation}, using estimates 
    of Theorem~\ref{thm:AbsorbH3H1Elliptic}, together with the Rellich--Kondrachov compact embedding theorem in dimension three,
    \begin{equation*}
        H^1_{x,\theta} \subset\subset L^r_{x,\theta}\cap L^6_x(L^1_\theta).
    \end{equation*}
    for any $\frac{12}{5}\leq r < 6$, as was fixed earlier.
    
\end{proof}

%% file: GlobalAttractor.tex

In this section, we prove that the parabolic-parabolic semigroup as defined via Theorem~\ref{thm:FormicidaeParabolicSemiGroup} and the parabolic-elliptic semigroup of Theorem \ref{thm:FormicidaeEllipticSemiGroup} are dissipative, as in Definition~\ref{def:dissip}. We first show that the spatial density $\rho$ of the semigroup solution is absorbed in $L^p_x$ for $1\leq p \leq 6$, both in the parabolic and elliptic case. In the elliptic case, we prove that the full phase space density $f$ is absorbed in $L^2_{x,\theta}$, and is, consequently, absorbed in $H^1_{x,\theta}$. In the parabolic case, we use joint estimates for $f$ and $c$, to prove first the absorption in $L^2_{x,\theta}\times H^2_{x}$ and then the absorption in  $H^1_{x,\theta}\times H^3_{x}$, and finally in $H^4_x$ for $\rho$. Using the compactness embedding argument at the end of the proofs in Section~\ref{sec:SemigroupTheory}, this shows the dissipativity of the semigroups.

We now recall some results from functional analysis. We intensively use the following Gr\"onwall inequality \cite[Lemma 2.8]{robinson2001infinite} for absorption estimates.

\begin{lem}[Gr\"onwall Inequality]
\label{lem:gronwall}
    Let $x\in C(\mathbb{R}_+, \mathbb{R})$ satisfy the differential inequality
    \begin{equation}
        \frac{\dd}{\dd t}_+x\leq a x+b,
    \end{equation}
    such that $a,b\in\mathbb{R}$ are constants, and $\frac{\dd}{\dd t_+} x(t) \defeq \lim_{h \searrow 0
} (x(t+h) -x(t))/h $.
    
    Then,
    \begin{equation}
        x(t)\leq \left(x(0)+\frac{b}{a}\right)e^{at}-\frac{b}{a}.
    \end{equation}
\end{lem}

A proof of the following lemma can be found in \cite[Lemma 11.2]{robinson2001infinite}.

\begin{lem}
    \label{lem:weakL2convergenceLinfty}
    Let $H$ and $V$ be Banach spaces such that $V$ is compactly embedded in $H$, $V\subset \subset H $. Suppose that the sequence $(u_n)_n$ is uniformly bounded in $L^\infty([0,T], V)$,
    \begin{equation*}
        \esssup_{t\in [0,T]} \|u_n(t)\|_{V} \leq C,
    \end{equation*}
    and let $u_n \rightharpoonup u$ in $L^2([0,T],V)$, then,
    \begin{equation*}
        \esssup_{t\in [0,T]} \|u(t)\|_{V} \leq C.
    \end{equation*}
    Furthermore, if $u\in C^0([0,T], H)$, then in fact,
    \begin{equation*}
        \sup_{t\in [0,T]} \|u(t)\|_V \leq C.
    \end{equation*}
\end{lem}

Hereafter, we will use multiple times the Gagliardo-Nirenberg inequality with different exponents, we thus recall here its most general form \cite[Theorem in Lecture II]{nirenberg1959elliptic}.

\begin{prop}[Gagliardo--Nirenberg Inequality]
    \label{thm:GagliardoNirenbeg}
    Let $1 \leq q,r \leq \infty$, $j,m$ two integers such that $m>j\geq 0$, and let $p\geq 1$ and $\alpha \in \left[\frac{j}{m},1\right)$, such that the following relation holds,
    \begin{equation*}
        \frac{1}{p} = \frac{j}{d} + \alpha \left( \frac{1}{r} - \frac{m}{d}\right) + \frac{1-\alpha}{q}.
    \end{equation*}
   Then, for any $s\geq 1$, there exists a constant $C>0$, such that,

    \begin{equation*}
        \|D^j u\|_p \leq C \|D^m u\|^\alpha_r \|u\|^{1-\alpha}_q + C \|u\|_s,
    \end{equation*}

    for any function $u \in L^p(\Torus^d)$ such that $D^mu \in L^r(\Torus^d)$.

    Furthermore, if $m-j-d/r > 0$, one can take $\alpha = 1$.
\end{prop}

We now prove the absorption estimate on the spatial density $\rho$, that holds both in the parabolic and elliptic case from the regularity result of Proposition~\ref{prop:ActiveMatterEquation}.

\begin{prop}
    \label{prop:rhoestimLp}
    Let $B\in L^q_{loc,t}(L^{r}_x(L^\infty_\theta))$ be a given scalar field with $6\leq q,r\leq +\infty$, $\sigma_\theta >0, \sigma_x >0$, and let $f \in L^2_{loc,t}(H^1_{x,\theta})$ be the unique solution to the Fokker-Planck equation:
    \begin{equation}
        \label{eq:ActiveMaterEquation}
        \tag{$\mathcal{F}^B$}
        \begin{cases}
            \partial_t \ff = \nabla_x \cdot (\sigma_x \nabla_x \ff - \lambda v \ff) + \partial_\theta(\sigma_\theta\partial_\theta \ff -\chi B\ff),\\
            f(0) = f_0,
        \end{cases}
    \end{equation}
    given that $f_0 \in L^p_x(L^1_\theta)_+$, for some $1\leq p < + \infty$.
    
    Then, $\rho$ is absorbed in $L^p_x$ independently of $B$. That is, there exists a constant $C_p>0$ and $\alpha_p >0$, depending only on $p$, $\sigma_x$ and $\lambda$, such that, for any $\rho$ satisfying the above, the following estimate holds,
    \begin{equation}
        \label{est:rhoestimLp}
        \|\rho(t)\|_{L^p_x} \leq e^{-\alpha_p t}\|\rho_0\|_{L^p_x} + C_p, \text{ for all } t\geq 0.
    \end{equation}
\end{prop}

\begin{proof}
    We prove the absorption in $L^p_x$ by performing the following computations on a sequence of regularized solutions, obtained by mollifying the scalar field $B$ and the initial condition $\rho_0$, so that all the following integrals are well defined, and then applying the stability argument as in the proof of Proposition~\ref{prop:ActiveMatterEquation}, using \cite[Lemma 4.2]{bertucci2024curvature}

    Integrating over $\theta$ the equation on $f$ and multiplying by $\rho^{p-1}$, we obtain,
    \begin{equation*}
        \frac{\dd}{\dd t} \int \frac{\rho^p}{p} = \sigma_x \int \rho^{p-1}\Delta_x \rho - \lambda \int \rho^{p-1}\int v\cdot \nabla_x f \dd\theta.
    \end{equation*}
    Applying the divergence theorem on the right hand side,
    \begin{align*}
        \frac{\dd}{\dd t} \int \frac{\rho^p}{p(p-1)} & = - \sigma_x \int |\nabla_x \rho|^2 \rho^{p-2} - \lambda \int \rho^{p-2} \nabla_x \rho \cdot \int v f \dd\theta,\\
        & \leq- \sigma_x \int |\nabla_x \rho|^2 \rho^{p-2}  + \lambda \int \rho^{p-2} |\nabla_x \rho| \int |f| \dd\theta,\\
        & \leq - \sigma_x \int |\nabla_x \rho|^2 \rho^{p-2}  + \lambda \int \rho^{\frac{p-2}{2}} |\nabla_x \rho|\rho^{\frac{p}{2}},\\
        & \leq - \sigma_x \int |\nabla_x \rho|^2 \rho^{p-2}  +\mu \int |\nabla_x \rho|^2 \rho^{p-2}  + \frac{\lambda^2}{4\mu} \int \rho^p,
    \end{align*}
    for any $\mu > 0$, where we used that $ f \geq 0$.
    
    Taking $\mu = \frac{\sigma_x}{2}$,
    \begin{equation}
        \frac{\dd}{\dd t} \int \frac{\rho^p}{p(p-1)} \leq - \frac{\sigma_x}{2} \int |\nabla_x \rho|^2 \rho^{p-2}  + \frac{\lambda^2}{2\sigma_x} \int \rho^p,
    \end{equation}
    and using the identity,
    \begin{equation*}
        \frac{4}{p^2}\left|\nabla_x(\rho^\frac{p}{2})\right|^2 = \left|\nabla_x \rho\right|^2\rho^{p-2},  
    \end{equation*}
    this leads to,
    \begin{equation}
        \label{eq:rhoLpestim}
        \frac{\dd}{\dd t} \int \frac{\rho^p}{p(p-1)} \leq - \frac{2\sigma_x}{p^2} \int \left|\nabla_x(\rho^\frac{p}{2})\right|^2  + \frac{\lambda^2}{2\sigma_x} \int \rho^p\\.
    \end{equation}
Recalling the following Gagliardo-Nirenberg inequality in dimension two \cite[Chapter 8, p.~233, 1.(iii)]{brezis2011functional},

\begin{equation*}
    \|u\|^2_{L^2} \leq C_{GN}\|u\|_{1}\|u\|_{H^1}, \text{ for any } u \in H^1_x.
\end{equation*}
Applying the above to $u = \rho^{\frac{p}{2}}$, using the Young product inequality and rearranging, we have the following estimate on the first right term of \eqref{eq:rhoLpestim},

\begin{equation*}
     - \int \left|\nabla_x (\rho^{\frac{p}{2}})\right|^2 \leq - \frac{1-\mu}{\mu}\int \rho^p + \frac{C^2_{GN}}{\mu^2} \left( \int \rho^{\frac{p}{2}}\right)^2,
\end{equation*}
for any $\mu>0$. Plugging this in \eqref{eq:rhoLpestim}, we get,

\begin{equation*}
    \frac{\dd}{\dd t} \int \frac{\rho^p}{p(p-1)} \leq \left(\frac{\lambda^2}{2\sigma_x}
    - \frac{2\sigma_x}{p^2} \frac{1-\mu}{\mu} \right)\int \rho^p + \frac{2\sigma_x}{p^2}\frac{C^2_{GN}}{\mu^2} \left( \int \rho^{\frac{p}{2}}\right)^2.
\end{equation*}
Choosing $\mu = \frac{2\sigma_x^2}{2(\sigma_x^2+\lambda^2p^2)}$, we obtain, that there exist $\alpha_p>0$ and $\beta_p>0$, such that,
\begin{equation}
    \label{est:RhoLpd/dtIntermediate}
    \frac{\dd}{\dd t} \int \rho^p \leq -\alpha_p \int \rho^p + \beta_p \left( \int \rho^{\frac{p}{2}}\right)^2.
\end{equation}

For $p = 2$, by mass preservation and the Gr\"onwall inequality from Lemma~\ref{lem:gronwall}, we obtain

\begin{equation}
    \label{eq:fLpestim2}
    \int (\rho(t))^2 \leq e^{-\alpha_2 t }\int \rho^2_0+ \frac{\beta_2}{\alpha_2}\left(1-e^{-\alpha_2 t}\right),
\end{equation}
leading to the required result.

For $p>2$, we note the following inequality from interpolating in $L^p_x\cap L^1_x$,

\begin{equation*}
    \left( \int \rho^{\frac{p}{2}}\right)^2 = \|\rho\|_{L^{\frac{p}{2}}_x}^p \leq \left(\|\rho\|_{L^p_x}^{(p-2)/(p-1)}\|\rho\|_{L^1_x}^{1/(p-1)} \right)^p \leq \|\rho\|_{L^p_x}^{p(p-2)/(p-1)}.
\end{equation*}
The last inequality holds from the conservation of the mass. Note that since $0 < \alpha \defeq \frac{p-2}{p-1} <1$, we can apply the Young product inequality with exponents, 

\begin{equation*}
    \frac{1}{\alpha} = \frac{p-1}{p-2}, \frac{1}{1-\alpha} = p-1,
\end{equation*}
so that, for any $\mu >0$,
\begin{equation*}
     \beta_p \left( \int \rho^{\frac{p}{2}}\right)^2 \leq \beta_p\|\rho\|_{L^p_x}^{p(p-2)/(p-1)} \leq \mu \int \rho^p + \frac{1}{p-1}\left(\frac{p-2}{\mu(p-1)}\right)^{p-2}\beta_p ^{p-1}.
\end{equation*}
Using the previous inequality in \eqref{est:RhoLpd/dtIntermediate}, with $\mu = \frac{\alpha_p}{2}$, we obtain, 
\begin{equation}
    \label{est:RhoLpd/dt}
    \frac{\mathrm{d}}{\mathrm{d}t} \int \rho^p \leq -\frac{\alpha_p}{2}\int \rho^p +  \frac{1}{p-1}\left(\frac{2(p-2)}{\alpha_p(p-1)}\right)^{p-2}\beta_p ^{p-1}.
\end{equation}
We conclude the case $p>2$ using the Gr\"onwall inequality from Lemma~\ref{lem:gronwall}.
\end{proof}

\subsection{Parabolic-parabolic case}
In this section, we prove that the parabolic-parabolic semi-group~\eqref{thm:FormicidaeParabolicSemiGroup} is dissipative. We begin by providing the following joint estimate on the chemotactic field $c$ and the spatial density $\rho$.

\begin{prop}
    \label{lem:parabolicLpC}
    Let $\sigma_c, \gamma, \sigma_x > 0, p>4$, $c_0 \in W^{2,p}_x$ and $f_0 \in L^p_{x,\theta}$, and let,
    \begin{equation*}
        c \in L^{p}_{t,loc}(W^{2,p}_x)\cap W^{1,p}_{t,loc}(L^p_x), f\in C_t(L^p_{x,\theta})\cap L^2_{t,loc}(H^1_{x,\theta})
    \end{equation*}
    be the unique solution of \eqref{sys:FormicidaeParabolic} with initial data $(f_0,c_0)$.
    
    Then there exists $C_p>0$ and $\alpha_p>0$ depending only on $\sigma_c, \gamma, \sigma_x $ and $p$, such that $\forall t \geq 0$,
    \begin{align}
        \label{est:parabolicCLp}
        \|c(t)\|_p + \|\rho(t)\|_p  &\leq C_p\left(e^{-\alpha_p t}(\|c_0\|_p+\|\rho_0\|_p) + 1\right),\\
        \label{est:parabolicGradCL2}
        \|\nabla_x c(t)\|_2 + \|\rho(t)\|_2  &\leq C_p\left(e^{-\alpha_p t}(\|\nabla_x c_0\|_2+\|\rho_0\|_2) + 1\right).
    \end{align}
\end{prop}

\begin{remark}
    The condition $p>4$ is only required for the well-posedness of the solution and is not crucial for this estimate. But since the domain is of finite measure, estimates \eqref{est:parabolicCLp} and \eqref{est:parabolicGradCL2} controls any $q$ such that $p \geq q \geq 1$, we leave the statement as such.
\end{remark}

\begin{proof}
    We first note that the chemotactic field $c$ is in the space $C_t(W^{1,p}_x)$, by the continuity of the semi-group. Multiplying the equation by $pc^{p-1}$, integrating by parts similarly as in Proposition~\ref{est:rhoestimLp}, and applying the following Young product inequality,

    \begin{equation*}
        ab\leq \frac{1}{\varepsilon^{p-1}p}a^p + \frac{(p-1)\varepsilon}{p}b^{p/(p-1)}, \forall \varepsilon > 0, \ \ \forall  a,b\geq 0, \ \ \forall p > 1,
    \end{equation*}
    with $\varepsilon = \gamma /2(p-1)$, we obtain, via Lions--Magenes,

    \begin{align}
        \frac{\mathrm{d}}{\mathrm{d}t}\int c^{p} & = -\frac{\sigma_c (p-1)}{4p}\int |\nabla_x (c^{p/2})|^2 -\gamma p \int c^p + p\int \rho c^{p-1},\nonumber\\
        &\leq  -\frac{\gamma p}{2} \int c^p + \frac{(2(p-1))^{p-1}}{p^{2p-1}\gamma^{p-1}}\int \rho^p, \label{est:parabolicCLpIntermediate}
    \end{align}
    
    Adding the adequate multiple $m_p>0$ of the intermediate estimate \eqref{est:RhoLpd/dt} on $\frac{d}{dt}\int \rho^p$ from Proposition~\ref{prop:rhoestimLp}, so that the last term drops, we obtain, that there exists a constant $C_p>0, \alpha_p>0$, such that,
    \begin{equation*}
        \frac{\mathrm{d}}{\mathrm{d}t}\int c^{p} + m_p\frac{\mathrm{d}}{\mathrm{d}t}\int \rho^{p} \leq -\alpha_p\left( \int c^{p} + m_p\int \rho^{p}\right) + C_p.
    \end{equation*}
    The conclusion follows from the Gr\"onwall inequality from Lemma~\eqref{lem:gronwall}.

    In order to obtain \eqref{est:parabolicGradCL2}, we note the following property,
    \begin{equation*}
        \frac{\mathrm{d}}{\mathrm{d}t} \int |\nabla_xc|^2 = - 2 \int \partial_t c \Delta_x c,
    \end{equation*}
    thus multiplying the equation by $-\frac{1}{2}(\Delta_xc)$ and integrate by part, yields the following equation,
    \begin{align*}
        \frac{\mathrm{d}}{\mathrm{d}t} \int |\nabla_x c|^2 & = -2 \sigma_c \int (\Delta_x c)^2 + 2\gamma \int c(\Delta_x c) - 2\int \rho (\Delta_x c),\\
                                        & \leq - \sigma_c \int (\Delta_x c)^2 - 2\gamma \int |\nabla_x c|^2  + \frac{1}{\sigma_c}\int \rho^2,\\
                                        & \leq - 2\gamma \int |\nabla_x c|^2  + \frac{1}{\sigma_c}\int \rho^2.
    \end{align*}
    The result follows similarly as the previous $L^p$ estimate, adding the adequate multiple of the estimate on $\frac{d}{dt} \int \rho^2$ from Proposition~\ref{prop:rhoestimLp} to drop the last positive term involving $\int \rho^2$, and conclude with the Gr\"onwall inequality from Lemma~\eqref{lem:gronwall}.
\end{proof}

We now prove the $L^2_{x,\theta} \times H^2_x$ absorption.

\begin{prop}
    \label{prop:parabolicH2L2Absorb}
    Let $\sigma_x, \sigma_\theta, \sigma_c , \gamma >0$, then there exists $C_1>0$ and $C_2 >0$, such that for any solution $(f, c)$ to the system \eqref{sys:FormicidaeParabolic} in the space
    \begin{equation*}
        (L^2_t(H^1_{x,\theta})\cap C_t (L^6_{x,\theta})) \times (L^{6}_{t,loc}(W^{2,6}_x) \cap W^{1,6}_{t,loc}(L^6_x)),    
    \end{equation*} 
    there exists $t_1>0$ depending on $\|c_0\|_{L^6_x}$, $\|\nabla^2_x c_0\|_{L^2_x}$, $\|\rho_0\|_{L^6_x}$ and $\|f_0\|_{L^2_x}$, such that,
    \begin{equation*}
        \forall t \geq t_1(\|c_0\|_{L^6_x}, \|\nabla^2_x c_0\|_{L^2_x},\|f\|_{L^2_{x,\theta}}, \|\rho_0\|_{L^6_x}) >0,
    \end{equation*}
    the following estimates hold,
    \begin{equation}
        \label{est:parabolicL2fL2nabla2c}
        \|\f(t)\|_{L^2_{x,\theta}}  + \|\nabla^2_x c(t)\|_{L^2_x} \leq C_1,
    \end{equation}

    \begin{equation}
        \label{est:parabolicL2nabalfL2nabla3c}
        \int_t^{t+1}\int |\nabalaxtheta f|^2 \mathrm{d}\theta \mathrm{d}x\mathrm{d}s +\int_t^{t+1} \int |\nabla^3_x c|^2 \mathrm{d}x\mathrm{d}s\leq C_2.
    \end{equation}

    The constants $C_1$ and $C_2$ can be expressed as, $C_1 = \chi^{p_1}\Tilde{C}_1, C_2 = \chi^{p_1}\Tilde{C}_2$, for some $p_1 > 1$ independent of the parameters and $\Tilde{C}_1, \Tilde{C}_2$ depending on all the parameters but $\chi$.
\end{prop}

\begin{proof}
    We start with the following energy estimate on the phase space density $f$, where we apply the Lions--Magenes lemma.
    \begin{align}
        \frac{\mathrm{d}}{\mathrm{d}t}\int \frac{f^2}{2} & = -\sigma_x \int |\nabla_x f|^2 - \sigma_\theta \int (\partial_\theta f)^2 - \chi \int f \partial_\theta (B[c]f) - \lambda \int f \nabla_x \cdot ( vf),\nonumber\\
        & = -\sigma_x \int |\nabla_x f|^2 - \sigma_\theta \int (\partial_\theta f)^2 + \chi \int B[c] \partial_\theta (f^2) + \lambda \int \frac{v}{2} \nabla_x f^2,\nonumber\\
        \label{est:parabolicfL2Intermediate}
        & = -\sigma_x \int |\nabla_x f|^2 - \sigma_\theta \int (\partial_\theta f)^2 - \frac{\chi}{2} \int f^2 \partial_\theta B[c].
    \end{align}
    Using that $\rho \in L^2_t(H^1_x)$, we obtain the following equations on the spatial derivatives of $c$,
    \begin{equation}
        \label{est:parabolicnabla2Cintermidiatexi}
        \partial_t \partial_{x_i} c = \sigma_c \Delta_x \partial_{x_i} c - \gamma \partial_{x_i} c + \partial_{x_i} \rho, \text{ for } i = 1,2.
    \end{equation}
    This regularity implies that $c\in L^2_t(H^3_x), \partial_t c \in L^2_{t,x}$, and from the Aubin--Lions lemma, we have that $c\in C_t(H^2_x)$. 
    
    Multiplying both equations~\eqref{est:parabolicnabla2Cintermidiatexi} by the corresponding $-2\Delta_x \partial_{x_i} c$, and integrating, we obtain,
    \begin{align}
        \frac{\mathrm{d}}{\mathrm{d}t}\int |\nabla_x \partial_{x_i} c |^2 & = -2\sigma_c \int (\Delta_x \partial_{x_i} c)^2 + \gamma2 \int \partial_{x_i} c \Delta_x \partial_{x_i} c - 2\int \partial_{x_i} \rho (\Delta_x \partial_{x_i} c),\nonumber \\
        \label{est:parabolicnabla2Cintermidiate}
        & \leq -\sigma_c \int (\Delta_x \partial_{x_i} c)^2 - 2\gamma \int | \nabla_x \partial_{x_i} c |^2 + \frac{1}{\sigma_c}\int |\partial_{x_i} \rho|^2.
    \end{align}
    We then note by integration by parts that,
    \begin{align*}
        \int (\Delta_x \partial_{x_i} c)^2 & = \sum_{1\leq j,k \leq 2} \int (\partial_{x_j}^2 \partial_{x_i} c )(\partial_{x_k}^2 \partial_{x_i} c) = \sum_{1\leq j,k \leq 2} \int (\partial_{x_k}\partial_{x_j} \partial_{x_i} c )(\partial_{x_k}\partial_{x_j}\partial_{x_i} c )\\
        &= \int |\nabla^2_x \partial_{x_i} c|^2.
    \end{align*}
    Using this in \eqref{est:parabolicnabla2Cintermidiate} and suming for both $i = 1,2$, we obtain,
    \begin{equation}
        \label{est:parabolicnabla2C}
         \frac{\mathrm{d}}{\mathrm{d}t} \int |\nabla^2_x c |^2 \leq -\sigma_c \int |\nabla^3_x c|^2 - 2\gamma \int |\nabla^2_x c|^2 + \frac{1}{\sigma_c} \int |\nabla_x \rho|^2.
    \end{equation}
    We then treat separately the third term in \eqref{est:parabolicfL2Intermediate}, with H\"older inequality,
    \begin{equation}
        \label{est:parabolicBf2Bound1}
        \int |\partial_\theta B[c] f^2| \leq \|f^2\|_{L^3_x(L^1_\theta)}\|\partial_\theta B[c]\|_{L^{3/2}_x(L^\infty_\theta)}.
    \end{equation}
    Noting the following identity, 
    \begin{equation*}
        \|f^2\|_{L^3_x(L^1_\theta)} = \|f\|_{L^6_x(L^2_\theta)}^2,
    \end{equation*}
    interpolating $f$ in $L^6_x(L^1_\theta) \cap L^6_{x,\theta}$, and the Sobolev embedding from \cite[Corollary 1.2]{benyi2013sobolev},this implies that there exists $C_1>0$, such that
    \begin{equation}
        \label{est:parabolicL6L1f}
        \|f^2\|_{L^3_x(L^1_\theta)} \leq \left(\|f\|^{2/5}_{L^6_x(L^1_\theta)}\|f\|^{3/5}_{L^6_{x,\theta}}\right)^2 \leq C_1 \|\rho\|_{L^6_x}^{4/5}\left(\|\nabalaxtheta f \|^2_2 + 1\right)^{3/5}.
    \end{equation}
    We then treat the scalar field as follows, from its definition the following inequality holds,
    \begin{equation*}
        \|\partial_\theta B[c]\|_{L^{3/2}_x(L^\infty_\theta)} \leq C_\tau \left(\|\nabla_x c\|_{L^{3/2}_x} + \|\nabla^2_x c\|_{L^{3/2}_x}\right),
    \end{equation*}

    for some $C_\tau>0$ depending on $\tau$.

    
    We then apply the Gagliardo-Nirenberg exponents in dimension 2, with the following exponents, 
    \begin{align*}
        j = 2, m = 3, p = 3/2, r = 2, q = 6, s = 1, \alpha = 3/7,\\
        j = 1, m = 3, p = 3/2, r = 2,  q = 1, s = 1, \alpha = 1/6,
    \end{align*}
    and obtain, for some constant $C_{GN}>0$,
     \begin{align*}
        \|\nabla^2_x c\|_{L^{3/2}_x} \leq C_{GN}\left( \|\nabla^3_x c\|^{3/7}_{L^2_x}\| c \|^{4/7}_{L^6_x} + \|c\|_{L^1_x}\right),\\
        \|\nabla_x c\|_{L^{3/2}_x} \leq C_{GN}\left(\|\nabla^3_x c\|^{1/6}_{L^2_x}\| c \|^{5/6}_{L^1_x} + \|c\|_{L^1_x}\right).
     \end{align*}

     This allows to conclude, using \eqref{est:parabolicL6L1f} and the absorption result in $L^6_x\times L^6_x$ on $(c,\rho)$ from Proposition~\ref{lem:parabolicLpC}, that there exists $C_2>0$, such that for initial conditions $(c_0, f_0)\in (L^6_x,L^2_{x,\theta}\cap L^6_x(L^1_\theta)_+)$, there exists $t_1(\|c_0\|_6, \|\rho_0\|_6)>0$ depending on $\|c_0\|_6$ and $ \|\rho_0\|_6$, such for any $t \geq t_1 >0$,
     \begin{align}
        \label{est:parabolicL3/2LinftyfBc}
         \|\partial_\theta B[c(t)]\|_{L^{3/2}_x(L^\infty_\theta)} & \leq C_2\left(\|\nabla^3_x c(t)\|^{3/7}_{L^2_x} + 1\right),\\
         \label{est:parabolicL3/2f2Bound}
         \|f^2(t)\|_{L^3_x(L^1_\theta)} & \leq C_2 \left(\|\nabalaxtheta f(t) \|^2_2 + 1\right)^{3/5}.
     \end{align}

     Plugging \eqref{est:parabolicL3/2LinftyfBc} and \eqref{est:parabolicL3/2f2Bound}  in \eqref{est:parabolicBf2Bound1} leads to the following estimate for any $t \geq t_1>0$,
     \begin{equation}
         \label{est:parabolicBf2Bound2}
         \int |\partial_\theta B[c(t)] f^2(t)| \leq C\left(\|\nabla^3_x c(t)\|^2_{L^2_x} + 1\right)^{3/14}\left(\|\nabalaxtheta f(t) \|^2_2 + 1\right)^{3/5}.
     \end{equation}

    Recalling the following Young product inequality for three terms.

    Suppose that $0<\alpha, \beta < 1$, such that $\alpha + \beta < 1$, then,
     \begin{equation}
        \label{thm:YoungProductThreeTerms}
         a^\alpha b^\beta c \leq \varepsilon a + \eta b + \left( \frac{c}{(\varepsilon/\alpha)^\alpha(\eta/\beta)^\beta} \right)^{\frac{1}{1-\alpha -\beta}} \ \ \ \forall a,b,c\geq 0\text{ and } \varepsilon,\eta >0.
     \end{equation}

     Adding estimates \eqref{est:parabolicfL2Intermediate} and $(\sigma_x \sigma_c/ 2)$-times estimates  \eqref{est:parabolicnabla2C}, and using that,
     \begin{equation*}
         \int |\nabla_x \rho|^2\mathrm{d}x \leq \int |\nabla_x f|^2\mathrm{d}\theta \mathrm{d}x,
     \end{equation*}
     we obtain,
     \begin{align*}
          \frac{\mathrm{d}}{\mathrm{d}t} \int f^2 + \frac{\sigma_x \sigma_c}{2} \frac{\mathrm{d}}{\mathrm{d}t}  \int |\nabla^2_x c|^2 &\leq  -\frac{\sigma_x\wedge\sigma_c}{2} \int |\nabalaxtheta f|^2 -\frac{\sigma_x \sigma_c^2}{2} \int |\nabla^3_x c|^2 \\
          &\hspace{2em} - 2\gamma \frac{\sigma_x \sigma_c}{2} \int |\nabla^2_x c|^2 + \frac{\chi}{2} \int |f^2 \partial_\theta B[c]|,
     \end{align*}

     Bounding the last term by \eqref{est:parabolicBf2Bound2} and applying Young product \eqref{thm:YoungProductThreeTerms}, with 
     \begin{equation*}
         \alpha = 3/14, \beta = 3/5, \varepsilon = \frac{\sigma_x \sigma_c^2}{4}\text{ and } \eta = \frac{\sigma_x\wedge\sigma_c}{4},
     \end{equation*}
     
     this implies that there exists $C>0$ depending on $\chi, \tau, \sigma_x, \sigma_c, \sigma_\theta$, such that for any $t\geq t_1>0$,
     \begin{align}
          \frac{\mathrm{d}}{\mathrm{d}t} \int f^2 + \frac{\sigma_x \sigma_\theta}{2} \frac{\mathrm{d}}{\mathrm{d}t}  \int |\nabla^2_x c|^2 &\leq  -\frac{\sigma_x\wedge\sigma_c}{4} \int |\nabalaxtheta f|^2 -\frac{\sigma_x \sigma_c^2}{4} \int |\nabla^3_x c|^2\nonumber\\
          \label{est:parabolicfL2Nabla2cL2JointIntermediate}
          &\hspace{3em}- \gamma \sigma_x \sigma_c \int |\nabla^2_x c|^2 + C.
     \end{align}
     Then applying Poincaré inequality and the conservation of the mass on $f$,
     \begin{equation*}
         \|f\|^2_{L^2_{x,\theta}} \leq 2\left(C^2_P \|\nabalaxtheta f\|^2_{L^2_{x,\theta}} + \frac{1}{2\pi}\right),
     \end{equation*}
     where $C_{\mathcal{P}}>0$ is the Poincaré constant on the three-dimensional torus $\Torus_1^2\times \Torus_{2\pi}$.
     Using this in estimate~\eqref{est:parabolicfL2Nabla2cL2JointIntermediate}, yields that there exists a new constant $C>0$ such that, the following holds, on $t\geq t_1 >0$,
     \begin{equation}
        \label{est:parabolicfL2Nabla2cL2JointIntermediate2}
          \frac{\mathrm{d}}{\mathrm{d}t} \int f^2 + \frac{\sigma_x \sigma_c}{2} \frac{\mathrm{d}}{\mathrm{d}t}  \int |\nabla^2_x c|^2 \leq -\frac{\sigma_x\wedge\sigma_c}{8C_{\mathcal{P}}^2} \int |f|^2 - \gamma \sigma_x \sigma_c \int |\nabla^2_x c|^2 + C,
     \end{equation}
     Applying the Gr\"onwall inequality from Lemma~\eqref{lem:gronwall}, we obtain that there exists a constant $C>0$, depending only on the parameters, $\chi, \tau,\sigma_x,\sigma_\theta,\lambda, \sigma_c$, such that for any initial condition $(f_0, c_0)$, there exists $t_0(\|c_0\|_{L^6_x},\|\rho_0\|_{L^6_x})>0$, such that $\forall t \geq t_0> 0$,
     \begin{equation*}
         \|\f(t)\|^2_{L^2_{x,\theta}}  + \|\nabla^2_x c(t)\|^2_{L^2_x} \leq C\left( e^{-\alpha (t-t_0)/2}\left(\|\f(t_0)\|^2_{L^2_{x,\theta}}  + \|\nabla^2_x c(t_0)\|^2_{L^2_x}\right) + 1\right),
     \end{equation*}
     where $\alpha > 0$ is given by,
     \begin{equation*}
         \alpha = \frac{\sigma_x\wedge\sigma_c}{8C_{\mathcal{P}}^2} \wedge 2\gamma.
     \end{equation*}

     The above computations, also imply growth estimates on the finite time interval $[0,t_0]$, for $\nabla^2_x c$ and $f$, so that the norm $\|\f(t_0)\|^2_{L^2_{x,\theta}}  + \|\nabla^2_x c(t_0)\|^2_{L^2_x}$ only depends on the norms of the initial condition. We conclude the desired result by choosing $t_1\geq t_0$, sufficiently large, depending on $\|\rho_0\|_{L^2_x}, \|\nabla^2_x c_0\|_{L^2_x},\|f_0\|_{L^2_{x,\theta}}$, to obtain estimate~\eqref{est:parabolicL2fL2nabla2c}.
     
     Finally, integrating \eqref{est:parabolicfL2Nabla2cL2JointIntermediate2} from $t$ to $t+1$, for $t > t_1$,   using the bound \eqref{est:parabolicL2fL2nabla2c}, we obtain \eqref{est:parabolicL2nabalfL2nabla3c}.
\end{proof}

We now prove the $H^1_{x,\theta} \times H^3_x$ absorption.
\begin{thm}
    \label{thm:AbsorbH3H1Parabolic}
    Let $\sigma_x, \sigma_\theta, \sigma_c , \gamma >0$, then there exists $C_3>0$ depending on the parameters, such that for any solution $(f, c)$ to the system \eqref{sys:FormicidaeParabolic} in the space $(L^2_{t,loc}(H^1_{x,\theta})\cap C_t (L^6_{x,\theta})) \times (L^{6}_{t,loc}(W^{2,6}_x) \cap W^{1,6}_{t,loc}(L^6_x)$, there exists $t_2>0$ depending on $\|c_0\|_{L^6_x}$, $\|\nabla^2_x c_0\|_{L^2_x}$, $\|\rho_0\|_{L^6_x}$ and $\|f_0\|_{L^2_x}$, such that,
    \begin{equation*}
        c\in C([t_2, \infty), H^3_x), f\in C([t_2, \infty), H^1_{x,\theta}),
    \end{equation*}
    and for any $t\geq t_2(\|c_0\|_{L^6_x}, \|\nabla^2_x c_0\|_{L^2_x}, \|\rho_0\|_{L^6_x},\|f_0\|_{L^2_{x,\theta}})>0$, the following estimate holds,
    \begin{equation*}
        \|c(t)\|_{H^3_x}+\|f(t)\|_{H^1_{x,\theta}} \leq C_3.
    \end{equation*}
    Furthermore, the following estimate holds,
    \begin{equation}
        \label{est:L2H4L2H2}
        \int_t^{t+1}\int |\nabla^4_x c|^2 \mathrm{d}x \mathrm{d}s + \int^{t+1}_t \int |\nabalaxtheta^2 f|^2 \mathrm{d}\theta \mathrm{d}x \mathrm{d}s \leq C_4, \text{ for all } t \geq t_2.
    \end{equation}
    Finally, the constants $C_3$, and $C_4$ can be expressed as, $C_3 = (1+\chi^{p_2})\Tilde{C}_3, C_4 = (1+\chi^{p_2})\Tilde{C}_4$, where $\Tilde{C}_3$ and $\Tilde{C}_4$ depend on all the parameters but $\chi$, and $p_2>1$ is independent of the parameters.
\end{thm}

\begin{proof}
    We start by considering a smooth initial condition $(c_0,f_0)$, the propagation of regularity \cite[Theorem 3.6]{bertucci2024curvature}, ensures that the associated solution is smooth and the following computations hold.
    In the following, we use the following differential operator notations, 
    \newcommand{\Tildenablaxtheta}{\widetilde{\nabla}_{\xi}}
    \begin{equation*}
        \mathcal{L} = \sigma_x \Delta_x + \sigma_\theta 
        \partial_{\theta\theta}\text {, and }\Tildenablaxtheta = \begin{pmatrix}
            \sqrt{\sigma_x}\partial_{x_1} \\
            \sqrt{\sigma_x}\partial_{x_2} \\
            \sqrt{\sigma_\theta}\partial_{\theta}
        \end{pmatrix}.
    \end{equation*}
    Multiplying the equation for $f$ by $-\mathcal{L}f$ and integrating yields,
    \begin{align}
        \label{est:parabolicH1absorbfFirst}
        \frac{\dd}{\dd t} \int \frac{|\Tildenablaxtheta\f|^2}{2} &= -\int (\mathcal{L}f)^2 + \chi \int (\mathcal{L}f)\partial_\theta(B[c]f) + \lambda \int (\mathcal{L}f) v\cdot \nabla_x f,\nonumber\\
        &\leq -\frac{1}{2}\int (\mathcal{L}f)^2 + \chi^2 \int (\partial_\theta(B[c]f))^2 + \lambda^2 \int |\nabla_x f|^2.
    \end{align}

    Using integration by part, we obtain for the first term on the right and side,
    \begin{align}
        -\int (\mathcal{L}\f)^2 &= -\int \sigma_x^2(\Delta_x \f)^2 + 2\sigma_\theta \sigma \Delta_x \f \partial_{\thth} \f + \sigma_\theta^2(\partial_\thth \f)^2,\nonumber\\
        &= -\int \sigma_x^2(\Delta_x \f)^2 + 2\sigma_\theta \sigma (\nabla_x\partial_\theta\f)^2 + \sigma_\theta^2(\partial_{\thth} \f)^2,\nonumber\\
        &\leq -\sigma_x^2\wedge \sigma_\theta^2 \int |\nabalaxtheta^2 \f|^2.\label{est:parabolicH2Deltafsquared}
    \end{align}

    We then treat the second term as follows,
    \begin{align}
        \int (\partial_\theta(B[c]\f))^2 &= \int (\partial_\theta B)^2 \f^2 + 2\partial_\theta B[c] \f B[c]\partial_\theta \f + B^2(\partial_\theta \f)^2,\nonumber\\
        &= \int (\partial_\theta B)^2 \f^2 + \partial_\theta B[c] B[c] \partial_\theta \f^2 - 2\partial_\theta B[c] B[c] \partial_\theta \f \f - B^2f\partial_{\thth} f,\nonumber\\
        &= \int (\partial_\theta B)^2 \f^2 + \partial_\theta B[c] B[c] \partial_\theta \f^2 - \partial_\theta B[c] B[c] \partial_\theta \f^2 - B^2f\partial_{\thth} f,\nonumber\\
        &= \int (\partial_\theta B)^2 \f^2 - B^2f\partial_{\thth} f,\nonumber\\
        &\leq \int \left((\partial_\theta B)^2+ \frac{1}{\mu} B^4\right)f^2 + \mu \int( \partial_{\thth} f)^2,\label{est:parabolicH2ProductFirst}
    \end{align}

    for any $\mu >0$. We then define $\Phi_\mu[c]$ as,
    \begin{equation*}
        \Phi_\mu[c] = \left((\partial_\theta B)^2+ \frac{1}{\mu} B^4\right).
    \end{equation*}

    Plugging estimates \eqref{est:parabolicH2Deltafsquared} and \eqref{est:parabolicH2ProductFirst} with $\mu = \frac{\sigma_x^2\wedge \sigma_\theta^2}{4\chi^2}$, in \eqref{est:parabolicH1absorbfFirst}, we obtain,

    \begin{equation}
        \label{est:parabolicH1absorbfSecond}
        \frac{\dd}{\dd t} \int \frac{|\Tildenablaxtheta\f|^2}{2} \leq -\frac{\sigma_x^2\wedge \sigma_\theta^2}{4} \int |\nabalaxtheta^2 \f|^2  + \lambda^2 \int |\nabla_x f|^2 + \chi^2 \int \Phi[c]f^2,
    \end{equation}
    where we dropped the dependence in $\mu$ of $\Phi[c]$.

    Similarly as in the proof of Proposition~\ref{prop:parabolicH2L2Absorb}, differentiating two times the equation for $c$, with respect to $x_i,x_j$, multiplying by $\Delta_x c_{x_ix_j}$, using Young product inequality on the term involving $\Delta_x c_{x_ix_j} \rho_{x_ix_j}$, and summing over $i,j=1,2$, we obtain,

    \begin{equation}
        \label{est:parabolicH3absorbcFirst}
        \frac{\dd}{\dd t}\int |\nabla^3_x c|^2 \leq - \sigma_c \int |\nabla^4_x c|^2 - 2\gamma \int |\nabla^3_x c|^2 + \frac{1}{\sigma_c} \int |\nabla_x^2 f|^2.
    \end{equation}

    Adding the adequate multiples of estimates \eqref{est:parabolicH3absorbcFirst} and \eqref{est:parabolicH1absorbfSecond} we obtain,

    \begin{align}
          \frac{\dd}{\dd t}\left(\frac{\sigma_c\sigma_x^2\wedge\sigma_\theta^2}{8}\int |\nabla^3_x c|^2 + \int |\Tildenablaxtheta\f|^2\right) &\leq - \frac{\sigma_c^2\sigma_x^2\wedge\sigma_\theta^2}{8} \int |\nabla^4_x c|^2 -\frac{\sigma_x^2\wedge \sigma_\theta^2}{4} \int |\nabalaxtheta^2 \f|^2  \nonumber\\ 
          \label{est:nabla3cnabla1fTotal}
          & \hspace{2.em}+ 2\lambda^2 \int |\nabla_x f|^2 + 2\chi^2 \int \Phi[c]f^2.
    \end{align}

    We then treat the last product with H\"older inequality,
    \begin{equation}
        \label{est:PhifsqHolder}
        \int \Phi[c]f^2 \leq \|\Phi[c]\|_{L^{\frac{3}{2}}_x(L^\infty_\theta)}\|f^2\|_{L^3_x(L^1_\theta)},
    \end{equation}
    Similarly as in the proof of Proposition~\ref{prop:parabolicH2L2Absorb}, interpolating in $L^6_{x,\theta}\cap L^6_x(L^1_\theta)$ the phase space density, we obtain,
    \begin{equation}
        \label{est:f2HolderEstim}
        \|f^2\|_{L^3_x(L^1_\theta)} = \|f\|_{L^6_x(L^2_\theta)}^2\leq \|f\|_{L^6_x(L^1_\theta)}^{\frac{4}{5}} \|f\|_{L^6_{x,\theta}}^{\frac{6}{5}}.
    \end{equation}
    Then using Gagliardo-Nirenberg inequality from Proposition~\ref{thm:GagliardoNirenbeg} in dimension $3$, with parameters,
    \begin{equation*}
        j = 0, m = 2, p = 6, r = 2, q = 2, s = 1, \alpha = 1/2,
    \end{equation*}

    There exists a constant $C>0$ independent of $f$, such that,
    \begin{equation}
        \label{est:fL6GNS}
        \|f\|_{L^6_{x,\theta}} \leq C\left(\|\nabalaxtheta^2 f \|^{1/2}_{L^2_{x,\theta}}\|f\|^{1/2}_{L^2_{x,\theta}}+\|f\|_{L^1_{x,\theta}}\right).
    \end{equation}
    On the other end, there exists a constant $C_{\tau, \sigma_x, \sigma_\theta}>0$ depending on $\tau, \sigma_x, \sigma_\theta$, such that,
    \begin{align}
        \|\Phi[c]\|_{L^{3/2}_x(L^\infty_\theta)} \leq& C_{\tau, \sigma_x, \sigma_\theta}\left(1 + \||\nabla_xc|^4\|_{L^{3/2}_x} + \||\nabla^2_xc|^4\|_{L^{3/2}_x}\right),\nonumber \\
        \label{est:PhiL3voer2HesscL6}
        \leq& C_{\tau, \sigma_x, \sigma_\theta}\left(1 + \|\nabla_xc\|^4_{L^{6}_x} + \|\nabla^2_xc\|^4_{L^{6}_x}\right).        
    \end{align}
    Applying Gagliardo-Nirenberg inequality from Lemma~\ref{thm:GagliardoNirenbeg} in dimension $2$, on each of the gradient of $c$ with the parameters,
    \begin{equation*}
        j = 0, m = 1, p = 6, r = 2, q = 2, s = 2, \alpha = 2/3.
    \end{equation*}
    We proceed similarly for each component of the Hessian of $c$ with the parameters,
    \begin{equation*}
        j = 0, m = 2, p = 6, r = 2, q = 2, s = 2, \alpha = 1/3,
    \end{equation*}
    and summing over all the components, we obtain, that there exists $C>0$ independent of $c$, such that,
    \begin{align}
        \label{est:nablacL6GNS}
        \|\nabla_x c\|_{L^6_x} \leq C\left(\|\nabla^2_x c\|^{2/3}_{L^2_x}\|\nabla_x c\|^{1/3}_{L^2_x} + \|\nabla_x c\|_{L^2_x}\right),\\
        \label{est:HesscL6GNS}
        \|\nabla^2_x c\|_{L^6_x} \leq C\left(\|\nabla^4_x c\|^{1/3}_{L^2_x}\|\nabla^2_x c\|^{2/3}_{L^2_x} + \|\nabla^2_x c\|_{L^2_x}\right).
    \end{align}

    Since $\sigma_c, \sigma_\theta, \sigma_x >0$, from Proposition~\ref{prop:parabolicH2L2Absorb} and Proposition~\ref{lem:parabolicLpC}, we know that there exists $C_1>0$, only depending on the parameters, such that for any initial condition, there exists $t_1>0$ depending only on $\|c_0\|_{L^6_x}, \|\nabla^2_xc_0\|_{L^2_x},\|f\|_{L^2_{x,\theta}}, \|\rho_0\|_{L^6_x}$, such that for any $t \geq t_1$, the following norms are bounded by $C_1$,

    \begin{equation*}
        \|\nabla_x c(t)\|_{L^2_x},\|\nabla^2_x c(t)\|_{L^2_x},\|f(t)\|_{L^2_{x,\theta}},\|f(t)\|_{L^6_{x} (L^1_\theta)} \leq C_1.
    \end{equation*}

    Using the estimates \eqref{est:fL6GNS},\eqref{est:PhiL3voer2HesscL6}, \eqref{est:nablacL6GNS} and \eqref{est:HesscL6GNS}, in \eqref{est:f2HolderEstim}, we obtain that there exists $\widetilde{C}_0>0$ independent of $c$ and $f$, such that, for any $t \geq t_1 >0$,

    \begin{equation}
        \label{est:Phif2prod}
        \int \Phi[c(t)]f^2(t) \leq \widetilde{C}_0\left(1+\|\nabla^4_xc(t)\|^{4/3}_{L^2_x}\right)\left(1+\|\nabalaxtheta^2 f(t)\|^{3/5}_{L^2_x}\right).
    \end{equation}

    Since $2/3+ 3/10 < 1$, we can apply the above Young product inequality, as stated in \eqref{thm:YoungProductThreeTerms}, with $\varepsilon = \sigma_c^2(\sigma_x^2\wedge\sigma_\theta^2)/8$, and $\eta = \sigma_x^2\wedge\sigma_\theta^2/4$, on the right hand side of \eqref{est:Phif2prod}, and inject it in estimate \eqref{est:nabla3cnabla1fTotal}, leading to, after dropping the negative terms,

    \begin{equation}
        \label{est:}
          \frac{\dd}{\dd t}\left(\frac{\sigma_c\sigma_x^2\wedge\sigma_\theta^2}{2}\int |\nabla^3_x c(t)|^2 + \int |\Tildenablaxtheta\f(t)|^2\right) \leq 2\lambda^2 \int |\nabla_x f(t)|^2 + \widetilde{C}_1, \forall t \geq t_1,
    \end{equation}

    for some $\widetilde{C}_1>0$. Integrating this equation from $t$ to $t+1$, and using estimate \eqref{est:parabolicL2nabalfL2nabla3c} of Proposition~\ref{prop:parabolicH2L2Absorb}, that is,

    \begin{equation*}
        \int_t^{t+1} \int |\nabalaxtheta f|^2 \dd s \leq C_2, \ \ \forall t \geq t_1,
    \end{equation*}

    this implies, that there exists $C_3>0$, such that for any initial condition, there exists $t_2>0$, only depending on the norms $\|c_0\|_{L^6_x}, \|\nabla^2_x c_0\|_{L^2_x}, \|\rho_0\|_{L^6_x},\|f_0\|_{L^2_x}$, such that,
    \begin{equation*}
         \forall t \geq t_2(\|c_0\|_{L^6_x}, \|\nabla^2_x c_0\|_{L^2_x}, \|\rho_0\|_{L^6_x},\|f_0\|_{L^2_x}) >0,
    \end{equation*}
     the following estimate holds,
    \begin{equation*}
        \int |\nabla^3_x c(t)|^2 + \int |\nabalaxtheta f(t)|^2 \leq C_3.
    \end{equation*}
    
    Finally, since $t_2$ does not depend on the higher regularity of the initial condition, we can conclude the proof as follows. Let $(c^n_0,f^n_0)$ be a sequence of smooth initial conditions converging to $(c_0,f_0)\in W^{2,6}_x\times L^6_x$, with norms controlled as,
    \begin{equation*}
        \|c^n_0\|_{H^2_x} \leq \|c_0\|_{H^2_x},\\
        \|f^n_0\|_{L^6_{x,\theta}} \leq \|f_0\|_{L^6_{x,\theta}}.
    \end{equation*}
    By a weak compactness argument, the sequence of smooth solutions converges weakly to the unique weak solution associated with initial condition $(c_0,f_0)$. And we conclude from Lemma~\ref{lem:weakL2convergenceLinfty}, the required result.
    The last estimate \eqref{est:L2H4L2H2} follows similarly as in the proof of Theorem~\ref{prop:parabolicH2L2Absorb}, by integrating between $t$ and $t+1$ for $t\geq t_2$, the estimate \eqref{est:nabla3cnabla1fTotal} after absorbing the non-linear term.
\end{proof}

\begin{cor}
    \label{cor:H4CParabolic}
    Under the same hypothesis as Theorem~\ref{thm:AbsorbH3H1Parabolic} above.
    There exists $C_5>0$, such that for any initial condition, there exists $t_2>0$ depending on $\|c_0\|_{L^6_x}$, $\|\nabla^2_x c_0\|_{L^2_x}$, $\|\rho_0\|_{L^6_x}$ and $\|f_0\|_{L^2_x}$, such that,
    \begin{equation*}
        c\in C([t_2, \infty), H^4_x)
    \end{equation*}
    and for any $t\geq t_2(\|c_0\|_{L^6_x}, \|\nabla^2_x c_0\|_{L^2_x}, \|\rho_0\|_{L^6_x},\|f_0\|_{L^2_x})>0$, the following estimate hold,
    \begin{equation*}
        \|\rho(t)\|_{H^2_x} ,\|c(t)\|_{H^4_x} \leq C_5.
    \end{equation*}
    Finally, $C_5$ can be expressed as, $C_5 = (1+\chi^{p_3})$, where $\Tilde{C}_5$ depends on all the parameters but $\chi$, and $p_3>1$ is independent of the parameters.
\end{cor}

\begin{proof}
    We apply the same strategy as above. We then obtain for a regular solution, the estimates,
    \begin{align*}
        \frac{\dd}{\dd t} \int |\nabla^4_x c|^2 &\leq -\sigma_c \int |\nabla^5_x c|^2 -2\gamma  \int |\nabla^4_x c|^2 + \frac{1}{\sigma_c} \int |\nabla^3_x \rho|^2,\\
        \frac{\dd}{\dd t} \int |\nabla^2_x \rho|^2 &\leq -\sigma_x \int |\nabla^3_x \rho|^2 + \frac{\lambda^2}{\sigma_x}\int \int |\nabla_x^2f|^2.
    \end{align*}

    Combining both estimates we obtain,
    \begin{equation*}
        \frac{\dd}{\dd t}\left( \frac{1}{\sigma_c\sigma_x}\int |\nabla^2_x \rho(t)|^2 +  \int |\nabla^4_x c(t)|^2 \right) \leq \frac{\lambda^2}{\sigma_c \sigma_x^2} \int \int |\nabla^2_x f(t)|.
    \end{equation*}
    We then integrate between $t$ and $t+1$ for $t>t_2$ and we bound the right-hand side with the estimate~\eqref{est:L2H4L2H2} of Theorem~\ref{thm:AbsorbH3H1Parabolic}. The claim follows for non-necessarily smooth initial datum using Lemma~\ref{lem:weakL2convergenceLinfty}.
\end{proof}

\subsection{Parabolic-Elliptic Case}
In this section, we prove that the elliptic semigroup~\eqref{thm:FormicidaeEllipticSemiGroup} is dissipative. For broader applicability, since we believe that this strategy can be applied with minor modifications to other active matter models, the estimates in Proposition~\ref{prop:elliL2f} and Proposition~\ref{prop:fabsorbedH1} are proved for general solutions to the active matter equation~\eqref{eq:ActiveMaterEquation} given scalar field $B$. Then in Theorem~\ref{thm:AbsorbH3H1Elliptic}, we use Proposition~\ref{prop:rhoestimLp} to obtain the required estimates on the scalar field $B$ and conclude the absorption of the spatial density $f$ in $H^1_{x,\theta}$.

We now prove the $L^2_{x,\theta}$ absorption.

\begin{prop}
    \label{prop:elliL2f}
    Let $B$, be a scalar field such that $B,\partial_\theta B \in L^\infty_t(L^6_x(L^\infty_\theta))$, $\sigma_x >0$ and $\sigma_\theta>0$. Then, there exists a constant $C_1>0$, depending on $\lambda, \sigma_x, \sigma_\theta, \|\partial_\theta B\|_{L^\infty_t(L^6_x(L^\infty_\theta))}$, such that for any solution $f$ in $C_t(L^6_{x,\theta})\cap L^2_{loc,t}(H^1_{x,\theta})$ of the Fokker-Planck equation~\eqref{eq:ActiveMaterEquation} associated to $B$, there exists $t_0>0$, depending on $\|\rho_0\|_{L^6_x},\|f\|_{L^2_{x,\theta}}$, such that,
    \begin{equation*}
        \forall t \geq t_0(\|\rho_0\|_{L^6_x},\|f\|_{L^2_{x,\theta}}),
    \end{equation*}
    the following estimates holds,
    \begin{equation}
        \label{est:absorbL2f}
        \|f(t)\|_{L^2_{x,\theta}} \leq C_1,
    \end{equation}
    and,
    \begin{equation}
        \int^{t+1}_t\int |\nabalaxtheta f|^2 \dd s \leq C_1.
        \label{est:absorbintTnablaf}
    \end{equation}
\end{prop}

\begin{proof}
    We start by recalling the same computations as in the proof of Proposition~\ref{prop:parabolicH2L2Absorb}, that follows from Lions-Magenes lemma,
    \begin{align*}
        \frac{\dd}{\dd t}\int \frac{f^2}{2} & = -\sigma_x \int |\nabla_x f|^2- \sigma_\theta \int |\partial_\theta f|^2 + \chi \int \partial_\theta f Bf + \lambda \int v \cdot \nabla_x f f,\\
        & = -\sigma_x \int |\nabla_x f|^2- \sigma_\theta \int |\partial_\theta f|^2 + \frac{\chi}{2} \int B \partial_\theta (f^2) + \frac{\lambda}{2} \int v\cdot \nabla_x (f^2),\\
        & = -\sigma_x \int |\nabla_x f|^2- \sigma_\theta \int |\partial_\theta f|^2 - \frac{\chi}{2} \int \partial_\theta B f^2 + \frac{\lambda}{2} \int \nabla_x \cdot (vf^2).
    \end{align*}

    We  drop the last integral term by the divergence theorem, and we estimate the third term by H\"older inequality to obtain,

    \begin{equation}
        \label{eq:fL2estim1}
        \frac{\dd}{\dd t}\int \frac{f^2}{2} \leq -\sigma_x \int |\nabla_x f|^2- \sigma_\theta \int |\partial_\theta f|^2 + \frac{\chi}{2} \|\partial_\theta B\|_{L^6_x(L^\infty_\theta)} \left\| \int f^2 \dd\theta \right\|_{L^{\frac{6}{5}}_x}.
    \end{equation}
    Noting that,
    \begin{equation*}
        \left\| \int f^2_t \dd\theta \right\|_{L^{\frac{6}{5}}_x} \leq \|f_t\|^2_{L^{\frac{12}{5}}_{x,\theta}},
    \end{equation*}
    and using the interpolation in $L^1_{x,\theta}\cap L^6_{x,\theta}$ with,
    \begin{equation*}
        \frac{5}{12} = \frac{\alpha}{1} + \frac{(1-\alpha)}{6},
    \end{equation*}
    leading to $\alpha = 3/10$, together with the conservation of mass, we obtain 

    \begin{equation}
         \left\| \int f^2_t \dd\theta \right\|_{L^{\frac{6}{5}}_x} \leq \|f_t\|^2_{L^{\frac{12}{5}}_{x,\theta}}\leq \left(\|f_t\|^{\frac{3}{10}}_{L^1_{x,\theta}}\|f_t\|^{\frac{7}{10}}_{L^6_{x,\theta}}\right)^2 \leq \|f_t\|^{\frac{14}{10}}_{L^6_{x,\theta}}.
    \end{equation}

    Using Sobolev inequality on the three-dimensional torus \cite[Prop. 1.1 p.3]{benyi2013sobolev} and the conservation of the mass, obtain,
    \begin{equation}
         \left\| \int f^2_t \dd\theta \right\|_{L^{\frac{6}{5}}_x} \leq \|f_t\|^{\frac{14}{10}}_{L^6_{x,\theta}}\leq C_{S}\left(\|\nabalaxtheta f_t\|^2_{L^2_{x,\theta}} + \frac{1}{2\pi}\right)^{\frac{7}{10}},
    \end{equation}

    for some $C_S >0$. Plugging the previous estimate in \eqref{eq:fL2estim1} together with Young product inequality for $p = \frac{10}{7}$,

    \begin{align*}
        \frac{\dd}{\dd t}\int \frac{f^2}{2} & \leq -\sigma_x\wedge \sigma_\theta \int |\nabalaxtheta f|^2 + C_S\frac{\chi}{2} \|\partial_\theta B\|_{L^6_x(L^\infty_\theta)}\left(  \|\nabalaxtheta f\|^2_{L^2_{x,\theta}} + 4\pi\right)^{\frac{7}{10}},\\
        & \leq -\sigma_x\wedge \sigma_\theta \int |\nabalaxtheta f|^2 + \mu\|\nabalaxtheta f\|^2_{L^2_{x,\theta}} + C_0\left(\mu^{-\frac{7}{3}}\|\partial_\theta B\|^{\frac{10}{3}}_{L^6_x(L^\infty_\theta)} + 1\right),\\
    \end{align*}
    for some constant $C_0>0$, for any $\mu> 0$. Then choosing $\mu$ to be $\frac{\sigma_x\wedge\sigma_\theta}{2}$, we obtain,
    \begin{equation}
        \label{eq:fL2estim2}
        \frac{\dd}{\dd t}\int \frac{f^2}{2} \leq -\frac{\sigma_x\wedge \sigma_\theta}{2} \int |\nabalaxtheta f|^2 + C\left(\|\partial_\theta B\|^{\frac{10}{3}}_{L^6_x(L^\infty_\theta)} + 1\right).
    \end{equation}
    Using Poincaré inequality together with the conservation of the mass, we have the estimate,
    \begin{equation*}
        \|f_t\|^2_{L^2} \leq 2C_P^2\|\nabalaxtheta f_t\|^2_{L^2} + 8\pi^2 \left|\int f_t \right|^2.
    \end{equation*}

    So that we obtain a bound only involving $f$,
    \begin{equation*}
        \frac{\dd}{\dd t}\int \frac{f^2}{2} \leq -\frac{\sigma_x\wedge \sigma_\theta}{4C_p^2} \int f^2 + C_0\left(\|\partial_\theta B\|^{\frac{10}{3}}_{L^6_x(L^\infty_\theta)} + 1\right),
    \end{equation*}
    for some constant $C_0>0$ independent of $f$ and $B$.

    We conclude that from Lemma~\ref{lem:gronwall}, that the following estimate holds, 
    \begin{equation*}
        \|f(t)\|_{L^2_{x,\theta}} \leq e^{-\alpha t}  \|f_{(0)}\|_{L^2_{x,\theta}} + C_0\left(\|\partial_\theta B\|^{\frac{10}{3}}_{L^\infty_t(L^6_x(L^\infty_\theta))} + 1\right), \text{ for all } t\geq 0,
    \end{equation*}
    with $\alpha = \frac{\sigma_x\wedge \sigma_\theta}{4C_p^2}$, and estimate~\eqref{est:absorbL2f} follows, by choosing 
    \begin{equation*}
        C_1 = C_0\left(\|\partial_\theta B\|^{\frac{10}{3}}_{L^\infty_t(L^6_x(L^\infty_\theta))} + 1\right) + 1
    \end{equation*}
     and $t_1>\log(\|f_{(0)}\|_{L^2_{x,\theta}} )/\alpha$.

    Finally, coming back to \eqref{eq:fL2estim2}, integrating between $t$ and $t+1$, and plugging the previous bound we obtain \eqref{est:absorbintTnablaf}.
\end{proof}

We now prove the $H^1_{x,\theta}$ absorption for a general phase $f$ space density associated with a scalar field $B$.

\begin{prop}
    \label{prop:fabsorbedH1}
    Let $B$ be a scalar field such that $B,\partial_\theta B, \partial_{\theta\theta}B \in L^\infty_t(L^6_x(L^\infty_\theta))$, $\sigma_x>0$ and $\sigma_\theta>0$. Then, there exists a constant $C_2>0$, depending on $\|B\|_{L^6_x(L^\infty_\theta)},\|\partial_{\theta}B\|_{L^6_x(L^\infty_\theta)}$ and $\|\partial_{\theta\theta}B\|^4_{L^6_x(L^\infty_\theta)}$, such that for any solution $f$ in $C_t(L^6_{x,\theta})\cap L^2_{loc,t}(H^1_{x,\theta})$ of the Fokker-Planck equation~\eqref{eq:ActiveMaterEquation} associated to $B$, there exists $t_2>0$, depending on $\|\rho_0\|_{L^6_x},\|f_0\|_{L^2_{x,\theta}}$, such that,
    \begin{equation*}
        f\in C([t_1,+\infty), H^1_{x,\theta}),
    \end{equation*}
    and, 
    \begin{equation*}
        \forall t\geq t_2(\|\rho_0\|_{L^6_x},\|f_0\|_{L^2_{x,\theta}})>0,
    \end{equation*}
    the following estimates hold,
    \begin{equation}
        \label{est:ActiveMatterH1Absorb}
        \|f(t)\|_{H^1_{x,\theta}} \leq C_2,
    \end{equation}
    and
    \begin{equation}
        \label{est:ActiveMatterL2H2}
       \int^{t_2+1}_{t_2} \int |\nabalaxtheta^2 f|^2 \dd s \leq C_2.
    \end{equation}
\end{prop}

\begin{proof}
    We start by considering a smooth initial condition $f_0$, and a mollified scalar field $B$, so that the following computations hold.
    In the following, we use the following differential operator notations, 
    \begin{equation*}
        \mathcal{L} = \sigma_x \Delta_x + \sigma_\theta 
        \partial_{\theta\theta}\text {, and }\widetilde{\nabla} = \begin{pmatrix}
            \sqrt{\sigma_x}\partial_{x_1} \\
            \sqrt{\sigma_x}\partial_{x_2} \\
            \sqrt{\sigma_\theta}\partial_{\theta}
        \end{pmatrix}.
    \end{equation*}
    Multiplying the equation for $f$ by $-\mathcal{L}f$ and integrating yields,
    \begin{align}
        \label{eq:H1absorbf}
        \frac{\dd}{\dd t} \int \frac{|\widetilde{\nabla}_{\xi}f|^2}{2} &= -\int (\mathcal{L}f)^2 + \chi \int (\mathcal{L}f)\partial_\theta(Bf) + \lambda \int (\mathcal{L}f) v\cdot \nabla_x f,\nonumber\\
        &\leq -\frac{\sigma_x^2\wedge\sigma_\theta^2}{2}\int |\nabalaxtheta^2 f|^2 + \underset{\eqdef\text{(I)}}{\underline{\chi^2 \int (\partial_\theta(Bf))^2}} + \lambda^2 \int |\nabla_x f|^2.
    \end{align}
    Let us treat separately (I) as follows,
    \begin{align}
        \label{est:ItermH1Absorbf}
        \chi^2 \int (\partial_\theta(Bf))^2 &\leq 2\chi^2 \int (\partial_\theta B)^2f^2 + B^2(\partial_\theta f)^2,\nonumber\\
        &\leq 2\chi^2 \int \left[(\partial_\theta B)^2f^2 - \partial_\theta(B^2)(\partial_\theta f)f - B^2f \partial_{\theta\theta}f\right],\nonumber\\
        &\leq 2\chi^2 \int \left[(\partial_\theta B)^2f^2 - \partial_\theta(B^2)\partial_\theta \left(\frac{f^2}{2}\right) + \frac{\chi^2}{2\mu}B^4f^2 \right] + \mu \int (\partial_{\theta\theta}f)^2,\nonumber\\
        &\leq 2\chi^2 \int \left(\left[(\partial_\theta B)^2 + \partial_{\theta\theta}\left(\frac{B^2}{2}\right) + \frac{\chi^2}{2\mu}B^4\right]f^2\right) + \mu \int (\partial_{\theta\theta}f)^2.
    \end{align}
    We define $\Phi_\mu$ as,
    \begin{equation*}
        \Phi_\mu = 2\chi^2\left((\partial_\theta B)^2 + \partial_{\theta\theta}\left(\frac{B^2}{2}\right) + \frac{\chi^2}{2\mu}B^4\right) = \chi^2\left(3(\partial_\theta B)^2 + B\partial_{\theta\theta}B + \frac{\chi^2}{\mu}B^4\right).
    \end{equation*}

    We then estimate its norm as follows, developing the product and bounding the squares,
    \begin{equation*}
        |\Phi_\mu| \leq \chi^2C\left(1+(\partial_\theta B)^4 + (\partial_{\theta\theta}B)^4 + \frac{\chi^2}{\mu}B^4\right).
    \end{equation*}

    By hypothesis on $B,\partial_\theta B,\partial_{\theta,\theta}B$, and the fact that the domain as finite volume, we obtain that,  
    \begin{equation*}
        \Phi_\mu\in L^\infty_t(L^{\frac{3}{2}}_x(L^\infty_\theta)),
    \end{equation*}
    with estimate, 
    \begin{equation*}
        \|\Phi_\mu\|_{L^{\frac{3}{2}}_x(L^\infty_\theta)} \leq C\left(1+\frac{1}{\mu}\right)\left(1+\|B\|^4_{L^6_x(L^\infty_\theta)} + \|\partial_{\theta}B\|^4_{L^6_x(L^\infty_\theta)} +\|\partial_{\theta\theta}B\|^4_{L^6_x(L^\infty_\theta)} \right),
    \end{equation*}

    for some $C>0$ only depending on $\chi$.

    Coming back to \eqref{est:ItermH1Absorbf},
    bounding the first integral term with H\"older inequality,
    \begin{equation}
        \int \Phi_\mu \f^2 \leq \int \|\Phi_\mu\|_{L^\infty_\theta}\int \f^2 d\theta dx \leq \|\Phi_\mu\|_{L^{\frac{3}{2}}_x(L^\infty_\theta)}\left\|\int \f^2 \dd\theta\right\|_{L^3_x}.
        \label{est:ItermH1AbsorbffirstTerm}
    \end{equation}

    Interpolating $\f$ in $L^1_\theta\cap L^6_\theta$ with exponents,
    \begin{equation*}
            \frac{1}{2} = \frac{\alpha}{1}+\frac{(1-\alpha)}{6}  \implies \alpha = \frac{2}{5},
    \end{equation*}
    followed by H\"older inequality and Sobolev-Gagliardo-Nirenberg inequality Proposition~\ref{thm:GagliardoNirenbeg} we obtain,
    \begin{align}
        \left\|\int \f^2 \dd\theta\right\|_{L^3_x} & \leq \left( \int \left(\|\f\|_{L^1_\theta}^{\frac{2}{5}}\|\f\|_{L^6_\theta}^{\frac{3}{5}}\right)^6dx\right)^{\frac{1}{3}},\\
        &\leq \|\rho\|_{L^6_x}^{\frac{4}{5}}\|\f\|^{\frac{6}{5}}_{L^6_{x,\theta}},\\
        &\leq C_{GN}\|\rho\|_{L^6_x}^{\frac{4}{5}}\left(\|\nabalaxtheta \f\|^2_{L^2_{x,\theta}}+\|\f\|^2_{L^2_{x,\theta}}\right)^{\frac{3}{5}}.
    \end{align}
    Injecting the previous estimate in \eqref{est:ItermH1AbsorbffirstTerm}, we deduce that there exists a constant $C_2>0$ such that,
    \begin{align}
        \label{eq:Phif2estim}
        \int \Phi_\mu \f^2 &\leq C_{GN}\|\Phi_\mu\|_{L^{\frac{3}{2}}_x(L^\infty_\theta)}\|\rho\|_{L^6_x}^{\frac{4}{5}}\left(\|\nabalaxtheta\f\|^2_{L^2_{x,\theta}}+\|\f\|^2_{L^2_{x,\theta}}\right)^{\frac{3}{5}},\nonumber\\
        &\leq C_2\left(\|\Phi_\mu\|_{L^{\frac{3}{2}}_x(L^\infty_\theta)}\|\rho\|_{L^6_x}^{\frac{4}{5}}\right)^{\frac{5}{2}} + \left(\|\nabalaxtheta\f\|^2_{L^2_{x,\theta}}+\|f\|^2_{L^2_{x,\theta}}\right),
    \end{align}
    Plugging \eqref{eq:Phif2estim} and \eqref{est:ItermH1Absorbf} with  $\mu = \frac{\sigma_x^2\wedge\sigma_\theta^2}{4}$ in \eqref{eq:H1absorbf}, we obtain,
    \begin{align}
        \frac{\dd}{\dd t} \int \frac{|\widetilde{\nabla}_{\xi} f|^2}{2}
        &\leq -\frac{\sigma_x^2\wedge\sigma_\theta^2}{4}\int |\nabalaxtheta^2 f|^2 + C_2\left(\|\Phi\|_{L^{\frac{3}{2}}_x(L^\infty_\theta)}\|\rho\|_{L^6_x}^{\frac{4}{5}}\right)^{\frac{5}{2}} \nonumber\\
        \label{eq:H1estimeQuasiFinal}
        &\hspace{3em}+ \|\nabalaxtheta\f\|^2_{L^2_{\xi}}+\|\f\|^2_{L^2_{x,\theta}} + \lambda^2 \int |\nabla_x \f|^2.
    \end{align}
Finally integrating over $s$ between $t$ and $t+1$, 

    {\small
    \begin{equation}
        \int \frac{|\widetilde{\nabla}_{\xi} f(t+1)|^2}{2}
        \leq  C_2\left(\|\Phi\|_{L^\infty_t(L^{\frac{3}{2}}_x(L^\infty_\theta))}\|\rho\|_{L^\infty_t(L^6_x)}^{\frac{4}{5}}\right)^{\frac{5}{2}} + C\int^{t+1}_t \left(\int f^2 + \int|\nabalaxtheta \f|^2 \right),
    \end{equation}}

    for some $C>0$. We obtain the required estimate~\eqref{est:ActiveMatterH1Absorb} by using the estimates \eqref{est:absorbintTnablaf}, \eqref{est:rhoestimLp} and \eqref{est:absorbL2f}. Then estimate~\eqref{est:ActiveMatterL2H2} is also obtain by integrating equation~\eqref{eq:H1estimeQuasiFinal} but moving the first negative term on the left hand side and using estimate~\eqref{est:ActiveMatterL2H2}. The result for general scalar fields $B$ and initial conditions, holds since $t_2$ does not depend on the higher regularity of the initial condition nor the scalar field B. We thus conclude similarly as in Theorem~\ref{thm:AbsorbH3H1Parabolic} by taking a sequence of smooth initial conditions and mollified scalar fields converging to $f_0,B$, and the conclusion follows from Lemma~\ref{lem:weakL2convergenceLinfty}.
\end{proof}

\begin{thm} 
    \label{thm:AbsorbH3H1Elliptic}
    Let $\sigma_x>0$ and $\sigma_\theta>0$. Then, there exists a constant $C_2>0$, such that for any solution $f$ in $C_t(L^6_{x,\theta})\cap L^2_{loc,t}(H^1_{x,\theta})$ of the system~\eqref{sys:FormicidaeElliptic}, there exists $t_2>0$, depending on $\|\rho_0\|_{L^6_x},\|f_0\|_{L^2_{x,\theta}}$, such that,
    \begin{equation*}
        f\in C([t_1,+\infty), H^1_{x,\theta}),
    \end{equation*}
    and, 
    \begin{equation*}
        \forall t\geq t_2(\|\rho_0\|_{L^6_x},\|f_0\|_{L^2_{x,\theta}})>0,
    \end{equation*}
    the following estimates hold,
    \begin{equation}
        \label{est:fH1absorbElliptic}
        \|f(t)\|_{H^1_{x,\theta}} \leq C_2,
    \end{equation}
    and
    \begin{equation}
        \label{est:fL2H2absorbElliptic}
       \int^{t_2+1}_{t_2} \int |\nabalaxtheta^2 f|^2 \dd s \leq C_2.
    \end{equation}   
\end{thm}

\begin{proof}
    We first apply Proposition~\ref{prop:rhoestimLp}, to obtain that there exists $C_0>0$ such that for any initial condition, there exists $t_0$, such that,
    \begin{equation*}
        \|\rho(t)\|_{L^6_x} \leq C_0, \forall t\geq t_0 >0.
    \end{equation*}
    This implies from classical $L^p$-elliptic theory that,
    \begin{equation*}
        \|c(t)\|_{W^{2,6}_x} \leq C_0, \forall t\geq t_0 >0.
    \end{equation*}
    This yields the hypothesis of Proposition~\ref{prop:fabsorbedH1} after $t_0$. Using the estimates from Proposition~\ref{prop:elliL2f} on the finite interval $[0,t_0]$, to control the norms $\|\rho(t_0)\|_{L^6_x}$ and $\|f(t_0)\|_{L^2_{x,\theta}}$ only by the same noms at $t=0$, allows to conclude the result of Proposition~\ref{prop:fabsorbedH1}.
\end{proof}

%% file: LinearisedOperator.tex
\subsection{Linearized equations around the homogeneous solution}

In this section, we study the linearization of the parabolic system~\eqref{sys:FormicidaeParabolic} and the elliptic system~\eqref{sys:FormicidaeElliptic} around their respective homogeneous normalized steady states, 
$u_*$ and $\f_*$, given by
\begin{equation*}
 f_*(t,x,\theta)\equiv \frac{1}{2\pi}, c_*(t,x)\equiv  \frac{1}{\gamma}, u_* \defeq (f_*,c_*).
\end{equation*}
The linearized equation is obtained by considering a linear perturbation around the homogeneous solution. As such, the linearized equation for \eqref{sys:FormicidaeParabolic} is
\begin{equation}
    \label{eq:LinearizedHomogeneousParabolic}
    \begin{cases}
        \partial_t \ff=\nabla_\xx\cdot(\sigma_\xx\nabla_\xx\ff -\lambda v \ff)+\partial_\theta(\sigma_\theta\partial_\theta \ff) -\chi f_*  \partial_\theta B_\tau[c],\\
        \partial_t c =  \sigma_c \Delta_x c -\gamma c + \int f \dd\theta.
    \end{cases}
\end{equation}
We also define the linear operator $\LinOpPara :  L^2_{x,\theta} \times L^2_{x}  \supset D(\LinOpPara) \to L^2_{x,\theta} \times L^2_{x}$,

\begin{equation}
        \tag{$\LinOpPara$}
        \label{def:LinOpHomogPara}
        \LinOpPara \begin{pmatrix}
            g\\
            c
        \end{pmatrix} =\begin{pmatrix}\nabla_\xx\cdot(\sigma_\xx\nabla_\xx g -\lambda v g)+\partial_\theta(\sigma_\theta\partial_\theta g) -\chi f_* \partial_\theta B_\tau[c]\\
        \sigma_c \Delta_x c -\gamma c + \int g \dd\theta
        \end{pmatrix}
\end{equation}
The linearized equation for \eqref{sys:FormicidaeElliptic} is
\begin{equation}
    \label{eq:LinearizedHomogeneousElliptic}
    \begin{cases}
        \partial_t \ff=\nabla_\xx\cdot(\sigma_\xx\nabla_\xx\ff -\lambda v \ff)+\partial_\theta(\sigma_\theta\partial_\theta \ff) - \chi f_* \partial_\theta B_\tau[c],\\
        \gamma c - \sigma_c \Delta_x c = \int f \dd\theta.
    \end{cases}
\end{equation}

We also define the following linear operator $\LinOpEllip :  L^2_{x,\theta} \supset D(\LinOpEllip) \to L^2_{x,\theta}$,

\begin{equation}
        \tag{$\LinOpEllip$}
        \label{def:LinOpHomogEllip}
        \LinOpEllip g =\nabla_\xx\cdot(\sigma_\xx\nabla_\xx g -\lambda vg)+\partial_\theta(\sigma_\theta\partial_\theta g) -\chi f_* \partial_\theta B_\tau \left[(\gamma -\sigma_c \Delta_x)^{-1} \int g \dd\theta\right].
\end{equation}
The starting point for the linear analysis for both systems is the fully inviscid ($\sigma_\xx=\sigma_\theta=0$) instability condition for an integer wavenumber $k \geq 1$
\begin{equation}
        \chi(2\pi k \tau + 1) > \lambda (\gamma + 4\pi \sigma_c k^2).
\end{equation}
Given this condition we show a lower bound on the dimension of the eigenspace of the inviscid operator in terms of $k$. Using a perturbation argument with the Riesz projector, similar as in \cite{albritton2022non}, we conclude by demonstrating that we have the same lower bound to the dimension of the eigenspace of the viscous operator for $\sigma_\xx,\sigma_\theta$ sufficiently small.

\subsection{Parabolic-Elliptic Case}
\input{LinearInstabilityElliptic}

\subsection{Parabolic-Parabolic Case}
\input{LinearInstabilityParabolic}

%% file: LinearInstabilityElliptic.tex
In this section we give a lower bound on the dimension of the unstable manifold of the linearized equation for the parabolic-elliptic system \eqref{sys:FormicidaeElliptic}.
\begin{thm}
    \label{thm:ExistenceEigenFunLin}
    Suppose that there exists an integer wavenumber $k\geq 1$, such that,
    \begin{equation}
        \label{cond:InviscidInstabKwave}
        \chi(2\pi k \tau + 1) > \lambda (\gamma + 4\pi \sigma_c k^2).
    \end{equation}
    Then there exists $\sigma_\theta^*>0$ such that for $\sigma_\theta \in [0,\sigma_\theta^*)$ there exists $\sigma_x^*(\sigma_\theta)>0$ such that for $\sigma_x \in [0,\sigma_x^*)$ there exist $k$ unstable eigenvalues of $\LinOpEllip$,
    \begin{equation*}
        \mu^1,\dots,\mu^k \in \{\mu \in \dC | \Real\mu >0\} \cap \Sigma(\LinOpEllip),
    \end{equation*}
    where $\Sigma(\LinOpEllip)$ is the spectrum of $\LinOpEllip$. Moreover if $X$ is the real invariant subspace associated with the $\mu^i$'s and their complex conjugates, it has the dimensional lower bound,
    \begin{equation*}
        \dim X \geq 4k.
    \end{equation*}
    Furthermore, we can choose a $4k$-dimensional orthogonal sub-basis of real functions in $E$, such that each of the functions are constant in $x_1$ or $x_2$.
\end{thm}

For the proof we proceed as follows. We first introduce a family of ansatzes that are constant in the $x_2$ variable. Using the ansatz, we show that the fully inviscid operator has positive real eigenvalues if the fully inviscid instability condition~\eqref{cond:InviscidInstabKwave} is satisfied. This result is stated in Lemma~\ref{lem:InviscidEigenVal}. Then, in Theorem~\ref{thm:EigenValSystem}, we show that the Riesz projector for the case $\sigma_\xx=0,\sigma_\theta>0$ converges to the Riesz projector for $\sigma_\xx=0,\sigma_\theta=0$ as $\sigma_{\theta}$ goes to zero, using results of elliptic theory obtained in Lemma~\ref{lem:InversibilityEllipticOperator}. The link between the dimensional number in Theorem~\ref{thm:EigenValSystem} and Theorem~\ref{thm:ExistenceEigenFunLin} is made clear by exchanging $x_1$ with $x_2$ and rotating in $\theta$.

For an integer wavenumber $k\geq 1$, we define the ansatz $f^k(\theta,x)$ as
\begin{equation}
    \label{eq:LinFamilyEllip}
    f^k(\theta, x_1, x_2) = a(\theta) \cos(2\pi k x_1) +  b(\theta) \sin(2\pi k x_1),
\end{equation}
where $a,b$ are functions defined on $\mathbb{T}_{2\pi}$. We use the notation $\Bar{\cdot}$ to indicate the average in $\theta$. That is, given a function $a$ defined on $\mathbb{T}_{2\pi}$
\begin{equation}
    \label{def:BarnotationIntegral}
   \Bar{a} \defeq \int a(\theta) \dd\theta.
\end{equation}

We can solve explicitly the equation for the chemotactic field for the family $f^k$, such that

\begin{equation}
    \label{eq:LinFamilyChemotactic}
    (\gamma - \sigma_c \Delta _x)^{-1} \int f^k \dd\theta = \frac{1}{\gamma + 4\pi^2\sigma_ck^2}[\Bar{a} \cos(2\pi k x_1 ) + \Bar{b} \sin(2\pi  k x_1 )].
\end{equation}

Hence, applying $\LinOpEllip$ to $f^k$, and writing $\xonepik$ for $2\pi k x_1$, we obtain
\begin{align*}
    \LinOpEllip f^k = 
    & - 4\pi^2 k^2 \sigma_x \left(a(\theta) \cos(\xonepik) + b(\theta) \sin(\xonepik)\right) + \sigma_\theta \left(a''(\theta) \cos(\xonepik) + b''(\theta) \sin(\xonepik)\right) \\
     &- 2\pi k \lambda \cos(\theta) \left(b(\theta)\cos(\xonepik) - a(\theta)\sin(\xonepik)\right) \\
     &-  \frac{\chi f_*}{\gamma + 4\pi^2 \sigma_c k^2}\Big(
        2\pi k \left(\Bar{b}\cos(\xonepik) - \Bar{a}\sin(\xonepik)\right) \cos(\theta) \\
     & \hspace{8em} + 4\pi^2 k^2 \tau \left(\Bar{a}\cos(\xonepik) + \Bar{b}\sin(\xonepik)\right) 
        \left(\cos^2(\theta) - \sin^2(\theta)\right) \Big).
\end{align*}

Finally, we denote $\ABLinVec$ for the vector $(a,b)^\top$, such that 
\begin{equation*}
    f^k(x,\theta)=\ABLinVec\cdot(\cos\xonepik,\sin\xonepik)^\top.
\end{equation*}
By identifying the terms $\sin(\xonepik)$ and $\cos(\xonepik)$, the eigenproblem can be written in terms of $\ABLinVec$ as
\begin{equation}
    \begin{split}
        \label{eq:ABVecFull}
        \mu A & =-4\pi^2 k^2 \sigma_x\ABLinVec +\sigma_\theta \frac{d^2}{d\theta^2}\ABLinVec + 2\pi k \lambda\begin{bmatrix}0 & -\cos(\theta) \\ \cos(\theta) & 0 \end{bmatrix}\ABLinVec \\
        & \hspace{3em} + \frac{\chi k }{\gamma + 4\pi^2\sigma_c k^2}\begin{pmatrix} -2\pi k \tau \cos(2\theta) & \cos(\theta) \\ -\cos(\theta) & -2\pi k \tau \cos(2\theta) \end{pmatrix}\overline{\ABLinVec},    
    \end{split}
\end{equation}
where $\overline{\ABLinVec}$ indicates the integral of the vector, $\overline{\ABLinVec} \defeq \begin{pmatrix}\Bar{a}\\\bar{b}\end{pmatrix}$.

For the elliptic system, the $\sigma_\xx$-term will only shift the eigenvalues to the left in the complex plane. Hence, we study the eigenproblem for the case $\sigma_\xx=0$, and the eigenproblem can be written as
\begin{equation}
    \label{eq:SystemLinPhi}
    \LinOpThet \ABLinVec=\mu \ABLinVec, \ \text{where} \ \LinOpThet \ABLinVec\defeq \sigma \frac{d^2}{d\theta^2}\ABLinVec + \VOpLinVec \ABLinVec+  \cchi \AmatLinVec \overline{\ABLinVec}, 
\end{equation}
with the notation
\begin{equation}
    \label{def:LinearisedEllipticVandBDefinition}
    \VOpLinVec \defeq \llambda \begin{bmatrix}
        0 &  - \cos(\theta) \times \Id  \\
        \cos(\theta)\times\Id & 0
    \end{bmatrix}, \AmatLinVec(\theta) \defeq \begin{pmatrix}
    -\ttau \cos(2\theta) & \cos(\theta)\\
    -\cos(\theta) & -\ttau \cos(2\theta)\end{pmatrix},
\end{equation}
and with the parameter notation
\begin{equation}
    \cchi = \frac{\chi k}{\gamma + 4\pi^2\sigma_c k^2}, \ttau = 2\pi k \tau, \sigma = \sigma_\theta, \llambda = 2\pi k \lambda.
\end{equation}
For this eigenproblem written in terms of $\ABLinVec$ we have the following existence result.
\begin{thm}
    \label{thm:EigenValSystem}
    Suppose that for $\cchi >0, \ttau > 0$ and $\llambda > 0 $ the following condition holds,
    \begin{equation}
        \label{hyp:EigenHypotesis}
        2\pi\cchi(\ttau + 1) > \llambda.
    \end{equation}
    Then, there exists $\sigma^*>0$ such that for any $\sigma \in [0,\sigma^*)$, there exists a real strictly positive eigenvalue $\mu$ of multiplicity $2$ for the operator $\LinOpThet$ defined in \eqref{eq:SystemLinPhi}.

    
\end{thm}

Using a rotation of the original ansatz
\begin{equation}
    (f^k)^\perp(x,\theta)=f^k(x_2,x_1,\theta+\pi/2),
\end{equation}
we note that we get $2$ times $k\times 2$ linearly independent eigenfunctions in total, so that Theorem~\ref{thm:EigenValSystem} implies Theorem~\ref{thm:ExistenceEigenFunLin} and gives the right eigenspace dimensional number.

Theorem~\ref{thm:EigenValSystem} will be proven at the end of this section. As explained above, we start with proving an existence result for the fully inviscid ($\mathsf{L}^E_{\sigma=0}$) eigenproblem.

\newcommand{\dLVecComp}{\mathbb{L}^2_\theta}
\newcommand{\dHVecComp}{\mathbb{H}_\theta}

For the proofs that are to come, we define the Hilbert space $\dLVecComp$ as the product space of square-integrable complex valued functions $(L^2(\mathbb{T}_{2\pi},\dC))^2$. Similarly, we define $\dHVecComp^k$ as the product space of square-integrable complex valued functions with $k\geq 1$ square integrable derivatives, both equipped with their natural norm.

\begin{lem}[Inviscid case]
    \label{lem:InviscidEigenVal}
    Suppose that for $\cchi >0, \ttau > 0$ and $\llambda > 0 $ the instability condition \eqref{hyp:EigenHypotesis} holds.

    Then, for the operator
    \begin{equation}
        \label{eq:SystemLinPhiInviscid}
        \mathsf{L}^E_{0}=\left(\VOpLinVec+  \cchi \AmatLinVec \overline{[\cdot]}\right),
    \end{equation}
    there exists an isolated real eigenvalue $ \muinvi> 0$ of multiplicity $2$.
\end{lem}


\begin{proof}
    We first rewrite the eigenproblem using the explicit inverse of $\mu\Id-\VOpLinVec$. Indeed, for any $\mu \in \dC$, such that $\Real(\mu) >0$, the operator
    \begin{equation}
        \label{eq:MuVopInvertible}
        \mu \Id - \VOpLinVec,
    \end{equation}
    has inverse
    \begin{equation}
        \label{eq:InverseTransportTerm}
        \left(\mu \Id - \VOpLinVec\right)^{-1} = \frac{1}{\mu^2 + \llambda^2\cos^2(\theta)}\begin{bmatrix}
            \mu \Id & -\llambda \cos(\theta) \times \Id\\
             \llambda \cos(\theta) \times \Id & \mu \Id
            \end{bmatrix}.
    \end{equation}
    Therefore, the eigenproblem associated to the operator \eqref{eq:SystemLinPhiInviscid}, can be recasted as
    \begin{equation}
        \label{eq:EigenProbInviscidRecast}
        \ABLinVec = \cchi \left(\mu \Id - \VOpLinVec\right)^{-1} \AmatLinVec \overline{\ABLinVec}.
    \end{equation}
    Integrating in $\theta$ and rewriting gives
    \begin{equation*}
        \cchi \overline{\left(\mu \Id - \VOpLinVec\right)^{-1} \AmatLinVec}\overline{\ABLinVec}=\overline{\ABLinVec}.
    \end{equation*}
    We now observe that the eigenproblem, as in \eqref{eq:SystemLinPhiInviscid}, is equivalent to the finite-dimensional eigenproblem
    \begin{equation*}
        M_\mu\overline{\ABLinVec}=\overline{\ABLinVec}, \text{ with } M_\mu \defeq \cchi \overline{\left(\mu \Id - \VOpLinVec\right)^{-1} \AmatLinVec}.
    \end{equation*}
    We note that $M_\mu\overline{\ABLinVec}=\overline{\ABLinVec}$ implies the existence of an eigenvalue for the original problem by mapping $w=(\Bar{a},\Bar{b})^\top$ to $\left( \cchi \left(\mu \Id - \VOpLinVec \right)^{-1} \AmatLinVec w\right)$.
    By explicit computations, we obtain that $M_\mu$ is of the form
    \begin{equation}
        \label{eq:InviscidEllipMdef}
        M_\mu = \cchi\begin{pmatrix}
           \mathcal{I}(\ttau, \llambda, \mu) & 0\\
           0 & \mathcal{I}(\ttau, \llambda, \mu)
        \end{pmatrix},
    \end{equation}
    with,
    \begin{align}
        \mathcal{I}(\ttau, \llambda, \mu) & = \int_0^{2\pi}\frac{-\ttau \mu \cos(2\theta)+ \llambda \cos^2(\theta)}{\mu^2+\llambda^2 \cos^2(\theta)}\dd\theta \nonumber\\
         \label{eq:InviscidEllipIntegraldef}
        &= \frac{2\pi}{\llambda^2 \sqrt{\mu^2 + \llambda^2}}\left(\llambda^2\ttau + (\llambda - 2\ttau \mu)\left(\sqrt{\mu^2+\llambda^2}-\mu\right)\right).
    \end{align}
    Since $\mathcal{I}$ is a strictly decreasing, continuous, convex function with respect to $\mu$ on $[0,+\infty)$ and its limit is zero for $\mu\to+\infty$, the existence for the eigenproblem, by the intermediate value theorem, is equivalent to $\cchi\mathcal{I}(\ttau, \llambda, 0) > 1$. That is,
    \begin{equation}
        \label{eq:ConditionCchiTTauLlambda}
        2\pi \cchi (\ttau + 1) > \llambda.
    \end{equation}
    Since $\mathcal{I}(\ttau, \llambda, \cdot)$ is holomorphic for $\mu\in\dC\cap\{\Real(\mu)>0\}$, under the inviscid instability condition, the fully inviscid eigenproblem admits an isolated real eigenvalue in the right half plane, that we denote by $\muinvi>0$. The two-dimensional eigenspace for this eigenvalue is explicitly given by the image of the map $\dR^2\to \dLVecComp$ given by
    \begin{equation}
        w  \mapsto  \cchi \left(\muinvi \Id - \VOpLinVec \right)^{-1} \AmatLinVec w.
    \end{equation}
\end{proof}

We proceed by showing a stability estimate for the operator $\LinOpThet$ with $\chi = 0$, 
\begin{equation*}
    \left(\sigma \frac{\dd^2}{\dd\theta^2} + \VOpLinVec\right).
\end{equation*}

\begin{lem}
    \label{lem:InversibilityEllipticOperator}
    Suppose that $\mu \in \dC$ such that $\Real (\mu) > 0$. Then, for any real $\llambda$, and for any $\sigma \geq 0$, the operator,
    \begin{equation*}
        \left(\mu - \left(\sigma \frac{\dd^2}{\dd\theta^2} + \VOpLinVec\right)\right),
    \end{equation*}
    admits an inverse $\mathsf{R}^\sigma(\mu)$ as a map $\dLVecComp \to \dHVecComp^2$ for $\sigma > 0$ and as a map $\dLVecComp\to\dLVecComp$ for $\sigma = 0$. The operator $\mathsf{R}^\sigma(\mu)$ satisfies the estimate
    \begin{equation}
        \label{est:EllipResolLCompL2L2}
        \|\mathsf{R}^\sigma(\mu)\|_{\dLVecComp\to \dLVecComp} \leq \frac{1}{\Real(\mu)}, \ \ \forall \sigma \geq 0.
    \end{equation}
    Furthermore, for any $g \in \dHVecComp^2$, we have the stability estimate,
    \begin{equation}
        \label{est:RsigToR0}
        \| [\mathsf{R}^\sigma(\mu) - \mathsf{R}^0(\mu)]g\|_{\dLVecComp} \leq \frac{\sqrt{\sigma}}{\Real(\mu)} \sqrt{\|g\|_{\dLVecComp} \left\|\frac{\dd^2}{\dd\theta^2}\mathsf{R}^0(\mu)g\right\|_{\dLVecComp}}.
    \end{equation}
\end{lem}


\begin{proof}
    We first use the continuity method to prove the existence of the inverse. After this, we prove the stability estimate.
    
    For the existence part, the case $\sigma = 0$ is explicit and already treated in the proof of Lemma~\ref{lem:InviscidEigenVal}. 
    
    For the continuity argument, for fixed $\sigma >0$ and $\mu \in \dC$, such that $\Real(\mu) > 0$, we show that if for $\eta \in [0,1]$, and any $g \in \dLVecComp$, there exists a unique solution $\phi \in \dHVecComp^2$ to the equation,
    \begin{equation}\label{eq:etag}
        \left(\mu - \left(\sigma \frac{\dd^2}{\dd\theta^2} + \eta \VOpLinVec\right)\right) \phi = g,
    \end{equation}
    then there exists a larger number strictly greater than $\eta$ for which the equation \eqref{eq:etag} can also be uniquely solved. For the case $\eta = 0$ equation~\eqref{eq:etag} can be uniquely solved by classical elliptic theory. 
    
    Now, let $\varepsilon > 0$, which is to be fixed later, and introduce the map $F : \dLVecComp \to \dLVecComp$ which writes as $\psi\mapsto \phi$ where $\phi$ is the unique solution of 
    \begin{equation*}
    \left(\mu - \left(\sigma \frac{\dd^2}{\dd\theta^2} + \eta \VOpLinVec\right)\right) \phi = g + \varepsilon \VOpLinVec \psi.
    \end{equation*}
    We will show by a Banach contraction mapping theorem that $F$ has a fixed point, and as such, the continuity method works and \eqref{eq:etag} can be uniquely solved for all $\eta\in[0,1]$.
    
    For showing the contraction property, take $\psi_1, \psi_2 \in \dLVecComp$, let $\phi_1 = F(\psi_1), \phi_2 = F(\psi_2)$ and denote the differences by $\Tilde{\psi} = \psi_2 - \psi_1$, $\Tilde{\phi} = \phi_2 - \phi_1$. Then, $\Tilde{\phi}$ is the unique solution to,

    \begin{equation}
        \label{eq:EllipVectExistenceLip}
        \left(\mu - \left(\sigma \frac{\dd^2}{\dd\theta^2} + \eta \VOpLinVec\right)\right) \Tilde{\phi} = \varepsilon \VOpLinVec \Tilde{\psi}.
    \end{equation}
    Let $\Tilde{\phi}^*$ denote the complex conjugate of the vector $\Tilde{\phi}$. Taking the scalar product of \eqref{eq:EllipVectExistenceLip} with $\Tilde{\phi}^*$ and integrating in $\theta$ we obtain,
    \begin{equation}
        \label{eq:EllipticComplexEstim1}
        \mu \int |\Tilde{\phi}|^2 \dd\theta + \sigma \int \left|\frac{\dd}{\dd\theta} \Tilde{\phi}\right|^2 \dd\theta - \eta \int \Tilde{\phi}^* \VOpLinVec \Tilde{\phi} \dd\theta = \varepsilon \int \Tilde{\phi}^* \VOpLinVec \Tilde{\psi} \dd\theta.
    \end{equation}

    The following computation shows that $\Tilde{\phi}^* \VOpLinVec \Tilde{\phi}$ is purely imaginary. 
    
    Denote $\Tilde{\phi} = \begin{bmatrix}
        \Tilde{\phi}_1\\\Tilde{\phi}_2
    \end{bmatrix}=\begin{bmatrix}
        \Real\Tilde{\phi}_1+i\Imag\Tilde{\phi}_1\\ \Real\Tilde{\phi}_2+i\Imag\Tilde{\phi}_2
    \end{bmatrix}$, 
    
    \begin{align}
        \label{eq:IdentityVopPhiPhistar}
         \Tilde{\phi}^* \VOpLinVec \Tilde{\phi} & = \begin{bmatrix}
             \Real\Tilde{\phi}_1 - i \Imag\Tilde{\phi}_1\\\Real\Tilde{\phi}_2 - i \Imag\Tilde{\phi}_2 
         \end{bmatrix} \cdot \begin{bmatrix}
             -\cos \times (\Real\Tilde{\phi}_2 + i \Imag\Tilde{\phi}_2)\\\cos \times (\Real\Tilde{\phi}_2 + i \Imag\Tilde{\phi}_2)
         \end{bmatrix},\nonumber\\
         & = - \cos\times (\Real \Tilde{\psi}_1\Real \Tilde{\psi}_2 + \Imag \Tilde{\psi}_1\Imag \Tilde{\psi}_2 + i \Real \Tilde{\psi}_1\Imag \Tilde{\psi}_2 - i \Imag \Tilde{\psi}_1\Real \Tilde{\psi}_2)\nonumber\\
         & \hspace{1.em} + \cos\times (\Real \Tilde{\psi}_1\Real \Tilde{\psi}_2 + \Imag \Tilde{\psi}_1\Imag \Tilde{\psi}_2 - i \Real \Tilde{\psi}_1\Imag \Tilde{\psi}_2 + i \Imag \Tilde{\psi}_1\Real \Tilde{\psi}_2), \nonumber\\
         & = i2\cos \times (- \Real \Tilde{\psi}_1\Imag \Tilde{\psi}_2 + \Imag \Tilde{\psi}_1\Real \Tilde{\psi}_2).
    \end{align}
    Thus, taking the real part on both sides yields
    \begin{equation*}
        \mu \int |\Tilde{\phi}|^2 \dd\theta + \sigma \int \left|\frac{\dd}{\dd\theta} \Tilde{\phi}\right|^2 \dd\theta = \varepsilon \Real \left[ \int \Tilde{\phi}^* \VOpLinVec \Tilde{\psi}\dd\theta\right].
    \end{equation*}

    Now, bounding the right term by the modulus and applying the Cauchy--Schwartz inequality, we get

    \begin{equation}
        \Real \mu \int |\Tilde{\phi}|^2\dd\theta + \sigma \int \left|\frac{\dd}{\dd\theta} \Tilde{\phi}\right|^2 \dd\theta\leq \varepsilon \llambda \|\Tilde{\phi}\|_{\dLVecComp} \|\Tilde{\psi}\|_{\dLVecComp},
    \end{equation}
    where we used that $\| \VOpLinVec \|_{\dLVecComp\to \dLVecComp} = \llambda$. This implies the contraction property of $F$ for $\varepsilon$ sufficiently small, as

    \begin{equation*}
        \|F(\psi_1) - F(\psi_2)\|_{\dLVecComp} \leq \frac{\varepsilon\llambda}{\Real (\mu)} \|\psi_1 - \psi_2\|_{\dLVecComp}.
    \end{equation*}

    Now, applying the Banach contraction mapping theorem, this implies that if equation~\eqref{eq:etag} can be uniquely solved for any $\eta \in [0,\eta^*)$, it can also be uniquely solved for $\eta\in[0,\eta^\ast+\frac{\Real(\mu)}{\llambda})$. 
    We now denote by $R^\sigma(\mu) : \dLVecComp \to \dLVecComp$ $g\mapsto\phi$ for $\eta=1$.
    This concludes the proof of the existence of the inverse.
    
    Using the same computations as above, we obtain the estimate,
    \begin{equation*}
        \|\mathsf{R}^\sigma(\mu)g\|_{\dLVecComp} \leq \frac{1}{\Real (\mu)}\|g\|_{\dLVecComp}, \ \ \forall g \in \dLVecComp.
    \end{equation*}

    We now prove the stability estimate~\eqref{est:RsigToR0}. For this, let $g\in \dHVecComp^2$, and let $\phi_\sigma, \phi_0$ be the unique solutions of the equations,
    \begin{align*}
            \left(\mu -\left(\sigma \frac{\dd^2}{\dd\theta^2}+ \VOpLinVec\right)\right)\phi_\sigma = g,\\
            (\mu - \VOpLinVec)\phi_0 = g.
    \end{align*}
    The difference $\Tilde{\phi}_\sigma = \phi_\sigma - \phi_0$ solves,
    \begin{equation*}
        (\mu - \VOpLinVec)\Tilde{\phi}_\sigma = \sigma \frac{\dd^2}{\dd\theta^2} \phi_\sigma.
    \end{equation*}
    Multiplying by the complex conjugate $\Tilde{\phi}_\sigma^*$, using the identity~\eqref{eq:IdentityVopPhiPhistar}, integrating and taking the real part, we obtain,
    \begin{align*}
        \Real (\mu) \int |\Tilde{\phi}_\sigma|^2\dd\theta & = -\sigma \int \left|\frac{\dd}{\dd\theta}\phi_\sigma\right|^2\dd\theta + \sigma \Real \left[\int \phi_\sigma \frac{\dd^2}{\dd\theta^2}\phi^*_0 \dd\theta\right],\\
        & \leq \sigma \|\phi_\sigma\|_{\dLVecComp}\left\|\frac{\dd^2}{\dd\theta^2}\phi^*_0\right\|_{\dLVecComp},\\
        & \leq \frac{\sigma}{\Real (\mu)} \|g\|_{\dLVecComp}\left\|\frac{\dd^2}{\dd\theta^2}\mathsf{R}^0(\mu)g\right\|_{\dLVecComp},
    \end{align*}
    where we used the estimate~\eqref{est:EllipResolLCompL2L2} on $\mathsf{R}^\sigma(\mu)$. Rearranging the last inequality we obtain the desired estimate.
\end{proof}


We now use a perturbation argument to prove Theorem~\ref{thm:EigenValSystem}.


\begin{proof}[Proof of Theorem~\ref{thm:EigenValSystem}]

We first note that from Lemma~\ref{lem:InversibilityEllipticOperator} and adapting the computations of the proof, we obtain that for any $\cchi, \llambda, \ttau, \sigma > 0$, there exists a real positive $\mu$ sufficiently large, such that the operator,
\begin{equation*}
    \left(\mu - \LinOpThet\right),
\end{equation*}

admits an inverse. The inverse is a compact operator since its image is in $\dHVecComp^1$. This implies from \cite[Theorem 6.29 in Chapter III]{kato2013perturbation} that the spectrum of $\LinOpThet$ consists of isolated eigenvalues with finite multiplicities. Thus we are left to prove that the Riesz projector of $\LinOpThet$ associated with some curve lying in the positive half complex plane is well-defined and has non-zero image given $\sigma$ sufficiently small, concluding by \cite[Theorem 6.17 Chapter III]{kato2013perturbation}.

We proceed by defining a positively oriented closed simple curve $\Gamma$, such that
\begin{equation*}
    \Gamma \subset \dC \cap \{\zeta \in \dC | \Real \zeta > 0\} \cap \rho\left(\mathsf{L}^E_{0})\right), 
\end{equation*} 
where $\rho\left(\mathsf{L}^E_{0}\right)$ is the resolvent set of the operator $\mathsf{L}^E_{0}$, and such that the inside of the curve only contains $\muinvi$ and no other elements of the spectrum of the inviscid operator $\mathsf{L}^E_{0}$.

For any $\mu \in \Gamma$, we note that the operators can be written as
\begin{align*}
        (\mu - \mathsf{L}^E_0) &=(\mu - \VOpLinVec)\left(\Id- \cchi \mathsf{R}^0(\mu)\AmatLinVec\overline{[\cdot]}\right),\\
        (\mu - \LinOpThet) &= \left(\mu - \left(\sigma\frac{\dd^2}{\dd\theta^2} +\VOpLinVec\right)\right)\left(\Id- \cchi \mathsf{R}^\sigma(\mu)\AmatLinVec\overline{[\cdot]}\right).
\end{align*}

We recall the following result from functional analysis.

Let $T$ be a bounded operator, and suppose that it admits a bounded inverse $T^{-1}$,
then for any bounded operator $S$, such that, $\|S-T\| < \|T^{-1}\|^{-1}$, then $S$ is invertible, and $S^{-1}$ is bounded with 
\begin{equation*}
    \|S^{-1}\| \leq \frac{\|T^{-1}\|}{1-\|T^{-1}\|\|S-T\|}, \ \  \|S^{-1} - T^{-1}\| \leq \frac{\|T^{-1}\|^2\|S-T\|}{1-\|T^{-1}\|\|S-T\|},
\end{equation*}
where $\|\cdot\|$ is the operator norm.

Thus, since $(\mu - \VOpLinVec)$ is invertible for any $\Real \mu > 0$, and since $\mu \in \Gamma$, then 
\begin{equation*}
T^0(\mu) = \left(\Id- \cchi \mathsf{R}^0(\mu)\AmatLinVec\overline{[\cdot]}\right),    
\end{equation*}
is invertible, and we denote its inverse by $R\left(\mu, \mathsf{L}^E_{0}\right)$. 

Similarly, define $T^\sigma(\mu)$ as $T^\sigma(\mu) \defeq \left(\Id- \cchi \mathsf{R}^\sigma(\mu)\AmatLinVec\overline{[\cdot]}\right)$. Then
from the results of Lemma~\ref{lem:InversibilityEllipticOperator} and the fact that $\AmatLinVec \overline{[\cdot]}$ is of finite rank, we obtain that,

\begin{equation*}
    \sup_{\mu \in \Gamma} \| T^\sigma(\mu) - T^0(\mu)\| \leq \sqrt{\sigma}\cchi\sup_{\mu \in \Gamma} \frac{1}{\Real(\mu)}\sqrt{2\|\mathsf{B}_{\ttau}^{:,1}\|_{\dLVecComp}\left\|\frac{\dd^2}{\dd\theta^2}R^0(\mu)\mathsf{B}_{\ttau}^{:,1}\right\|_{\dLVecComp}},
\end{equation*}
where $\mathsf{B}_{\ttau}^{:,1}$ is the first column of the matrix $\AmatLinVec$. This implies, by the functional analysis result for inverses above, that $T^\sigma(\mu)$ is invertible for $\sigma$ small enough. Also, since $\left(\mu - \left(\sigma\frac{\dd^2}{\dd\theta^2} +\VOpLinVec\right)\right)$ is invertible, it follows that for some $\sigma^{\text{def}}>0$, sufficiently small, $\left(\mu - \LinOpThet\right)$ is invertible for any $\mu \in \Gamma, 0\leq\sigma<\sigma^{\text{def}}$. We denote by $R\left(\mu, \LinOpThet\right)$ this inverse.


This allows us to define the associated Riesz projector for any $0 \leq \sigma < \sigma^{\text{def}}$,
\begin{equation*}
    P_\Gamma(\LinOpThet) \defeq -\frac{1}{2\pi i}\oint_\Gamma R\left(\mu, \LinOpThet\right) \dd\mu.
\end{equation*}

The pointwise convergence of the projector to the non-trivial projector of $\mathsf{L}^E_0$ as $\sigma$ goes to zero gives that the projector is non-trivial given $\sigma$ sufficiently small. 




We now show that any eigenvalue of $\LinOpThet$ with a positive real part must be real and of multiplicity two.

Let $\ABLinVec$ be an eigenfunction of $\LinOpThet$ associated with an eigenvalue $\mu$ such that $\Real(\mu) >0$. Recalling the definition of $\LinOpThet$ and the elliptic theory result stated in Lemma~\ref{lem:InversibilityEllipticOperator}, we must have that, $\overline{\ABLinVec} \neq 0$, otherwise that would imply that $\ABLinVec=0$.
We introduce $\ABLinVec^\perp$, the rotation of $\ABLinVec$, using the following notations,
\begin{equation*}
    \ABLinVec \defeq \begin{pmatrix}
        a\\ b
    \end{pmatrix},\overline{\ABLinVec} \defeq \begin{pmatrix}
         \bar{a}\\ \bar{b}
    \end{pmatrix},
    \ABLinVec^\perp \defeq \begin{pmatrix}
        0 & -1\\
        1 & 0
    \end{pmatrix} \ABLinVec, \overline{\ABLinVec}^\perp = \begin{pmatrix}
         -\bar{b}\\ \bar{a}
    \end{pmatrix}.
\end{equation*}

We then check through explicit computations, that if $\ABLinVec$ is an eigenfunction so is $\ABLinVec^\perp$. Indeed, we have that

\begin{align*}
    \left(\mu - \left(\sigma\frac{\dd^2}{\dd\theta^2} +\VOpLinVec\right)\right)\ABLinVec^\perp & = \mu \begin{pmatrix}
         -b\\ a
    \end{pmatrix} - \sigma\frac{\dd^2}{\dd\theta^2}  \begin{pmatrix}
         -b\\ a
    \end{pmatrix} - \begin{pmatrix}
         -\llambda \cos_1 a\\ -\llambda \cos_1 b
    \end{pmatrix},\\
    & = \cchi \begin{pmatrix}
        \cos_1 \bar{a} + \ttau \cos_2 \bar{b}\\
        \cos_1 \bar{b} - \ttau \cos_2 \bar{a}\\
    \end{pmatrix},\\
    & = \cchi \AmatLinVec \overline{\ABLinVec}^\perp.
\end{align*}

Using  the invertibility of operator $\left(\mu - \left(\sigma\frac{\dd^2}{\dd\theta^2} +\VOpLinVec\right)\right)$ from Lemma~\ref{lem:InversibilityEllipticOperator}, since $ \Real(\mu) >0$, and integrating in $\theta$ on both sides, we get
\begin{equation}
    \label{eq:LinOpisRealIntegral}
    \begin{split}
        \overline{\ABLinVec} &= \cchi \overline{\left(\mu - \left(\sigma\frac{\dd^2}{\dd\theta^2} +\VOpLinVec\right)\right)^{-1}\AmatLinVec}\overline{\ABLinVec},\\
        \overline{\ABLinVec}^\perp &= \cchi \overline{\left(\mu - \left(\sigma\frac{\dd^2}{\dd\theta^2} +\VOpLinVec\right)\right)^{-1}\AmatLinVec}\overline{\ABLinVec}^\perp.
    \end{split}
\end{equation}
Since $\overline{\ABLinVec} \neq 0$ and hence $\{\overline{\ABLinVec},\overline{\ABLinVec}^\perp\}$ spans $\dC^2$,\eqref{eq:LinOpisRealIntegral} implies,

\begin{equation*}
    \cchi \overline{\left(\mu - \left(\sigma\frac{\dd^2}{\dd\theta^2} +\VOpLinVec\right)\right)^{-1}\AmatLinVec} = \Id_2.
\end{equation*}
We conclude that the eigenspace associated with $\mu$ is of multiplicity $2$, and is given by the image of the map $\dC^2\to \dHVecComp^2$ defined as,

\begin{equation*}
    w \mapsto \left(\mu - \left(\sigma\frac{\dd^2}{\dd\theta^2} +\VOpLinVec\right)\right)^{-1}\AmatLinVec w.
\end{equation*}

\end{proof}

%% file: LinearInstabilityParabolic.tex
In this section we give a lower bound on the dimension of the unstable manifold of the linearized equation for the parabolic-parabolic system \eqref{sys:FormicidaeParabolic}.

\begin{thm}
    \label{thm:ExistenceEigenFunLinParabolic}
    Suppose that, there exists an integer wavenumber $k\geq 1$, such that,
    \begin{equation}
        \chi(2\pi k \tau + 1) > \lambda (\gamma + 4\pi \sigma_c k^2).
    \end{equation}
    Then there exists $\sigma_\theta^*>0$ such that for $\sigma_\theta \in [0,\sigma_\theta^*)$ there exists $\sigma_x^*(\sigma_\theta)>0$ such that for $\sigma_x \in [0,\sigma_x^*)$ there exist $k$ unstable eigenvalues of $\LinOpPara$,
    \begin{equation*}
        \mu^1,\dots,\mu^k \in \{\mu \in \dC | \Real\mu >0\} \cap \Sigma(\LinOpPara),
    \end{equation*}
    where $\Sigma(\LinOpPara)$ is the spectrum of $\LinOpPara$, and such that $X$ is the real invariant subspace associated with the $\mu^i$'s and their complex conjugates, and has the dimensional lower bound,
    \begin{equation*}
        \dim X \geq 4k.
    \end{equation*}
    Furthermore, we can choose $4k$-dimensional orthogonal sub-basis of real functions in $X$, such each of the functions are constant in $x_1$ or $x_2$.
\end{thm}
Similarly as in the elliptic case, we first introduce a family of ansatzes. Then we show that the parabolic linear operator has positive real eigenvalues for the case $\sigma_\xx>0,\sigma_\theta=0$. This result is stated in Lemma~\ref{lem:lininstabparab}. Then we can use a same Riesz projector argument to conclude with Theorem~\ref{thm:ExistenceEigenFunLinParabolic}, and the rotation argument for the multiplicity. 

For an integer wavenumber $k\geq 1$, we define the ansatz $(f^k, c^k)^\top$ as

\begin{equation}
    \label{eq:LinFamilyPara}
    \begin{cases}
        f^k(\theta, x_1, x_2) = a(\theta) \cos(2\pi k x_1) +  b(\theta) \sin(2\pi k x_1),\\
        c^k(\theta, x_1, x_2) = \alpha \cos(2\pi k x_1) +  \beta \sin(2\pi k x_1),
    \end{cases}
\end{equation}
where $a,b$ are some functions defined on $\mathbb{T}_{2\pi}$, and $\alpha, \beta$ are real numbers.

Applying $\LinOpPara$ to $(f^k,c^k)^\top$, and writing $\xonepik$ for $2\pi k x_1$, we obtain component-wise

\begin{align*}
    \left(\LinOpPara \begin{pmatrix}
        f^k\\
        c^k
    \end{pmatrix} \right)_1 = &  - 4\pi^2 k^2 \sigma_x (a(\theta) \cos(\xonepik) + b(\theta) \sin(\xonepik)) + \sigma_\theta (a''(\theta) \cos(\xonepik) + b''(\theta) \sin(\xonepik))\\
    & - 2\pi k \lambda \cos(\theta)(b(\theta) \cos(\xonepik) - a(\theta)\sin(\xonepik))\\
    &-2\pi k \chi(\cos(\theta)(\alpha \sin(\xonepik) - \beta \cos(\xonepik)) + 2\pi k \tau \cos(2\theta)(\alpha \cos(\xonepik) + \beta\sin(\xonepik))),\\
    \left(\LinOpPara \begin{pmatrix}
        f^k\\
        c^k
    \end{pmatrix} \right)_2 =& -(4\pi^2k^2\sigma_c + \gamma)(\alpha \cos(\xonepik) + \beta \sin(z)) + (\Bar{a}\cos(\xonepik) + \Bar{b} \sin(\xonepik))
\end{align*}

We denote $\ABLinVecPara$ for the vector $(a,b,\alpha, \beta)^\top \in \dLVecComp \times \dR^2$, and we use the following notations
\begin{equation*}
    (f^k,c^k)^\top=\ABLinVecPara\cdot(\cos\xonepik,\sin\xonepik)^\top,
\end{equation*}
and 
\begin{equation*}
    \Bar{\ABLinVecPara} \defeq (\Bar{a}, \Bar{b}, \alpha, \beta)^\top, \ \ \  \frac{\dd^2}{\dd\theta^2}\ABLinVecPara \defeq ( a'', b'', 0, 0)^\top.
\end{equation*}
We first study the eigenproblem for $\sigma_\theta=0$, such that
\begin{equation}\label{eq:EigenproblemParaAnsatz}
    \LinOpThetPara\ABLinVecPara=\mu\ABLinVecPara, \ \text{where} \ \LinOpThetPara\ABLinVecPara\defeq\SmatLinVecPara \ABLinVecPara + \VOpLinVecPara \ABLinVecPara + \AmatLinVecPara \Bar{\ABLinVecPara},
\end{equation}
with the notation

\newcommand{\zeroLR}{0_{\dLVecComp\to\dR^2}}
\newcommand{\zeroRL}{0_{\dR^2\to \dLVecComp}}
\newcommand{\zeroMat}{0}

\begin{equation*}
    \SmatLinVecPara = \begin{bmatrix}
                        -\ssigmax \Id & \zeroRL\\
                         \zeroLR &  \zeroMat_2
                        \end{bmatrix}, \VOpLinVecPara = \begin{bmatrix}
                        \VOpLinVec & \zeroRL\\
                         \zeroLR &  \zeroMat_2
                        \end{bmatrix}, \AmatLinVecPara = \begin{bmatrix}
                         \zeroRL & \cchi \AmatLinVec\\
                         \Id_2 & -\vvarphi \Id_2
                        \end{bmatrix},
\end{equation*}

where,
\begin{equation*}
    \zeroLR :  \dLVecComp \ni u \mapsto 0\in \dR^2, \zeroRL :  \dR^2 \ni w \mapsto \mathbf{0}\in \dLVecComp,
\end{equation*}
$\Id$ is the identity in $\dLVecComp$, $\Id_2$ is the identity matrix in $\mathcal{M}_{2,2}(\dR)$, and $\VOpLinVec$ and $\AmatLinVec$ are defined in \eqref{def:LinearisedEllipticVandBDefinition}.

In this section, for the parabolic-parabolic system, we use the following notations
\begin{equation*}
    \ssigmax = 4\pi^2 k^2\sigma_\xx, \sigma = \sigma_\theta, \llambda = 2\pi k \lambda, \cchi = 2\pi k \chi, \ttau = 2\pi k \tau, \vvarphi = 4\pi^2 \sigma_c + \gamma.
\end{equation*}







\begin{lem}[$\ssigmax>0,\sigma=0$]\label{lem:lininstabparab}
    Suppose that for $\cchi>0,\ttau>0$ and $\llambda>0$ the following condition holds
    \begin{equation}\label{ineq:parainstab}
       2\pi\cchi(\ttau + 1) > \vvarphi\llambda.
    \end{equation}
    Then, there exists $\sigma_x^*>0$ such that for any $\sigma_x\in[0,\sigma_x^*)$, there exists a real strictly positive isolated eigenvalue $\mu$ of multiplicity 2 for the operator $\LinOpThetPara$.
\end{lem}

\begin{proof}
    Similar as in the elliptic case we first rewrite the eigenproblem. For $\mu \in \dC$, such that $\Real(\mu)>0$, the operator,
    \begin{equation*}
        \left(\mu - \SmatLinVecPara - \VOpLinVecPara \right),
    \end{equation*}
    has inverse
    \begin{equation*}
        \left(\mu - \SmatLinVecPara - \VOpLinVecPara \right)^{-1} = \begin{bmatrix}
            \left((\mu + \ssigmax )-\VOpLinVec\right)^{-1} & \zeroRL\\
            \zeroLR & \mu^{-1}\Id_2
        \end{bmatrix}
    \end{equation*}
    Therefore, the eigenproblem \eqref{eq:EigenproblemParaAnsatz}, can be recasted as
    \begin{equation*}
        \ABLinVecPara =  \left(\mu - \SmatLinVecPara - \VOpLinVecPara \right)^{-1}\AmatLinVecPara \Bar{\ABLinVecPara}.
    \end{equation*}
    Integrating in $\theta$ gives
    \begin{equation}
        \label{eq:IntegratedProblemParabolicInviscid}
        \overline{\left(\mu - \SmatLinVecPara - \VOpLinVecPara \right)^{-1}\AmatLinVecPara} \Bar{\ABLinVecPara}=\Bar{\ABLinVecPara}.
    \end{equation}
    By explicit computations we obtain,
    \begin{equation*}
        \overline{\left(\mu - \SmatLinVecPara - \VOpLinVecPara \right)^{-1}\AmatLinVecPara} = \begin{bmatrix}
            \zeroRL & \cchi\overline{((\mu + \ssigmax) - \VOpLinVec)^{-1}}\\
            \mu^{-1}\Id_2 & -\vvarphi\mu^{-1}\Id_2
        \end{bmatrix}.
    \end{equation*}

    Hence, the eigenproblem is equivalent to finding positive real $\mu\in\dC$ and $(\Bar{a},\Bar{b}, \alpha, \beta)^\top\in \dR^4$, such that,

    \begin{align*}
            \begin{pmatrix}
                \Bar{a} \\
                \Bar{b}
            \end{pmatrix} & = \cchi\overline{\left((\mu + \ssigmax) - \VOpLinVec)^{-1}\right)} \begin{pmatrix}
                \alpha \\
                \beta
            \end{pmatrix}, \\
            \begin{pmatrix}
                \alpha \\
                \beta
            \end{pmatrix} & = \mu^{-1}\begin{pmatrix}
                \Bar{a} \\
                \Bar{b}
            \end{pmatrix} - \vvarphi\mu^{-1} \begin{pmatrix}
                \alpha \\
                \beta
            \end{pmatrix}.
    \end{align*}
    We can rewrite this as, after some algebra,
    \begin{align*}
            \begin{pmatrix}
                \Bar{a} \\
                \Bar{b}
            \end{pmatrix} & = \frac{\cchi}{\mu+\vvarphi}\overline{\left((\mu + \ssigmax) - \VOpLinVec)^{-1}\right)} \begin{pmatrix}
                \Bar{a} \\
                \Bar{b}
            \end{pmatrix}, \\
            \begin{pmatrix}
                \alpha \\
                \beta
            \end{pmatrix} & = \frac{1}{\mu+\vvarphi}\begin{pmatrix}
                \Bar{a} \\
                \Bar{b}
            \end{pmatrix}.
    \end{align*}
    Using the computations from the proof of Lemma~\ref{lem:InviscidEigenVal}, we obtain for the eigenproblem the equation
    \begin{equation*}
        \frac{\cchi}{\mu+\vvarphi}\overline{\left((\mu + \ssigmax) - \VOpLinVec)^{-1}\right)} = \frac{\cchi}{\mu+\vvarphi}\begin{pmatrix}
            \mathcal{I}(\ttau, \llambda, \mu + \ssigmax) & 0\\
            0 & \mathcal{I}(\ttau, \llambda, \mu + \ssigmax)
        \end{pmatrix},
    \end{equation*}
    with $\mathcal{I}$ as defined in \eqref{eq:InviscidEllipIntegraldef}.
    
    Finally, since $\cchi \mathcal{I}(\ttau, \llambda, \mu + \ssigmax)/(\mu + \vvarphi)$ is a strictly decreasing, continuous, convex function with respect to $\mu$ on $[0,+\infty)$ and its limit is zero for $\mu\to+\infty$, the existence for the eigenproblem \eqref{eq:EigenproblemParaAnsatz}, by the intermediate value theorem is equivalent to $\cchi \mathcal{I}(\ttau, \llambda, \ssigmax)/\vvarphi>1$.
    
    This condition can be written as,


    \begin{equation}
        2\pi \cchi\left(\ttau + \frac{1-2\ttau (\ssigmax/\llambda)}{\sqrt{(\ssigmax/\llambda)^2 + 1}+\ssigmax/\llambda}\right) > \vvarphi\llambda\sqrt{(\ssigmax/\llambda)^2 + 1}.
    \end{equation}
    The condition is satisfied given \eqref{ineq:parainstab} and $\ssigmax$ is sufficiently small.

    Lastly, since $\mathcal{I}(\ttau, \llambda, \mu+\ssigmax)$ is holomorphic for $\mu\in\dC\cap\{\Real(\mu)>0\}$, under the instability condition, the fully inviscid eigenproblem admits an isolated real eigenvalue $\mu$ in the right half plane. The two-dimensional eigenspace for this eigenvalue is explicitly given by the image of the map $\dR^4\to \dLVecComp\times\dR^2$ given by
    \begin{equation}
        \Bar{\ABLinVecPara}  \mapsto \left(\mu - \SmatLinVecPara - \VOpLinVecPara \right)^{-1}\AmatLinVecPara \Bar{\ABLinVecPara}.
    \end{equation}
\end{proof}

The proof of the existence for the eigenproblem for the case $\ssigmax>0,\sigma>0$ follows by the same stability argument as used for Theorem~\ref{thm:EigenValSystem} and making use of the arguments as in Lemma~\ref{lem:InversibilityEllipticOperator}. This gives a proof for Theorem~\ref{thm:ExistenceEigenFunLinParabolic}. Also, similarly, the operator $\LinOpPara$ for $\ssigmax>0,\sigma > 0 $ is of compact resolvent, and thus the spectrum consists of isolated eigenvalues. 

%% file: NonLinearInstability.tex
\subsection{Linear Instability implies NonLinear Instability}
In this section we show that the existence of a real strictly positive eigenvalue of the linearized operator around the homogeneous steady state implies a nonlinear instability result around the homogeneous state. In other words, we show a lower bound for the dimension of the unstable manifold of the steady state, as precluded in Section~\ref{sec:mainresults}.

We first show an adapted instability theorem for sectorial operators based on \cite[Theorem 5.1.3]{henry2006geometric}. We then apply this theorem to the parabolic-parabolic system \eqref{sys:FormicidaeParabolic}.

We start by recalling the following definitions and results from semigroup theory of sectorial operators. 

\begin{definition}
    We call a linear operator $A$ in a Banach space a \textit{sectorial operator} if it is a closed densely defined operator such that, for some $\omega \in (0, \pi/2)$, some $M \geq 1$ and a real number $a\in\mathbb{R}$, the sector,
    \begin{equation*}
        S_{a,\omega} = \left\{ \mu \ \ \big| \ \ \omega \leq |\arg (\mu - a)| \leq \pi, \mu \neq a\right\},
    \end{equation*}
    is in the resolvent set of $A$ and
    \begin{equation*}
        \|(\mu - A)^{-1}\| \leq M/|\mu -a| \text{ for all } \mu \in S_{a,\omega}.
    \end{equation*}
\end{definition}

Sectorial operators generate analytical semi-groups and possess regularization properties allowing the control of stronger norms of solutions. We will denote by $(e^{-At})_{t\geq 0}$ such a semigroup.

\begin{definition}
    \label{def:betaBnormwrtA}
    Let $A$ and $B$ be sectorial operators densely defined on a Banach space $(X, \|\cdot\|)$. We say that $B$ is $\beta$-controlled by $A$ if,
    $D(B) = D(A)$, the spectrum of $B$ is in the right-half plane $\Real \Sigma(B) > 0$, and $(B-A)B^{-\beta}$ is a bounded operator for $0\leq \beta < 1$. 
    
    This gives a norm on $X^\beta \defeq D(B^\beta) \subset X$ defined as
    \begin{equation*}
        \left\| y \right\|_\beta \defeq \left\| B^\beta y \right\|.
    \end{equation*}
\end{definition}

\begin{remark}
We refer to \cite[Chapter 1]{henry2006geometric} for the constructions of $\beta$-fractional powers of sectorial operators. 

\end{remark}
We use the notation $\Real(\Sigma)\defeq\{\Real(\zeta)\in\Sigma:\zeta\in\Sigma\}$.
\begin{prop}[Theorem 1.5.2, 1.5.3,  and 1.5.4 in \cite{henry2006geometric}]
    \label{prop:SectorialOperatorEstim}
    Suppose that $A$ is a sectorial operator and $\Sigma_1$ a bounded subset of $\Sigma(A)$ such that $\Sigma_1$ and $\Sigma_2 \defeq \Sigma(A)\setminus\Sigma_1$ are closed in the extended plane $\dC \cup \{\infty\}$. Furthermore, assume that $\inf\Real (\Sigma_2)> \gamma$.
    
    Then, there exists a unique decomposition $X = X_1 \oplus X_2$, with associated projections $\Pi_1, \Pi_2$, such that each $X_i$ is invariant under $A$, and the restrictions $A_{i}$ of $A$ to the subspaces $X_i$ satisfy
    \begin{align*}
        A_1 : X_1 \to X_1, \text{ is bounded }, \Sigma(A_1) = \Sigma_1,\\
        D(A_2) = D(A)\cap X_2 \text{ and } \Sigma(A_2) = \Sigma_2.
    \end{align*}
    Furthermore, there exists $M>0$, such that the semigroup satisfies the following estimates, for $y \in X_2 \cap X^\beta$ and $t > 0$,
    \begin{align*}
        \left\| e^{-A_2 t}y\right\|_\beta &\leq M \left\|y\right\|t^{-\beta}e^{-\gamma t},\\
        \left\| e^{-A_2 t}y\right\|_\beta &\leq M \left\|y\right\|_\beta e^{-\gamma t}.
    \end{align*}
\end{prop}

We now state and prove a nonlinear instability theorem for sectorial operators with a $\beta$-controllable nonlinearity. A close result can be found in \cite[Theorem 5.1.3]{henry2006geometric}, proved using the same strategy, with the difference that we here control the dimension of the unstable manifold via the Hilbertian structure of the Banach space we use.

\begin{prop}
    \label{prop:UnstableManifoldAbstract}
    Suppose that $(-L)$ is a sectorial operator with compact resolvent on a real Banach space $(X, \|\cdot\|)$.
    Suppose that the set
    \begin{equation*}
        \Sigma_1 = \Sigma(L) \cap \{\Real (\mu) > 0 \},
    \end{equation*}
    is non-empty, and let $N:D(B^\beta) \subset X \to X$ be a closely defined operator for some $\beta>0$ such that $N(0) = 0$ satisfying, for $y^1, y^2\in D(B^\beta)$,
    \begin{equation*}
        \|N(y^1) - N(y^2)\|\leq \varpi(r) \|y^1 -y^2\|_\beta, \text{  if } \|y^1\|_\beta, \|y^2\|_\beta \leq r,
    \end{equation*}
    for a continuity modulus $\varpi$, $\lim_{r \to 0^+} \varpi(r) = 0$.  

    Then, there exists $y$ in the space,
    \begin{equation}
        \label{def:SolutionSpaceIntegralInsta}
        L^2_{t,loc}(D(L))\cap C\left((-\infty, 0], H^\beta\right)\cap H^1_{t,loc}(H),
    \end{equation}
    which solves the following equation,
    \begin{equation}
    \label{eq:CauchyProblemSectorialOp}
        \frac{\dd}{\dd t} y = L y + N(y),
    \end{equation}
    which satisfies the decay estimate
    \begin{equation*}
        \|y(t)\|_{\beta} \leq C e^{2\alpha t}, \ \text{for all} \ t \leq 0,
    \end{equation*}
    and for some $\alpha>0,C > 0$.
     
    Furthermore, denoting $X_1$ is the invariant subspace associated with $\Sigma_1$, then for any $0< \varepsilon < 1$, there exists $C_\varepsilon >0$, such that for any $x \in X_1$, such that $\|x\|_{\beta} = C_\varepsilon$, there exists a solution $y$ as above, satisfying the estimate,
    \begin{equation}
        \label{est:UnstableAbstractEstimate}
        \|y(0) - x\|_{\beta} \leq \varepsilon\|x\|_{\beta}.
    \end{equation}
\end{prop}



\begin{remark}
    The proof is a modification of the proof of \cite[Theorem 5.1.3]{henry2006geometric} in order to obtain estimate~\eqref{est:UnstableAbstractEstimate}, and is given in the Appendix\ref{App:NonLinearInstab}.
\end{remark}

We will use the following result on the perturbation of an orthogonal basis in an Hilbert space.
\begin{lem}
    \label{lem:HilbertSpacePerturbationOfBasis}
    Let $\{v^i\}_{i=1}^d$ be a collection of $d$ orthogonal vectors in a Hilbert space $(H, \langle \cdot, \cdot \rangle)$, with constant norm $\|v^i\| = r$ for all $i$.
    If $\{w^i\}_{i=1}^d$ is a family such that, for all $i$, the following estimate holds,
    \begin{equation*}
        \|w^i-v^i\|< r/d.
    \end{equation*}
    Then $\{w^i\}_{i=1}^d$, forms a linearly independent family in $H$.
\end{lem}
\begin{proof}
    Let $a_i$ be such that $\sum_{j=1}^d a_i w^i = 0$, denoting $\delta^i = w^i - v^i$, and taking the scalar product with $v^i$, we obtain,
    \begin{align*}
        a_i \|v^i\|^2 &= - \sum_{j=1}^da_j \langle \delta^j, v^i\rangle,\\
        |a_i| \|v^i\|^2 &\leq \sum_{j=1}^d |a_j| \|\delta^j\| \|v^i\|,\\
        \sum_{j=1}^d |a_j| &\leq \left(\frac{d}{r} \max \|\delta^i\|\right)  \sum_{j=1}^d |a_j|,
    \end{align*}
    thus by hypothesis $\sum_{j=1}^d |a_i| = 0$.
\end{proof}

\subsection{Parabolic Nonlinear Instability}

\begin{thm}
    \label{thm:InstabilityParabolic}
    Suppose that, there exists an integer wave-number $k\geq 1$, such that,
    \begin{equation}
        \label{hyp:UnstableManifoldParabolic}
        \chi(2\pi k \tau + 1) > \lambda (\gamma + 4\pi \sigma_c k^2).
    \end{equation}
    Then there exists $\sigma_x^*>0$ such that for $\sigma_x \in (0,\sigma_x^*)$ there exists $\sigma_\theta^*(\sigma_x)>0$ such that for any $\sigma_\theta \in (0,\sigma_\theta^*)$, the following lower bound on the dimension of the unstable manifold of the homogeneous state $u_*=(1/2\pi,1/\gamma)^\top$ holds,
    \begin{equation*}
        \dim \mathcal{M}^u(u_*) \geq 4k.
    \end{equation*}
\end{thm}

\begin{proof}
    Since condition \eqref{hyp:UnstableManifoldParabolic} holds, Theorem~\ref{thm:ExistenceEigenFunLinParabolic} ensures the existence of $k$ eigenvalues in the positive half-plane, each associated with a four dimensional orthogonal basis, such that two of these vectors are constant in $x_1$ and two are constant in $x_2$. This enables us to apply individually Theorem~\ref{prop:UnstableManifoldAbstract} to the $4k$-vectors, imposing $\varepsilon = \frac{1}{4k + 1}$ in the estimate~\eqref{est:UnstableAbstractEstimate}, so that one concludes the dimensional lower bound on the nonlinear unstable manifold using the Lemma~\ref{lem:HilbertSpacePerturbationOfBasis}.
    
    For these functions with domain $\mathbb{T}_{2\pi}\times\mathbb{T}_1$ we use the following functional setup.
    \newcommand{\HspaceParabolic}{\mathcal{H}}
    \newcommand{\LinOpParabolic}{\mathcal{L}^P}
    For $m\geq 0$, we define the following Hilbert spaces $\HspaceParabolic^m$ as,
    \begin{equation*}
        \left\{ (h,c) \in H^{m}(\Torus_{2\pi}\times \Torus_1) \times H^{m+2}(\Torus_1)\Big| \int h\dd\theta \in H^{m+1}(\Torus_1),
        \int h \dd \theta \dd x = \int c \dd x = 0 \right\},
    \end{equation*}
    equipped with its canonical scalar product defined as follows.
    
    For $u^1 = (h^1,c^1) \in \HspaceParabolic^m $ and $u^2 = (h^2,c^2)\in \HspaceParabolic^m$,

    \begin{equation*}
        \langle u^1,u^2\rangle_{\HspaceParabolic^m} \defeq \langle h^1,h^2\rangle_{H^m_{x,\theta}} + \langle c^1,c^2\rangle_{H^{m+2}_{x}} + \left\langle \int h^1\dd\theta ,\int h^2 \dd\theta\right\rangle_{H^{m+1}_{x}}.
    \end{equation*}
    We then introduce the following closed densely defined linear operator $\mathcal{L}^P : \HspaceParabolic^2 \subset \HspaceParabolic^0 \to \HspaceParabolic^0$, 
    
    For $u = (h,c)\in \HspaceParabolic^2$,
    \begin{equation}
        \label{def:LinOpParabolic}
        \LinOpParabolic u \defeq \begin{pmatrix}
                                \sigma_x \partial_{xx} h + \sigma_\theta \partial_{\thth}h\\
                                \sigma_c \partial_{xx} c -\gamma c 
                            \end{pmatrix}
                            +\begin{pmatrix}
                                -\lambda \cos \partial_x h\\
                                \int h \dd\theta
                            \end{pmatrix}
                            +\begin{pmatrix}
                                -\chi f_*\partial_\theta (B_\tau[c]))\\
                                0
                            \end{pmatrix},
    \end{equation}
    with $B_\tau : H^2_x \to L^2_{x,\theta}$ defined as,
    \begin{equation*}
        B_\tau[c](x,\theta) \defeq -\sin(\theta) \partial_x c(x) -\tau\sin(\theta)\cos(\theta) \partial_{xx} c(x).
    \end{equation*}    
    
    The negative of the first bracketed term is sectorial, since it only has non-negative eigenvalues. The third term is a linear bounded operator. We can apply \cite[Theorem 1.3.2]{henry2006geometric} to control the middle term through Sobolev embeddings and conclude that $-\LinOpParabolic$ is sectorial.
    
    We denote by $A : \HspaceParabolic^2 \subset \HspaceParabolic^0 \to \HspaceParabolic^0 $, the operator,
    \begin{equation*}
        A u \defeq \begin{pmatrix}-\sigma_x \partial_{xx} h - \sigma_\theta \partial_{\thth}h\\
                                -\sigma_c \partial_{xx} c
                            \end{pmatrix}.
    \end{equation*}
    We first note from Fourier analysis that $\Sigma(A) < 0$ and we have the map $A^{-\frac{1}{2}} : \HspaceParabolic^{m} \to \HspaceParabolic^{m+1}$, so that,
    \begin{equation*}
        (A-(-\LinOpParabolic))A^{-\frac{1}{2}} \text{ is a bounded linear operator from } \HspaceParabolic^0 \text{ into itself.}
    \end{equation*}
    Let  $\|\cdot \|_{\HspaceParabolic^{m}, \frac{1}{2}}$ be the norm associated with $(\frac{1}{2}, A)$ as defined in Definition~\ref{def:betaBnormwrtA},
    \begin{equation*}
        \left\|u\right\|_{\HspaceParabolic^{m}, \frac{1}{2}} \defeq \left\|A^{\frac{1}{2}} u\right\|_{\HspaceParabolic^{m}}.
    \end{equation*}
    We further note that this norm is equivalent to $ \|\cdot\|_{\HspaceParabolic^{m+1}}$.
    To apply Theorem~\ref{prop:UnstableManifoldAbstract}, we need to control the norm of the nonlinearity. 
    
    Let $N: \HspaceParabolic^1 \subset \HspaceParabolic^0 \to \HspaceParabolic^0$, defined as,
    \begin{equation*}
        N(u) \defeq \begin{pmatrix}
                    -\chi \partial_\theta (B_\tau[c] h)\\
                    0
                 \end{pmatrix}.
    \end{equation*}
    Since the first term is equal to zero when integrating with respect to the $\theta$-variable, the $\HspaceParabolic^0$-norm is given by,
    \begin{equation*}
        \|N(u^1) - N(u^2)\|_{\HspaceParabolic^0} = \left\|\chi \partial_\theta (B_\tau[c^1] h^1) -\chi \partial_\theta (B_\tau[c^2] h^2) \right\|_{L^2_{x,\theta}}.
    \end{equation*}
    We then use the Morrey embedding theorem in dimension one, $H^3_x \hookrightarrow C^{2,\frac{1}{2}}_x$, so that, from the definition of $B_\tau$, we obtain that there exists a constant $C > 1$, depending on $\tau$, such that
    \begin{equation*}
        \|\partial_\theta B[c]\|_{L^\infty_{x,\theta}}+\| B[c]\|_{L^\infty_{x,\theta}} \leq C \|c\|_{H^3_x}.
    \end{equation*}
    Thus, taking $u^1, u^2$, such that $\|u^i\|_{\HspaceParabolic^1}\leq r$, for any $r>0$, we obtain
    \begin{align*}
        \|N(u^1) - N(u^2)\|_{\HspaceParabolic^0} & \leq \chi \left(\|\partial_\theta (B[c^1 - c^2] h^1)\|_{L^2_{x,\theta}} + \|\partial_\theta (B[c^1] (h^1-h^2))\|_{L^2_{x,\theta}}\right),\\
        & \leq \chi (\|\partial_\theta B[c^1 - c^2]\|_{L^\infty_{x,\theta}} \|h^1\|_{L^2_{x,\theta}} + \|B[c^1 - c^2]\|_{L^\infty_{x,\theta}} \|\partial_\theta h^1\|_{L^2_{x,\theta}}\\
        & \hspace{2em}+ \|\partial_\theta B[c^2]\|_{L^\infty_{x,\theta}} \|h^1-h^2\|_{L^2_{x,\theta}} + \|B[c^2]\|_{L^\infty_{x,\theta}} \|\partial_\theta (h^1-h^2)\|_{L^2_{x,\theta}}),\\
        &\leq \varpi(r) \|u^1-u^2\|_{\HspaceParabolic^1},
    \end{align*}
    with $\varpi(r) \defeq 4C\chi r$.

    We can then apply Theorem~\ref{prop:UnstableManifoldAbstract} to each of the $(4k)$-orthogonal vectors $\{e^i\}$ from Theorem~\ref{thm:ExistenceEigenFunLinParabolic}, since they are constant in one of their spatial variables, and obtain unstable solutions $u^i\in C((-\infty,0],\HspaceParabolic^1)$ such that, for some $M > 0, \alpha > 0$, the following estimate holds,
    \begin{equation}
        \label{est:DecayEstimNonLinIstaParabolic}
        \|u^i(t)\|_{\HspaceParabolic^1} \leq Me^{\alpha t}, \text{ for all } t\leq 0,
    \end{equation}
    and,
    \begin{equation*}
        \|u^i(0) - e^i\|_{\HspaceParabolic^1} \leq \frac{ \|e^i\|_{\HspaceParabolic^1}}{4k+1},
    \end{equation*}
    the dimensional lower bound then follows from Lemma~\ref{lem:HilbertSpacePerturbationOfBasis} after extending again the functions in $H^1(\Torus_{2\pi}\times \Torus^2_1)\times H^3(\Torus^2_1)$ by making them constant in the other variable and using that the norm in $\HspaceParabolic$ controls the norm in the full space for this injection.   
    
   We now show that these unstable solutions for the perturbation equation can be modified into a semigroup solution of the system~\eqref{sys:FormicidaeParabolic}, such that the first component of $u^i$, $h^i$, gives a solution
    \begin{align*}
        f^i = \frac{1}{2\pi} + h^i.
    \end{align*}
    We directly obtain that the functions $f^i$ are of mass 1. We finally check that they are positive by proving a regularity result for these functions. We drop the $i$ dependency in the rest of the computations.
    
    Let $g : (0,+\infty)\times \Torus_1\times \Torus_{2\pi}\to \dR_+$, be the fundamental solution to the anisotropic heat equation,
    \begin{equation*}
        \partial_t g = \sigma_x \partial_{xx} g + \sigma_{\thth} \partial_{\thth} g.
    \end{equation*}
    We recall the following integrability estimates on the fundamental solution of the heat equation on $\mathbb{T}_1\times \mathbb{T}_{2\pi}$ \cite[Appendix]{bertucci2024curvature}. That is, for any $1\leq p \leq +\infty $ there exists a constant $C>1$, such that for any $t>0$, the following estimates hold,
    \begin{equation*}
        \|g_t\|_{L^p} \leq C\left(1+ t^{-\frac{p-1}{p}}\right), \|\partial_\xi g_t\|_{L^p} \leq C\left(1+ t^{-\left(\frac{p-1}{p}+\frac{1}{2}\right)}\right),
    \end{equation*}
    where $\partial_\xi$ represents $\partial_\theta$ or $\partial_x$.
    
    Then using the Duhamel formula, for any $T>0$ and  $t\in (-T,0]$, we have that $f$ is the solution of the integral equation,
    \begin{equation}
        \label{eq:DuhamelNonLinInsta}
        f_t = g_T \ast f_{t-T} - \int^t_{t-T}  g_{t-s} \ast F[f,c](s)\dd s, 
    \end{equation}
    where $\ast$ represent the convolution on the variables $\theta,x$, and with $F[f,c]$ defined as,
    \begin{equation*}
        F[f,c] = \partial_\theta(B_\tau[c]f) + \cos \partial_x f.
    \end{equation*}
    Now, the Morrey embedding in dimension 1 for $c$,
    \begin{equation*}
        H^3(\Torus_1) \hookrightarrow C^{2,\frac{1}{2}}(\Torus_1),
    \end{equation*}
    and the decay estimates of $(f,c)$ in $\HspaceParabolic^1$ ensure that $F[f,c] \in C((-\infty,0],L^2_{x,\theta})$, with estimate
    \begin{equation*}
        \|F[f,c](t)\|_{L^2_{x,\theta}} \leq C e^{2\alpha t}, \text{ for all } t \geq 0.
    \end{equation*}
    Taking the derivative of the equation \eqref{eq:DuhamelNonLinInsta}, and applying the Young convolution inequality in dimension one with the exponents
    \begin{equation*}
        1 + \frac{1}{q} = \frac{1}{p} + \frac{1}{2},
    \end{equation*}
    for some $1<p<2$ so that $2<q<+\infty$, we obtain the estimate,
    \begin{equation}
        \|\partial_\xi f_t\|_{L^q_{x,\theta}} \leq \|g_T\|_{L^p_{x,\theta}} \|\partial_\xi f_{t-T}\|_{L^{2}_{x,\theta}} + C \int^t_{t-T} \|\partial_\xi g_{t-s}\|_{L^p_{x,\theta}} \|F[f,c](s)\|_{L^2_{x,\theta}}\dd s, 
    \end{equation}
    using the estimate on the fundamental solution to the heat equation above and the decay estimates on $\|\partial f_{t-T}\|_{L^2}$ and $\|F[f,c]\|_{L^2}$, we obtain,
    \begin{align*}
        \|\partial_\xi f_t\|_{L^q_{x,\theta}} & \leq C e^{2\alpha t} \left(\left(1+ T^{-\frac{p-1}{p}}\right) + \int^t_{t-T} \left(1+ {(t-s)}^{-\left(\frac{p-1}{p}+\frac{1}{2}\right)}\right) \dd s\right), \\
        &\leq C e^{2\alpha t} \left(1 + T +  T^{-\frac{p-1}{p}} + \frac{2p}{2-p}T^{\frac{1}{2}-\frac{p-1}{p}} \right),
    \end{align*}
    for all $T>0$ and $t<0$. This implies the converges of $f$ to the steady state $f_*$ in Sobolev space $W^{1,q}_{x,\theta}$ for any $2< q <+\infty$. Hence, by using the Morrey embedding in dimension $2$,
    \begin{equation*}
        W^{1,q}(\Torus_1 \times \Torus_{2\pi}) \hookrightarrow C^{0,1-\frac{2}{q}}(\Torus_1\times \Torus_{2\pi}),    
    \end{equation*}
    we have a uniform convergence of $f$ to $f_*$ as $t$ goes to $-\infty$. This implies that there exists $t^*< 0$, such that, 
    \begin{equation*}
        f(t) \geq 0 \text{ for all } t \leq t^*,
    \end{equation*}
    Finally, using the preservation of the sign forward in time \cite[Lemma 4.3]{bertucci2024curvature}, we obtain that, all the $f^i$ are positive on $(-\infty,0]$.

    \begin{equation*}
        f^i(t) \geq 0 \text{ for all } t \leq 0, \text{ for all } 1\leq i \leq k,
    \end{equation*}
    
    This concludes that $(f^i,c^i)$ is a semigroup solution of the system~\eqref{sys:FormicidaeParabolic} in the sense stated in Proposition~\ref{thm:FormicidaeParabolicSemiGroup}.
    
    
\end{proof}

\subsection{Elliptic Nonlinear Instability}

\begin{thm}
    \label{thm:InstabilityElliptic}
    Suppose that, there exists $k\geq 1$, such that,
    \begin{equation}
        \label{hyp:UnstableManifoldElliptic}
        \chi(2\pi k \tau + 1) > \lambda (\gamma + 4\pi \sigma_c k^2).
    \end{equation}

    Then there exists $\sigma_\theta^*>0$ such that for any $\sigma_\theta \in (0,\sigma_\theta^*)$ there exists $\sigma_x^*>0$ such that for any   $\sigma_x \in (0,\sigma_x^*)$, the following lower bound on the dimension of the unstable manifold of the homogeneous state $f_*$ holds,
    \begin{equation*}
        \dim \mathcal{M}^u(f_*) \geq 4k.
    \end{equation*}
\end{thm}

The proof is rigorously the same as in the proof of Theorem~\ref{thm:InstabilityParabolic} by adjusting the Hilbert space structure to,

\begin{equation*}
    \mathcal{H}^k = \left\{ f\in H^{k}(\Torus_{2\pi}\times \Torus_1) \Bigg| \int f\dd\theta \in H^{k+1}(\Torus_1)
        ,\int f \dd x \dd\theta = 0 \right\}.
\end{equation*}

%% file: NonLinearStability.tex
\subsection[Global Asymptotic Stability of the homogeneous solution]{Global Asymptotic Stability of the homogeneous solution under smallness of the interaction}
We here prove the global asymptotic stability of the homogeneous solution if $\chi$ is sufficiently small. We make use of the higher regularity estimates we encountered in the proof for the existence of the global attractor.

\begin{thm}
    \label{thm:GlobalAsymptoticStability}
    Suppose that $\sigma_x, \sigma_\theta, \sigma_c, \gamma >0$,
    Then for any $\lambda,\tau \geq 0$, there exists $\chi^*>0$, such that for any $0\leq \chi < \chi^*$, there exists $\alpha>0$ such that, for any initial condition for the parabolic system \eqref{sys:FormicidaeParabolic}, there exists $t_0>0$ such that the following estimate holds,
    \begin{equation*}
        \left\|f(t)-\frac{1}{2\pi}\right\|_{L^2_{x,\theta}} +  \|\nabla_x c\|_{L^2_x}+ \|\nabla^2_x c\|_{L^2_x} \leq C e^{-\alpha t}, \text{ for all } t\geq t_0,
    \end{equation*}
    where $C$ depends on the initial condition.
\end{thm}

\begin{proof}
        In the following we take $\chi>0$ and we will impose the smallness conditions at the end of the proof. The case $\chi=0$ follows easily from the computations.
        
        Recall from Corollary~\ref{cor:H4CParabolic}, we know that there exist $p>1$ and $C$ depending on all the parameters but $\chi$, such that for any initial condition there exists $t_0>0$ such that $c\in C([t_0, +\infty), H^4_x)$, and such that the following estimate holds,
        \begin{equation*}
            \|c(t)\|_{H^4_x} \leq (1+\chi^p)C, \text{ for all } t\geq t_0.
        \end{equation*}
        Using Morrey's embedding theorem in dimension two, that is, $H^4_x \hookrightarrow C^{2,\alpha}$ for any $\alpha < 1$, implies that
        \begin{equation}
            \label{est:LinftynablaHessc}
            \|\nabla_x c(t)\|_{L^\infty_x} + \|\nabla^2_x c(t)\|_{L^\infty_x} \leq (1+\chi^p)C_1, \text{ for all } t\geq t_0,
        \end{equation}
        for some constant $C_1>0$ depending on all the parameters but $\chi$.
        We refer to Section~\ref{sec:ExistenceGlobalAttractor} for the derivation of the following estimates,
        \begin{align}
            \label{est:StabilityFokkerPlanck}
            \frac{\dd}{\dd t} \int  \left(f-1/2\pi\right)^2 &= -2\sigma_x \int |\nabla_x f|^2 -2\sigma_\theta \int (\partial_\theta f)^2 -  2\chi \int B[c]f\partial_\theta f,\\
            \label{est:StabilityChemotaticHess}
            \frac{\dd}{\dd t} \int |\nabla^2_x c |^2 &\leq - 2\gamma \int |\nabla^2_x c|^2 + \frac{1}{\sigma_c} \int |\nabla_x \rho|^2,\\
            \label{est:StabilityChemotaticNabla}
             \frac{\dd}{\dd t} \int |\nabla_x c|^2 &  \leq - 2\gamma \int |\nabla_x c|^2  + \frac{1}{\sigma_c}\int (\rho-1)^2.
        \end{align}
        We bound the nonlinear term as follows, for $t\geq t_0$,
        \begin{align}
             2 \int |B[c]f\partial_\theta f| & \leq  \left(\int f^2 B[c]^2 + \int |\partial_\theta f|^2\right),\nonumber\\
                                                 & \leq  \left(2\int B[c]^2\left[\left(f-\frac{1}{2\pi}\right)^2 + \frac{1}{4\pi^2} \right] + \int |\partial_\theta f|^2\right),\nonumber\\
                                                 & \leq  \left(2 (1+\chi^p)C_2 \int \left(f-\frac{1}{2\pi}\right)^2 + \frac{1}{2\pi^2}\int B[c]^2 + \int |\partial_\theta f|^2\right),\nonumber\\
                                                 &  \leq 2(1+\chi^p)C_2 \int \left(f-\frac{1}{2\pi}\right)^2 + \int |\partial_\theta f|^2 \nonumber\\
                                                 \label{est:StabilityNonlinearTerm}
                                                 & \hspace{4em} + \frac{1}{2\pi^2}C_\tau \left(\int |\nabla_x c|^2 + \int |\nabla^2_x c|^2 \right),
        \end{align}
        for some $C_2$ depending on all the parameters but $\chi$, and $C_\tau$ depending only on $\tau$.
        We introduce the following notation for the sake of conciseness,
        \begin{equation*}
            \int \Psi_{\chi,\tau, \gamma} = \left( \frac{1}{\chi}\int (f-1/2\pi)^2 + \frac{C_\tau}{\gamma}\int \left(|\nabla_x c|^2 + |\nabla^2_x c|^2\right) \right).
        \end{equation*}
        Adding the appropriate multiples of estimates \eqref{est:StabilityFokkerPlanck}, \eqref{est:StabilityChemotaticHess}, and \eqref{est:StabilityChemotaticNabla}, we can drop the last term of estimate \eqref{est:StabilityNonlinearTerm}, and obtain,
        
        \begin{align*}
            \frac{\dd}{\dd t} \int \Psi_{\chi,\tau, \gamma}  & \leq -2\frac{\sigma_x}{\chi} \int |\nabla_x f|^2 - 2\frac{\sigma_\theta}{\chi} \int (\partial_\theta f)^2 + 2(1+\chi^p)C_2 \int \left(f-\frac{1}{2\pi}\right)^2 + \int |\partial_\theta f|^2\\
            & \hspace{1em}-C_\tau\int \left(|\nabla_x c|^2 + |\nabla^2_x c|^2\right) + \frac{C_\tau}{\sigma_c \gamma}\left(\int |\nabla_x \rho|^2 + \int (\rho -1)^2 \right),\\
            & \leq -\left(2\frac{\sigma_x}{\chi} -  \frac{C_\tau}{\sigma_c \gamma}\right) \int |\nabla_x f|^2 - \left(2\frac{\sigma_\theta}{\chi} - 1\right) \int (\partial_\theta f)^2\\
            & \hspace{1em}-C_\tau\int \left(|\nabla_x c|^2 + |\nabla^2_x c|^2\right)\\
            & \hspace{1em} + \left(\frac{C_\tau}{\sigma_c \gamma} + 2(1+\chi^p)C_2\right) \int \left(f-1/2\pi\right)^2
        \end{align*}
        Now assuming that,
        \begin{equation}
            \label{hyp:FirstSmallnessCondition}
            \left(2\frac{\sigma_x}{\chi} -  \frac{C_\tau}{\sigma_c \gamma}\right) \bigwedge \left(2\frac{\sigma_\theta}{\chi} - 1\right) >0.
        \end{equation}
        And using Poincaré inequality, we obtain,
        \begin{equation*}
            \frac{\dd}{\dd t} \int \Psi_{\chi,\tau, \gamma} \leq -W(\chi)\frac{1}{\chi}\int (f-1/2\pi)^2 - C_\tau \int \left(|\nabla_x c|^2 + |\nabla^2_x c|^2\right),
        \end{equation*}
        where $W$ is defined as,
        \begin{equation}
            \label{hyp:SecondSmallnessCondition}
            W(\chi) =  \left(\frac{1}{C^2_{P}}\left(2\sigma_x -  \chi\frac{C_\tau}{\sigma_c \gamma}\right) \bigwedge \left(2\sigma_\theta - \chi\right) - \chi\left(\frac{C_\tau}{\sigma_c \gamma} + 2(1+\chi^p)C_2\right)\right),
        \end{equation}
        with $C_{P}$ the Poincaré constant. This ensures that there exists  $\chi^* >0$ sufficiently small so for any $0< \chi< \chi^*$, \eqref{hyp:FirstSmallnessCondition} is satisfied and such that $W(\chi) > 0$.
        \begin{equation*}
            \frac{\dd}{\dd t} \int \Psi_{\chi,\tau, \gamma} \leq - W(\chi) \wedge \gamma \int \Psi_{\chi,\tau, \gamma}.
        \end{equation*}

        We conclude using Gr\"onwall inequality, that if $\chi^*$ is sufficiently small, from any initial condition, there exists $t_0 > 0$ and $C_0 > 0$ depending on the initial condition such that the following holds,
        \begin{equation*}
            \int (f(t)-1/2\pi)^2 + \frac{\chi C_\tau}{\gamma}\int \left(|\nabla_x c(t)|^2 + |\nabla^2_x c(t)|^2\right) \leq C_0 e^{-(W(\chi) \wedge \gamma )t}, \text{ for all } t\geq t_0.
        \end{equation*}
\end{proof}

Similar computations lead to the same result in the elliptic case.
\begin{thm}
    \label{thm:GlobalAsymptoticStabilityElliptic}
    Suppose that $\sigma_x, \sigma_\theta, \sigma_c, \gamma >0$,
    Then for any $\lambda,\tau \geq 0$, there exists $\chi^*>0$, such that for any $0\leq \chi < \chi^*$, there exists $\alpha>0$ such that, for any initial condition for the parabolic system \eqref{sys:FormicidaeElliptic}, there exists $t_0>0$ such that the following estimate holds,
    \begin{equation*}
        \left\|f(t)-\frac{1}{2\pi}\right\|_{L^2_{x,\theta}} \leq C e^{-\alpha t}, \text{ for all } t\geq t_0,
    \end{equation*}
    where $C$ depends on the initial condition.
\end{thm}

%% file: Appendix.tex
\begingroup
\renewcommand{\thesubsection}{\Alph{subsection}}
\renewcommand{\thethm}{\thesubsection.\arabic{thm}}
\renewcommand{\theprop}{\thesubsection.\arabic{prop}}
\renewcommand{\theequation}{\thesubsection.\arabic{equation}}

\appendix
\section*{Appendix}
\addcontentsline{toc}{section}{Appendix}
\setcounter{equation}{0}

\subsection{Nonlinear Instability for Sectorial operators}
\label{App:NonLinearInstab}

For the sake of completeness, we here give the modification of the proof of \cite[Theorem 5.1.3]{henry2006geometric} in order to obtain estimate~\eqref{est:UnstableAbstractEstimate}.

\begin{proof}[Proof of Theorem~\ref{prop:UnstableManifoldAbstract}:]
    We first decompose $\Sigma$ into two invariant subspaces. We denote by $\Sigma_1$ and $\Sigma_2$,
    \begin{equation*}
        \Sigma_1 = \Sigma(L) \cap \{\Real (\mu) > 0\}, \Sigma_2 = \Sigma(L) / \Sigma_1,
    \end{equation*}
    and let $X_1,X_2$ be the associated invariant subspaces of $\Sigma_1$ and $\Sigma_2$, such that $X = X_1 \oplus X_2$, $L_1,L_2$ be the associated restrictions of $L$, and $\Pi_1$ and $\Pi_2$, the associated projection operators. We note that since $L$ is defined on the real Banach space $X$, and since we decomposed the spectrum along the imaginary axis, $X_1, X_2$ are real Banach spaces, and $\Pi_1,\Pi_2,L_1,L_2$ are defined on the real Banach space $X$.

     We note that $\Sigma_1$ is bounded since $(-L)$ is sectorial and consists of a finite number of isolated eigenvalues (since $L$ is of compact resolvent). This implies that there exists $\alpha >0$ such that,
    \begin{equation*}
        \inf\Real (\Sigma_1) > 3\alpha \text{ and }\sup\Real(\Sigma_2) < \alpha.
    \end{equation*}
    
    From Proposition~\ref{prop:SectorialOperatorEstim}, we obtain the following estimates,
    for $t > 0$, and $y \in X_2$,
    \begin{equation*}
        \left\| e^{L_2 t}y\right\|_\beta \leq M \left\|y\right\|t^{-\beta}e^{\alpha t},\hspace{2em} \left\| e^{L_2 t}y\right\|_\beta \leq M \left\|y\right\|_\beta e^{\alpha t}.
    \end{equation*}
    and for $y\in X_1$ and $t\leq 0$,
    \begin{equation*}
        \left\| e^{L_1 t}y\right\|_\beta \leq M \left\|y\right\|e^{3\alpha t}, \hspace{2em} \left\| e^{L_1 t}y\right\|_\beta \leq M \left\|y\right\|_\beta e^{3\alpha t},
    \end{equation*}
    since $\dim X_1< \infty$ because eigenfunctions of $L$ are in $D(L)$.

    For some $x\in X_1$, with a norm that we will fix later, consider the integral equation
    \begin{equation}
        \label{eq:InstabIntegral}
        y(t) = e^{L_1 t}x - \int^0_t e^{L_1(t-s)} \Pi_1 N(y(s)) \dd s + \int^t_{-\infty} e^{L_2(t-s)} \Pi_2 N(y(s)) \dd s.
    \end{equation}
    \newcommand{\Yspace}{\mathcal{Y}^\alpha_\beta}
    For this integral equations, we define the Banach space $(\Yspace , \|\cdot \|_{\mathcal{Y}^\alpha_\beta})$ as follows,
    \begin{align*}
        &\Yspace = \left\{ y\in C\left((-\infty,0],H^\beta\right),  \sup_{t\leq 0} \|y\|_{\beta}e^{-2\alpha t} <+\infty \right\},\\
        &\|\cdot \|_{\Yspace} = \sup_{t\leq 0} \|\cdot\|_{\beta}e^{-2\alpha t}.
    \end{align*}
    We want to find a fixed point so that we get a solution for the integral equation \eqref{eq:InstabIntegral}. One can check that the right-hand side of \eqref{eq:InstabIntegral} is a well-defined mapping from $\Yspace$ to itself. We write down below the contraction estimate. Let $z_1,z_2 \in \Yspace$, such that their $\Yspace$-norm is smaller than some $r>0$ to be fixed later and let $y_1,y_2\in \Yspace$ be the associated integral as defined in \eqref{eq:InstabIntegral}.
    
    For $t\leq 0$, we get
    \begin{align*}
        \|y_1(t) - y_2(t)\|_\beta & \leq \int^0_t \left\|e^{L_1(t-s)} \Pi_1 (N(z_1(s))- N(z_2(s)))\right\|_\beta\\
        &\hspace{3em}+ \int^t_{-\infty}\left\|e^{L_2(t-s)} \Pi_2 (N(z_1(s))-N(z_2(s)))\right\|_\beta \dd s\\
         & \leq M\int^0_t\|N(z_1(s)) - N(z_1(s))\|e^{3\alpha(t-s)} \dd s \\
         &\hspace{3em}+M\int^t_{-\infty}\left\|(N(z_1(s))-N(z_2(s)))\right\|(t-s)^{-\beta} e^{\alpha(t-s)}\dd s\\
         & \leq M \varpi(r) \int^0_t\|z_1(s) - z_1(s)\|_\beta e^{3\alpha(t-s)} \dd s \\
         &\hspace{3em}+M\varpi(r)\int^t_{-\infty}\left\|z_1(s)-z_2(s)\right\|_\beta(t-s)^{-\beta} e^{\alpha(t-s)}\dd s\\
         & \leq M \varpi(r) \|z_1 - z_1\|_{\Yspace} \int^0_t e^{2\alpha s}e^{3\alpha(t-s)} \dd s \\
         &\hspace{3em}+M\varpi(r)\|z_1 - z_1\|_{\Yspace}\int^t_{-\infty} e^{2\alpha s}(t-s)^{-\beta} e^{\alpha(t-s)}\dd s\\
         &\leq M \varpi(r)e^{2\alpha t}\left( \frac{1}{\alpha} + \alpha^{\beta-1} \Gamma(1-\beta)\right) \|z_1 - z_1\|_{\Yspace},
    \end{align*} 
    where $\Gamma$ is the gamma function defined as $\Gamma(z) = \int_0^\infty t^{z-1} e^{-t} \dd t$. We then choose $r>0$ so that,
    \begin{equation}
        \label{eq:EstimeInstabilityIntegral}
        M \varpi(r)\left( \frac{1}{\alpha} + \alpha^{\beta-1} \Gamma(1-\beta)\right) \leq \frac{\varepsilon}{2M} < \frac{1}{2}.
    \end{equation}

    And impose the normalization of the $x$ to satisfy,
    \begin{equation*}
        \|x\|_\beta = \frac{r}{2M},
    \end{equation*}
    which ensures that the solution satisfies the estimate,
    \begin{equation}
        \label{eq:YsolIntEstim}
        \|y(t)\|_{\Yspace} \leq 2M\|x\|_{\beta}.
    \end{equation}
    
    Then the second inequality in \eqref{eq:EstimeInstabilityIntegral}, ensures that \eqref{eq:InstabIntegral} defines a contraction. Thus, from the Banach contraction mapping theorem, a unique solution of equation~\eqref{eq:InstabIntegral} exists. 
    The fact that $(-L)$ is sectorial implies from \cite[Lemma 3.3.2]{henry2006geometric} that $y$ is a solution of the equation
    \begin{equation*}
        \frac{\dd}{\dd t}y^i = Ly^i + N(y^i).
    \end{equation*}
    Finally using the decay estimate \eqref{eq:YsolIntEstim}, we control the deviation of $y(0)$ from our unstable eigenvector $x$, using
    \begin{equation}
        y(0)=x+\int_{-\infty}^0 e^{L_2(t-s)}\Pi_2 N(y(s))\dd s,
    \end{equation}
    we conclude that,
    \begin{equation*}
        \|y(0) - x\|_\beta \leq M \varpi(r) \int_0^\infty s^{-\beta} e^{\alpha s} 2M \|x\|_\beta e^{-2\alpha s}\dd s \leq \varepsilon \|x\|_{\beta},
    \end{equation*}
    where we used the first inequality in \eqref{eq:EstimeInstabilityIntegral}. 
\end{proof}
\endgroup